\documentclass[11pt, reqno, a4paper]{article}
\usepackage{color,fancyhdr}
\usepackage[top=2.5cm, bottom=2.2cm, left=2cm, right=2cm]{geometry}
\usepackage{graphicx,amsmath,amssymb,cite,booktabs}
\usepackage[colorlinks=true,breaklinks=true,linkcolor=lightblue,citecolor=lightblue,urlcolor=lightblue]{hyperref}
\usepackage{stackengine}
\usepackage{algorithm}
\usepackage{comment}
\usepackage{mathrsfs}
\usepackage[mathscr]{euscript}
\usepackage{overpic}

\usepackage{tikz}
\usetikzlibrary{decorations.pathmorphing,patterns}
 
\DeclareSymbolFontAlphabet{\amsmathbb}{AMSb}

\definecolor{lightblue}{rgb}{0.22,0.45,0.70}
\definecolor{mygray}{rgb}{0.7,0.7,0.7}

\usepackage{upgreek}

\newcommand\cero{\boldsymbol{0}}
\newcommand{\norm}[1]{\left\|#1\right\|}

\newcommand{\fracd}[2]{\displaystyle
	{\frac{{\displaystyle{#1}}}{{\displaystyle{#2}}}}}
\newcommand\vdiv{\mathrm{div}}

\newcommand\bcurl{\mathop{\mathbf{curl}}\nolimits}
\def\bdiv{\mathbf{div}}

\newcommand\tr{\mathop{\mathrm{tr}}\nolimits}

\newcommand{\beps}{\boldsymbol{\varepsilon}}

\newcommand{\bF}{\mathbf{F}}
\newcommand{\bG}{\mathbf{G}}
\newcommand{\bS}{\mathbf{S}}
\newcommand{\bbS}{\amsmathbb{S}}

\newcommand{\bA}{\mathbf{A}}
\newcommand{\bB}{\mathbf{B}}
\newcommand{\bC}{\mathbf{C}}
\newcommand{\bE}{\mathbf{E}}

\newcommand{\ba}{\mathbf{a}}
\newcommand{\bb}{\mathbf{b}}

\newcommand{\bu}{\mathbf{u}}

\newcommand{\bg}{\mathbf{g}}
\newcommand{\bq}{\mathbf{q}}
\newcommand{\bn}{\mathbf{n}}

\newcommand{\bv}{\mathbf{v}}
\newcommand{\bw}{\mathbf{w}}
\newcommand{\bx}{\mathbf{x}}

\newcommand{\bz}{\mathbf{z}}
\newcommand{\br}{\mathbf{r}}
\newcommand{\be}{\mathbf{e}}
\newcommand{\ff}{\mathbf{f}}

\renewcommand{\gg}{\mathbf{g}}

\newcommand{\brho}{\boldsymbol{\rho}}

\newcommand{\bnabla}{\boldsymbol{\nabla}}

\newcommand{\bPi}{\boldsymbol{\Pi}}

\newcommand{\sP}{\mathscr{P}}
\newcommand{\bsP}{\boldsymbol{\mathscr{P}}}

\newcommand\dt{{\,\mathrm{d}t\,}}

\newcommand\ds{{\,\mathrm{d}s\,}}

\newcommand\cA{\mathcal{A}}

\newcommand\cC{\mathcal{C}}

\newcommand\cG{\mathcal{G}}
\newcommand\cT{\mathcal{T}}

\newcommand\bK{\mathbf{K}}
\newcommand\bP{\mathbf{P}}
\newcommand\LL{\mathrm{L}}

\newcommand{\sF}{\mathscr{F}}
\newcommand{\sS}{\mathscr{S}}
\newcommand{\sR}{\mathscr{R}}

\newcommand\bV{\mathbf{V}}

\newcommand{\sD}{\mathscr{D}}
\newcommand{\sE}{\mathscr{E}}

\newcommand\R{\mathbb{R}}

\newcommand\PP{\mathrm{P}}
\newcommand\rQ{\mathrm{Q}}

\newcommand\bbH{\mathbb{H}}

\newcommand\bbP{\mathbb{P}}
\newcommand\bH{\mathbf{H}}
\newcommand\bL{\mathbf{L}}

\newcommand\rW{\mathrm{W}}
\newcommand\bbL{\mathbb{L}}
\newcommand\bbK{\mathbb{K}}

\newcommand\bbI{\mathbb{I}}
\newcommand\bZ{\mathbf{Z}}

\newcommand\HH{\mathrm{H}}

\newcommand\rd{\mathrm{d}}

\newcommand\bsigma{{\boldsymbol{\sigma}}}
\newcommand\btau{\boldsymbol{\tau}}

\newcommand\bomega{{\boldsymbol{\omega}}}
\newcommand\eeta{\boldsymbol{\eta}}
\newcommand{\bzeta}{\boldsymbol{\zeta}}

\newcommand\sk{{\mathrm{skew}}}

\def\va{a_{\rm v}}
\def\vb{b_{\rm v}}
\def\ca{a_{\rm c}}
\def\cb{b_{\rm c}}

\providecommand{\coef}[1]{\underline{#1}}      
\providecommand{\bMv}{\mathbf{M}_{\bv}}
\providecommand{\bMf}{\mathbf{M}_{\varphi}}
\providecommand{\bMt}{\mathbf{M}_{\theta}}
\providecommand{\bTc}{\mathbf{T}}
\providecommand{\blkS}{\mathsf{S}}
\providecommand{\blkH}{\mathsf{H}}
\providecommand{\blkB}{\mathsf{B}}
\providecommand{\blkC}{\mathsf{C}}
\providecommand{\blkM}{\mathsf{M}}
\providecommand{\yvec}{\underline{\mathbf{y}}}
\providecommand{\xivec}{\underline{\boldsymbol{\xi}}}

\DeclareMathOperator*{\esssup}{ess\,sup}

\newcommand{\diff}{\mathbb{D}}  

\def\wt{\widetilde}
\def\wh{\widehat}

\def\qin{{\quad\hbox{in}\quad}}
\def\qon{{\quad\hbox{on}\quad}}
\def\qan{{\quad\hbox{and}\quad}}

\numberwithin{equation}{section}
\numberwithin{figure}{section}
\numberwithin{table}{section}
\newtheorem{remark}{Remark}[section]
\newtheorem{assumption}{\sc Assumption}[section]
\newtheorem{theorem}{Theorem}[section]
\newtheorem{lemma}[theorem]{Lemma}

\newenvironment{proof}{\noindent{\it Proof.}}{\hfill$\square$}

\allowdisplaybreaks
\begin{document}

\title{Fully mixed finite element methods for the coupling \\ of viscoelasticity and reaction-diffusion models\thanks{\textbf{Funding:} This work has been partially supported by  the National Science Foundation through the grant DMS-2110781 and by the Australian Research Council through the {Future Fellowship} grant FT220100496.

\, $^{\sharp}$Corresponding author.}}
\author{
Alonso J. Bustos$^{\sharp,}$\thanks{Department of Mathematical Sciences and the Center for Mathematics and Artificial Intelligence (CMAI), George Mason University, Fairfax, VA 22030, USA.
E-mail: {\tt abustosb@gmu.edu}.}\,, \quad 
Jeonghun J. Lee\thanks{Department of Mathematics, Baylor University, Sid Richardson Science Building
1410 S.4th Street Waco, TX 76706, USA. E-mail: \texttt{jeonghun\_lee@baylor.edu}.}\,, 
\quad 
Ricardo Ruiz-Baier\thanks{School of Mathematics, Monash University, 9 Rainforest Walk, Melbourne 3800 VIC, Australia.
E-mail: \texttt{ricardo.ruizbaier@monash.edu}.}
}

\maketitle

\begin{abstract}
\noindent We propose and analyse a fully mixed finite element method for a two-way coupled mechanochemical model of calcium signalling in viscoelastic biological tissue. The mechanical response follows a Kelvin--Voigt law, with both the elastic and viscous stress components retained as unknowns together with the velocity and spin tensor, while the symmetry of the total stress is imposed weakly. The reaction-diffusion system governing the calcium concentration and the amount of inactivated inositol triphosphate is formulated in mixed form, with the diffusive flux introduced as an additional unknown. The coupling acts in both directions, with calcium entering the constitutive law through an active stress contribution and local dilation feeding back into the calcium equation through a source term proportional to the trace of the elastic stress. Global existence and uniqueness are established by reducing the coupled system to an integro-differential problem for the velocity and concentrations, from which the mixed variables are subsequently recovered. The semidiscrete scheme is formulated for general conforming finite element spaces satisfying suitable compatibility conditions and is shown to be well posed under a single inf-sup condition by eliminating the algebraic unknowns and reducing the differential-algebraic system to an ordinary differential equation. The \textit{a priori} error analysis relies on an interpolation operator that commutes with the divergence and preserves the weak symmetry pairing, leading to optimal error estimates. These results are specialised to three families of weakly symmetric finite elements, for which explicit convergence rates are obtained. For one of these families, we additionally establish a superconvergent local post-processing of the displacement. Numerical tests in two and three dimensions verify the predicted convergence rates and illustrate the method through viscoelastic benchmarks and mechanically driven calcium wavetrains.


\end{abstract}

\section{Introduction}

Biological tissues exhibit a coupled solid-fluid mechanical response. Solid components, such as collagen, deform elastically and reversibly under moderate stresses, whereas fluid phases, such as water and plasma, exhibit irreversible, time-dependent deformation associated with viscosity. This dual behaviour motivates the use of viscoelastic constitutive models at the tissue scale. In general, the governing equations couple momentum balance with mass transport through reaction-diffusion mechanisms arising, for instance, from water attraction by hyaluronic acid or cytoplasmic dissipation \cite{eliahoo2024viscoelasticity}. Other effects, including plasticity, nonlinear elasticity, and poroelasticity, may also be relevant but are not considered here. We are interested in the interaction between chemical species and viscoelastic materials at the mesoscale. More specifically, we consider the dynamics of calcium and inactivated inositol triphosphate (iIT) in epithelial cells, where a continuum description remains appropriate. The balance equations for linear and angular momentum are coupled to the dynamics of calcium and iIT through active stress contributions, while the chemical dynamics depend in turn on local volume changes through a dilation-dependent source term. The chemical component of the model originates from the single-pool description of calcium oscillations introduced by Atri et al.~\cite{atri93}. Models of this type were subsequently proposed in one dimension in~\cite{jones12,kaouri19} and extended to multiple dimensions in~\cite{kaouri22}, where a Herrmann formulation of the viscoelasticity equations was coupled with a primal formulation of the reaction-diffusion system. A related variant tailored to apical constriction was studied in~\cite{kaouri22frontiers}. These works focused primarily on modelling aspects and the qualitative behaviour of the solutions, together with a preliminary stability estimate, but did not provide a well-posedness or error analysis for the proposed discretisations.

The mathematical analysis and numerical approximation of this model pose several challenges arising from the interaction between its mechanical and chemical components. In particular, the coupling involves quantities that are not primary variables in standard displacement-based or primal formulations. The calcium concentration contributes to the mechanical response through an active stress, while the elastic stress enters directly into the source term of the calcium equation. This structure motivates a formulation in which stresses and diffusive fluxes are treated as primary unknowns. Accordingly, we adopt a fully mixed formulation for both components of the problem. The viscoelasticity equations are written in terms of the elastic and viscous stresses, the velocity, and the spin tensor, with the symmetry of the total stress imposed weakly. The reaction-diffusion equation governing the calcium concentration is likewise written in mixed form, with the diffusive flux as an additional unknown. Besides providing local mass conservation and direct approximation of the flux, as usual in mixed methods~\cite{boffi13,gatica14}, this formulation gives direct access to the elastic stress entering the mechanochemical coupling, without requiring any additional reconstruction or post-processing.

Mixed formulations with weakly imposed stress symmetry have been extensively studied for time-dependent elasticity and viscoelasticity. Rognes and Winther~\cite{rognes10} analysed mixed methods with weak symmetry for quasistatic linear viscoelasticity, while Arnold and Lee~\cite{arnoldlee14} treated elastodynamics and constructed an interpolation operator that simultaneously commutes with the divergence and preserves the weak symmetry pairing, a tool that will play a central role in our error analysis. Mixed formulations for the standard linear solid were studied in~\cite{lee17,gmm21}, while related hybrid, discontinuous Galerkin, virtual element, and stress-based formulations for time-dependent problems can be found in~\cite{marquez21,meddahi23,kumar25,garcia17}. These contributions, however, concern the viscoelasticity equations in isolation or coupled at most through prescribed data. For coupled mechanics-transport problems, fully mixed and mixed-primal formulations have proved effective when the stress itself, rather than the displacement, drives the transport process. Representative examples include the stress-assisted diffusion problems analysed in~\cite{ggr18,ggr19}, where elasticity is written in stress-rotation-displacement form and diffusion in primal or mixed form, and the twofold saddle-point formulation of Biot poroelasticity with stress-dependent diffusion introduced in~\cite{gomez23}. To the best of our knowledge, none of these works addresses a two-way coupling between a viscoelastic solid with weakly symmetric total stress and a reaction-diffusion system written in mixed form.

Two features distinguish the resulting problem from those considered in the references above and account for much of the technical analysis in this paper. First, the elastic and viscous stresses are retained as separate unknowns, since the elastic stress enters directly into the calcium equation, while $\bbH(\bdiv)$-conformity is required only for their sum. As a consequence, the resulting formulation does not fit directly into a standard saddle-point framework. Second, the coupling is   two-way and nonlinear, with the calcium concentration entering the constitutive law through a Hill function and the trace of the elastic stress entering the source term of the calcium equation. Against this background, the main contributions of this work are the following.
\begin{itemize}
\item We provide a well-posedness theory for the fully mixed continuous formulation. Rather than treating the coupled system directly, we reduce it to an integro-differential problem for the velocity and concentrations. A Banach fixed-point argument yields a unique local solution on a time interval whose length can be chosen independently of the initial data, allowing the construction to be iterated to obtain a global solution. The remaining mixed variables are then recovered uniquely.

\item We develop a semidiscrete framework for general conforming finite element spaces satisfying suitable compatibility and approximation properties. Although the resulting scheme is a differential-algebraic system, its algebraic variables can be eliminated under a single inf-sup condition. This leads to an ordinary differential equation for the remaining variables and provides existence and uniqueness of the semidiscrete solution together with a discrete stability estimate.

\item We establish a priori error estimates within this general discrete framework. The analysis exploits an interpolation operator that both commutes with the divergence and preserves the weak symmetry pairing. This structure allows the spin variable to be eliminated from the main energy argument and yields optimal-order estimates.

\item We show that the abstract discrete assumptions are satisfied by three established families of weakly symmetric mixed finite elements, namely the Arnold--Falk--Winther, Farhloul--Fortin, and PEERS families, and derive the corresponding convergence rates. For the Farhloul--Fortin family, we additionally obtain a superconvergent local post-processing of the displacement by exploiting the higher approximation order of the stress space.

\item Finally, we assess the theoretical results and the behaviour of the coupled model through numerical experiments in two and three dimensions. These experiments confirm the predicted convergence rates, reproduce characteristic creep and cyclic-loading responses of the Kelvin--Voigt model, and illustrate the influence of the mechanochemical feedback on calcium wavetrains.
\end{itemize}

\medskip
\noindent
\textbf{Outline.} The remainder of this section introduces the notation and functional setting used throughout the paper. Section~\ref{sec:model} presents the governing equations for viscoelasticity and reaction–diffusion in both primal and mixed form and specifies the assumptions on the constitutive laws. In Section~\ref{sec:solvability}, we establish the well-posedness of the continuous problem through a reduced auxiliary integro-differential system. Section~\ref{sec:FE} introduces the mixed finite element discretisation and establishes the well-posedness of the continuous-in-time semidiscrete problem under suitable assumptions. In Section~\ref{sec:errors}, we derive a priori error estimates for the semidiscrete problem in several norms. Section~\ref{sec:FE-examples} presents specific families of finite element spaces satisfying the assumptions of the abstract framework, derives the corresponding convergence rates, and discusses the superconvergent displacement post-processing. Finally, Section~\ref{sec:results} presents numerical experiments in two and three dimensions, including convergence tests, viscoelastic benchmarks, and simulations of calcium wavetrains in epithelial tissue.

\medskip
\noindent 
\textbf{Preliminary notations.} 
Let us consider a simply connected bounded and Lipschitz domain $\Omega\subset\mathbb{R}^d$, $d \in \{2,3\}$ occupied by a viscoelastic body. 
The domain boundary $\partial\Omega$ is partitioned into disjoint sub-boundaries where homogeneous displacement (and velocity) and traction-type boundary conditions are imposed  $\partial\Omega:= \overline{\Gamma^{\bu}} \cup \overline{\Gamma^{\bsigma}}$, and it is assumed for the sake of simplicity that both sub-boundaries are non-empty $|\Gamma^{\bu}|\cdot|\Gamma^{\bsigma}|>0$. Throughout the text, given a normed space $S$, by $\bS$ and $\bbS$ we will denote the vector and tensor extensions $S^d$ and $S^{d\times d}$, respectively. We introduce the Hilbert spaces $\bH(\vdiv;\Omega)=\left\{\bw \in \bL^2(\Omega): \vdiv(\bw) \in \LL^2(\Omega) \right\}$ and $\bH_\star(\vdiv;\Omega) := \{\bw\in \bH(\vdiv;\Omega): \bw \cdot \bn = 0 \text{ on } \Gamma^{\bu}\}$ endowed with the norm $\norm{\bw}_{\bH(\vdiv;\Omega)}^2:=\norm{\bw}^2_{\bL^2(\Omega)}+\norm{\vdiv(\bw)}^2_{\LL^2(\Omega)}$.
Next, we recall the definition of the tensorial Hilbert space
$\bbH(\bdiv;\Omega)=\left\{\btau \in \bbL^2(\Omega): \bdiv(\btau) \in \bL^2(\Omega) \right\}$ with its usual norm
$\norm{\btau}_{\bbH(\bdiv;\Omega)}^2:=\norm{\btau}^2_{\bbL^2(\Omega)}+\norm{\bdiv(\btau)}^2_{\bL^2(\Omega)}$,  
where the divergence acts on the rows of $\btau$.  
We also define the tensor spaces 
\begin{gather*}
    \bbH_\star(\bdiv;\Omega) := \big\{\btau\in \bbH(\bdiv;\Omega): \btau\bn = \cero \text{ on } \Gamma^{\bsigma}\big\}, 
    \\
    \bbL^2_{\mathrm{skew}}(\Omega) := \{\eeta \in \bbL^2(\Omega): \eeta = - \eeta^\top\} , \qan \bbL^2_{\mathrm{sym}}(\Omega) := \{\eeta \in \bbL^2(\Omega): \eeta = \eeta^\top\} \, .
\end{gather*}
We next introduce the notation for function spaces defined on the time interval $[0,t_{\rm end}]$ and taking values in a separable Hilbert space $H$ with norm $\|\bullet\|_H$. For a nonnegative integer $m$, and for $1\leq p< \infty$, we denote by $\LL^{p}(0,t_{\rm end};H)$ and $\rW^{m,p}(0,t_{\rm end};H)$ the spaces of (equivalence classes of) Bochner measurable functions $f:[0,t_{\mathrm{end}}]\rightarrow H$ such that
\begin{gather*}
\norm{f}_{\LL^{p}(0,t_{\rm end};H)}^p:=\int^{t_{\mathrm{end}}}_{0}\norm{f(t)}_{H}^p \dt < \infty \quad \text{if }p<\infty \, , \quad \norm{f}_{\LL^{\infty}(0,t_{\rm end};H)} := \esssup_{0\leq t \leq t_{\rm end}} \|f(t)\|_{H}<\infty \, , \\
\norm{f}_{\rW^{m,p}(0,t_{\rm end};H)}^p:=\sum_{l=0}^{m}\big\|\partial_t^l f\big\|_{\LL^p(0,t_{\rm end};H)}^p < \infty \quad \text{if }p<\infty \, , \\
\qan \norm{f}_{\rW^{m,\infty}(0,t_{\rm end};H)} := 
\sum_{l=0}^m \|\partial_t^l f\|_{\LL^\infty(0,t_{\rm end};H)}\, ,
\end{gather*}
where $\partial_t^l$ denotes the $l$-th (weak) time derivative of $f$. When $p=2$, we write $\rW^{m,2}(0,t_{\rm end};H) = \HH^m(0,t_{\rm end};H)$. We also use the space $C^m([0,t_{\mathrm{end}}],H)$ of $m$-times continuously differentiable functions. Furthermore, for brevity, we write $\dot{f}$ to denote $\partial_t f$.

\section{Governing equations}\label{sec:model}

Consider the equations of motion of a viscoelastic body for which the momentum balance reads
\begin{equation}\label{eq:momentum-balance}
\varrho\, \ddot{\bu} - \bdiv \bsigma = \varrho \, \bb \qin  \Omega\times (0,t_{\mathrm{end}}] \, ,
\end{equation}
with $\varrho > 0$ denoting the material density and $\bb$ representing the body force per unit mass. Regarding the constitutive relation for the stress, we assume that calcium localisation, represented by the concentration $\varphi$, affects the deformation of the viscoelastic structure. This effect is incorporated through an active stress contribution within a Kelvin--Voigt type model, in which an elastic spring and a viscous dashpot act in parallel, leading to
\begin{equation*}
\bsigma =\bsigma_1 + \bsigma_2  \, , \quad \text{with}\quad \bsigma_1:= \cC_1\,\beps(\bu) \qan \bsigma_2:=\cC_2\, \beps(\bv) - \alpha (\varphi)\, \mathbb{I} \, ,
\end{equation*}
where $\bv := \dot{\bu}$ is the velocity field, $\cC_i$ are the fourth-order isotropic stiffness tensors characterised by their actions
\begin{equation*}
\cC_i\,\btau = 2\mu_i\,\btau + \lambda_i \tr(\btau)\, \mathbb{I}\, ,\qquad \forall\, i \in \{1,2\} \, ,
\end{equation*}
with Lam\'e parameters $\mu_i,\lambda_i>0$. The function $\alpha$ is a Hill-type function that modulates the activation process and accounts for the degree of cooperative binding at the microscale. It is defined as
\begin{equation}\label{eq:def-Hill-function}
\alpha(\varphi) = \alpha_0 \frac{\varphi^2}{1+\varphi^2}\, ,
\end{equation}
where $\alpha_0>0$ is the saturation constant of the model. Dependences of this type are typical in models for mechanobiology (see, for instance, \cite{jones12,murray2003mathematical}). Notice that $\alpha \in C^\infty(\mathbb{R})$ is bounded with $0\leq \alpha \leq \alpha_0$ in $\R$, and all its derivatives are bounded; in particular $\alpha$ is Lipschitz continuous, that is, there exists $L_\alpha > 0$ such that
\begin{equation}\label{eq:alpha-Lipschitz}
|\alpha(\varphi^1)-\alpha(\varphi^2)|
\leq L_\alpha \,  |\varphi^1-\varphi^2|
\qquad \forall \, \varphi^1,\varphi^2 \in \mathbb{R} \, .
\end{equation}
The material properties are characterised by the compliance tensors $\cA_i:=\cC_i^{-1}$, which are symmetric, bounded, and uniformly positive definite linear operators. Here $\cA_1$ corresponds to the elastic spring and $\cA_2$ to the viscous dashpot of the Kelvin--Voigt model, and their action is given by
\begin{equation*}
\cA_i\btau= \frac{1}{2\mu_i}\biggl(\btau - \frac{\lambda_i}{2\mu_i+d\lambda_i} \tr(\btau)\mathbb{I}\biggr) \quad \forall\, i \in \{1,2\} \, ,
\end{equation*}
so that the constitutive relations read
\begin{equation}\label{eq:Ais} 
\cA_1\bsigma_1 = \beps(\bu)  \qan 
\cA_2\bsigma_2 = \beps(\bv) - \frac{\alpha(\varphi)}{2\mu_2+d\lambda_2}\, \bbI \, ,
\end{equation}
the second one expressing that the dashpot stress is proportional to the strain rate. Differentiating the first relation above in time gives $\cA_1\dot{\bsigma}_1 = \beps(\bv)$.

Finally, in order to impose the balance of angular momentum in a weak manner, we introduce the spin tensor $\bomega = \frac12(\bnabla \bv-(\bnabla\bv)^\top )$. The viscoelasticity system then reads
\begin{equation}\label{eq:strong-viscoelast}
\begin{alignedat}{2}
\cA_1\dot{\bsigma}_1
&= \bnabla\bv-\bomega 
&\qquad &\text{in }\Omega\times (0,t_{\mathrm{end}}]\, ,\\
\cA_2\bsigma_2
&= \bnabla\bv-\bomega - \frac{\alpha(\varphi)}{2\mu_2+d\lambda_2}\, \bbI
&\qquad &\text{in }\Omega\times (0,t_{\mathrm{end}}]\, ,\\
\varrho\,\dot{\bv}
-\bdiv(\bsigma_1+\bsigma_2)
&= \varrho\,\bb 
&\qquad &\text{in }\Omega\times (0,t_{\mathrm{end}}]\, ,\\
\bsigma_1+\bsigma_2
&= (\bsigma_1+\bsigma_2)^\top
&\qquad &\text{in }\Omega\times (0,t_{\mathrm{end}}]\, ,
\end{alignedat}
\end{equation}
with mixed boundary conditions
\begin{equation}\label{eq:bc-mech}
\bv = \cero \qon \Gamma^\bu \times (0,t_{\rm end}) \qan (\bsigma_1+\bsigma_2)\bn = \cero \qon \Gamma^{\bsigma}\times (0,t_{\mathrm{end}}]\, ,
\end{equation}
and prescribed initial data $\bu(0)=\bu_0$ and $\bv(0) = \bv_0$.

In addition to the viscoelasticity equations, we consider the motion of the calcium concentration $\varphi(t):\Omega\to\mathbb{R}$ and of the amount of inactivated Inositol Triphosphate $\theta(t):\Omega\to\mathbb{R}$, whose dynamics are governed by
\begin{subequations}
\begin{alignat}{2}
\label{eq:calcium-dynamics-1}
\dot{\varphi} - \vdiv(\diff \nabla \varphi) &=  f(\varphi,\theta) + \beta\, \vdiv\,\bu &\qquad& \text{in $\Omega\times (0,t_{\mathrm{end}}]$},\\
\label{eq:calcium-dynamics-2}
\dot{\theta} &= g(\varphi,\theta) &\qquad& \text{in $\Omega\times (0,t_{\mathrm{end}}]$} \, ,
\end{alignat}
\end{subequations}
with mixed boundary conditions 
\begin{equation}\label{eq:bc-phi}
\varphi=0 \qon \Gamma^{\bsigma}\times (0,t_{\mathrm{end}}] \qan \diff \nabla \varphi \cdot \bn =0 \qon \Gamma^{\bu}\times (0,t_{\mathrm{end}}] \, ,
\end{equation}
and the initial conditions $\varphi(0)=\varphi_0$, $\theta(0)=\theta_0$. The diffusion coefficient $\diff : \Omega\to \mathbb{R}^{d\times d}$ is a symmetric, uniformly bounded, and uniformly positive-definite tensor satisfying
\begin{equation}\label{eq:properties-D}
(\diff\, \bv) \cdot \bv \geq \underline{D}\, |\bv|^2  \quad \forall \, \bv \in \mathbb{R}^d, \qan
\|\diff(\bx)\|\leq \overline{D} \quad \text{for a.e. }\bx \in \Omega \, ,
\end{equation}
where $0<\underline{D}\leq \overline{D}$ are known constants, and $\beta>0$ is a coupling parameter. For the forthcoming analysis we suppose that the reaction terms $f$ and $g$ are globally Lipschitz continuous in both arguments:
\begin{subequations}
\begin{equation}\label{eq:assumption-f}
|f(\varphi^1,\theta^1)- f(\varphi^2,\theta^2)| \leq L_{f1} \, |\varphi^1-\varphi^2| + L_{f2} \, |\theta^1-\theta^2| \qquad \forall\, \varphi^1,\varphi^2,\theta^1,\theta^2 \in  \mathbb{R}\, ,
\end{equation}
\begin{equation}\label{eq:assumption-g}
|g(\varphi^1,\theta^1)- g(\varphi^2,\theta^2)| \leq L_{g1} \, |\varphi^1-\varphi^2| + L_{g2}\, |\theta^1-\theta^2| \qquad \forall\, \varphi^1,\varphi^2,\theta^1,\theta^2 \in \mathbb{R}\, . 
\end{equation}
\end{subequations}
\begin{remark}
Some physically meaningful reaction terms are given by
\begin{equation}\label{eq:specific-reaction-functions}
f (\varphi,\theta) = \gamma_0\, \theta\, \frac{\gamma_1+\varphi}{1+\varphi}-\gamma_2 \, \frac{\varphi}{\gamma_3+\varphi} \qan g(\varphi,\theta) = \frac{1}{1+\varphi^2}-\theta \, ,
\end{equation}
where $\gamma_i>0$, $i\in \{0,1,2,3\}$, are model parameters \cite{kaouri19,kaouri22}. However, $f (\varphi,\theta)$ is not globally Lipschitz continuous. In particular, $f$ might exhibit singular behaviour depending on the values attained by $\varphi$. In practice, as illustrated by the numerical experiments in Section \ref{sec:results}, both $\varphi$ and $\theta$ remain uniformly bounded and are uniformly away from the singular values $-1$, $-\gamma_3$, and therefore $f$ and $g$ are Lipschitz continuous on the range of the corresponding solutions. Moreover, a rigorous analysis of this case could be carried out by introducing suitable regularisations of the functions in \eqref{eq:specific-reaction-functions} and subsequently passing to the limit. To avoid these additional technicalities, we restrict our analysis to reaction terms satisfying the assumptions stated above.
\end{remark}

Using trace properties and the first relation in \eqref{eq:Ais}, the volumetric contribution in the last term on the right-hand side of \eqref{eq:calcium-dynamics-1} can be rewritten as
\begin{equation*}
\beta \, \vdiv\, \bu = \beta \tr(\cA_1\bsigma_1) = \frac{\beta}{2\mu_1+d\lambda_1}\tr(\bsigma_1) \, .
\end{equation*}

The model adopted here for coupled viscoelasticity and calcium dynamics was studied (using a different formulation) in~\cite{kaouri22}. Here we also include the diffusive flux as an additional unknown,
\[\bz = \diff \, \nabla \varphi \qin  \Omega\times (0,t_{\mathrm{end}}] \, ,\]
so that $\nabla\alpha(\varphi) = \alpha'(\varphi)\nabla\varphi = \alpha'(\varphi)\,\diff^{-1}\bz$. Following the usual approach, in weak form our problem reads: find $(\bsigma_1,\bsigma_2,\bv,\bomega,\bz,\varphi,\theta):(0,t_{\mathrm{end}}] \to \bbL^2(\Omega)\times \bbL^2(\Omega) \times \bL^2(\Omega)\times \bbL^2_{\mathrm{skew}}(\Omega)\times \bH_\star(\vdiv;\Omega)\times \LL^2(\Omega)\times \LL^2(\Omega)$ such that $(\bsigma_1+\bsigma_2)(t)\in \bbH_\star(\bdiv;\Omega)$, $\bsigma_1(0)=\cC_1\, \beps(\bu_0)$, $\bv(0)=\bv_0$, $\varphi(0)=\varphi_0$, $\theta(0)=\theta_0$, and, for a.e. $t\in (0,t_{\rm end}]$,
\begin{equation}\label{eq:mixed-weak}
\begin{array}{rcl}
(\cA_1\dot{\bsigma}_1,\btau_1)_{\Omega} + (\bdiv\, \btau_1,\bv)_{\Omega} + (\btau_1,\bomega)_{\Omega} & = & 0 \, , \\[1.25ex]
(\cA_2\bsigma_2,\btau_2)_{\Omega} + (\bdiv\, \btau_2,\bv)_{\Omega} + (\btau_2,\bomega)_{\Omega}  & = & - \frac{1}{2\mu_2+d\lambda_2}\, (\alpha(\varphi),\tr(\btau_2))_\Omega \, ,\\[1.25ex]
(\varrho\, \dot{\bv},\bw)_{\Omega} -(\bdiv(\bsigma_1+\bsigma_2),\bw)_{\Omega} - (\bsigma_1+\bsigma_2,\eeta)_{\Omega} & = & (\varrho\, \bb,\bw)_{\Omega} \, , \\[1.25ex]
(\diff^{-1}\bz,\br)_{\Omega} +(\vdiv\,\br,\varphi)_{\Omega} & = & 0 \, ,\\[1.25ex]
(\dot{\varphi},\psi)_{\Omega} - (\vdiv\,\bz, \psi)_{\Omega}  & = & \left(f(\varphi,\theta),\psi \right)_{\Omega} + \tfrac{\beta}{2\mu_1+d\lambda_1}(\tr\bsigma_1,\psi)_{\Omega}\, ,\\[1.25ex]
(\dot{\theta},\vartheta)_{\Omega} & = & (g(\varphi,\theta),\vartheta)_{\Omega} \, ,
\end{array}
\end{equation}
for all $(\btau_1,\btau_2,\bw,\eeta,\br,\psi,\vartheta)\in \bbH_\star(\bdiv;\Omega)\times \bbH_\star(\bdiv;\Omega) \times \bL^2(\Omega)\times \bbL^2_{\mathrm{skew}}(\Omega)\times \bH_\star(\vdiv;\Omega)\times \LL^2(\Omega)\times \LL^2(\Omega)$.

\section{Continuous solvability}\label{sec:solvability}

In this section, we establish the existence and uniqueness of a solution to~\eqref{eq:mixed-weak} by introducing an auxiliary integro-differential system formulated in terms of the primal variables and subsequently reconstructing the mixed variables.

\subsection{The reduced integro-differential equations}\label{sec:integro-diff-primal}

Recalling that $\bv = \dot{\bu}$ and writing
\begin{equation}\label{eq:u-recovery}
\bu(t) = \bu_0 + \int_0^t \bv(s)\ds\, ,
\end{equation}
the symmetric gradient of the displacement satisfies $\beps(\bu(t)) = \beps(\bu_0) + \int_0^t \beps(\bv(s))\ds$, so that the elastic and viscous stress components read
\begin{equation}\label{eq:stress-recovery}
\bsigma_1(t) = \cC_1\, \beps(\bu_0) + \cC_1\, \int_0^t \beps(\bv)\ds \, , \qan \bsigma_2(t) = \cC_2\, \beps(\bv(t)) - \alpha(\varphi(t)) \mathbb{I} \, .
\end{equation}
Substituting~\eqref{eq:u-recovery}--\eqref{eq:stress-recovery} into the momentum balance~\eqref{eq:momentum-balance} and replacing~\eqref{eq:u-recovery} into the volumetric contribution in~\eqref{eq:calcium-dynamics-1} eliminates $\bu$ in favour of $\bv$, leading to the reduced problem: find $(\bv,\varphi,\theta)$ such that
\begin{subequations}\label{eq:reduced-strong}
\begin{align}
\varrho\, \dot{\bv} - \bdiv\left( \cC_2\,\beps(\bv) + \cC_1\!\int_0^t \beps(\bv)\ds \right) + \nabla\alpha(\varphi) &= \varrho\,\bb + \bdiv\big(\cC_1\beps(\bu_0)\big)\, , \\
\dot{\varphi} - \vdiv(\diff\nabla\varphi) &= f(\varphi,\theta) + \beta\left( \vdiv\,\bu_0 + \int_0^t \vdiv\, \bv\ds \right)\, , \\
\dot{\theta} &= g(\varphi,\theta)\, ,
\end{align}
\end{subequations}
in $\Omega\times(0,t_{\mathrm{end}}]$, together with the boundary conditions, compatible with \eqref{eq:bc-mech},\eqref{eq:bc-phi}
\[
\bv = \cero \quad\text{on }\Gamma^{\bu},\qquad
\Big(\cC_2\beps(\bv)+\cC_1\!\int_0^t\beps(\bv)\ds + \cC_1\beps(\bu_0) - \alpha(\varphi)\bbI\Big)\bn = \cero \quad\text{on }\Gamma^{\bsigma},
\]
\[
\varphi = 0 \quad\text{on }\Gamma^{\bsigma},\qquad \diff\nabla\varphi\cdot\bn = 0 \quad\text{on }\Gamma^{\bu},
\]
and the initial conditions $\bv(0)=\bv_0$, $\varphi(0)=\varphi_0$, $\theta(0)=\theta_0$.

\begin{remark}\label{rem:closed}
System \eqref{eq:reduced-strong} is closed in the unknowns $(\bv,\varphi,\theta)$: once it has been solved, the displacement is recovered by the quadrature \eqref{eq:u-recovery} and the stresses by \eqref{eq:stress-recovery}, the datum $\bu_0$ entering only as a forcing offset. For consistency with the strong traction condition we assume the compatibility relation $\big(\cC_1\beps(\bu_0)+\cC_2 \beps(\bv_0)-\alpha(\varphi_0)\bbI\big)\bn = \cero$ on $\Gamma^{\bsigma}$; it is not required by the weak formulation below, where the traction condition is natural.
\end{remark}

In order to introduce a weak formulation for~\eqref{eq:reduced-strong}, we define the spaces
\[
\bV := \big\{ \bw\in \bH^1(\Omega)~: ~\bw = \cero \qon \Gamma^{\bu}\big\}\, , \qquad
\rQ := \big\{ \psi\in \HH^1(\Omega)~:~ \psi = 0 \qon \Gamma^{\bsigma} \big\} \, ,
\]
endowed with the norms $\|\bullet\|_{\bV} := \|\bullet\|_{\bH^1(\Omega)}$ and $\|\bullet\|_{\rQ}:=\|\bullet\|_{\HH^1(\Omega)}$, respectively. Furthermore, for $i\in \{1,2\}$, let $\wt a_i:\bV\times\bV\to\R$ be the symmetric bounded bilinear form given by
\begin{equation*}
\wt a_i(\bz,\bw) := (\cC_i\,\beps(\bz),\beps(\bw))_\Omega \, .
\end{equation*}
Since $|\Gamma^{\bu}|>0$, Korn's inequality and the uniform positive definiteness of $\cC_i$ yield the coercivity
\begin{equation}\label{eq:coercivity-a}
\wt a_i(\bw,\bw) \geq \kappa_i \norm{\bw}_{\bV}^2 \qquad \forall\,\bw\in\bV,\qquad \forall\, i\in\{1,2\} \, ,
\end{equation}
with $\kappa_i>0$ depending only on $\mu_i$, $\Omega$ and $\Gamma^{\bu}$. Likewise, by~\eqref{eq:properties-D} and the Poincar\'e inequality, we also have
\begin{equation}\label{eq:coercivity-D}
(\diff\nabla\psi,\nabla\psi)_\Omega \geq \underline{D}\,\norm{\nabla\psi}_{0,\Omega}^2 \geq  \kappa_D\,\norm{\psi}_{\rQ}^2  \qquad \forall\, \psi \in \rQ \, ,
\end{equation}
where $\kappa_D>0$ depends only on $\underline{D}$, $\Omega$ and $\Gamma^\bsigma$.

Now, testing \eqref{eq:reduced-strong} against $(\bw,\psi,\vartheta)\in\bV\times\rQ\times\LL^2(\Omega)$ and integrating by parts (the boundary contributions vanishing by virtue of the boundary conditions), we arrive at the weak formulation: find $\bv\in\LL^2(0,t_{\mathrm{end}};\bV)$, $\varphi\in\LL^2(0,t_{\mathrm{end}};\rQ)$, $\theta\in \HH^1(0,t_{\mathrm{end}};\LL^2(\Omega))$, with $\dot{\bv}\in\LL^2(0,t_{\mathrm{end}};\bV')$ and $\dot{\varphi}\in\LL^2(0,t_{\mathrm{end}};\rQ')$, such that $\bv(0)=\bv_0$, $\varphi(0)=\varphi_0$, $\theta(0)=\theta_0$, and for a.e. $t\in(0,t_{\mathrm{end}}]$,
\begin{subequations}\label{eq:reduced-weak}
\begin{align}
\varrho\,\langle\dot{\bv},\bw\rangle + \wt a_2(\bv,\bw) + \int_0^t \wt a_1(\bv(s),\bw)\ds &= (\varrho\,\bb,\bw)_\Omega - \wt a_1(\bu_0,\bw) + (\alpha(\varphi),\vdiv\,\bw)_\Omega\, , \label{eq:rw-v}\\
\langle\dot{\varphi},\psi\rangle + (\diff\nabla\varphi,\nabla\psi)_\Omega - \beta\, \left( \int_0^t\vdiv\,\bv\ds,\psi\right)_\Omega &= (f(\varphi,\theta),\psi)_\Omega + \beta(\vdiv\,\bu_0,\psi)_\Omega \, , \label{eq:rw-phi}\\
(\dot{\theta},\vartheta)_\Omega &= (g(\varphi,\theta),\vartheta)_\Omega \, , \label{eq:rw-theta}
\end{align}
\end{subequations}
for all $(\bw,\psi,\vartheta)\in\bV\times\rQ\times\LL^2(\Omega)$.

Here, and throughout this work, we assume the following regularity on the data.

\begin{assumption}\label{assum:data-regularity}
The initial data and the body force load satisfy
\begin{gather*}
\bu_0 \in \bH^1(\Omega) \, , \quad
\bv_0 \in \bL^2(\Omega) \, , \quad
\varphi_0 \in \LL^2(\Omega) \, , \quad
\theta_0 \in \LL^2(\Omega) \qan
\bb \in \LL^2(0,t_{\rm end};\bL^2(\Omega))\, . 
\end{gather*}
\end{assumption}

\begin{remark}\label{rem:lions-magenes}
Since $\bV \subset \bL^2(\Omega) \subset \bV'$ is a Gelfand triple, the Lions--Magenes lemma yields $\bv \in C([0,t_{\rm end}];\bL^2(\Omega))$ whenever $\bv \in \LL^2(0,t_{\rm end};\bV)$ and $\dot{\bv} \in \LL^2(0,t_{\rm end};\bV')$. Therefore, the initial condition $\bv(0)=\bv_0$ is meaningful. The same reasoning applies to $\varphi$, ensuring that $\varphi(0)=\varphi_0$ is well defined as well.
\end{remark}

\subsection{Well-posedness of the reduced problem}\label{sec:well-posedness-reduced-problem}

We first prove the solvability of the uncoupled linear viscoelastic subproblem (cf.~\eqref{eq:rw-v}) with a prescribed load.

\begin{lemma}\label{lem:primal-viscoelasticity-well-posedness}
There exists $T\in(0,t_{\rm end}]$, depending only on $\kappa_2$, $\varrho$, $\mu_1$ and $\lambda_1$, such that for every $\sF\in\LL^2(0,T;\bV')$ there is a unique $\bv\in\LL^2(0,T;\bV)$ with $\dot{\bv}\in\LL^2(0,T;\bV')$ satisfying $\bv(0)=\bv_0$ and
\begin{equation}\label{eq:primal-viscoelasticity-well-posedness}
\varrho\,\langle\dot{\bv},\bw\rangle + \wt a_2(\bv,\bw) + \int_0^t \wt a_1(\bv(s),\bw)\ds = \langle\sF(t),\bw\rangle \qquad\forall\,\bw\in\bV \, ,\ \text{a.e. }t\in(0,T] \, .
\end{equation}
\end{lemma}
\begin{proof}
We proceed by a fixed-point argument. Let $\sS:\LL^2(0,T;\bV) \to \LL^2(0,T;\bV)$ be the operator such that, for each $\wh \bv \in \LL^2(0,T;\bV)$, $\sS(\wh \bv) := \bv$, where $\bv$ is the unique solution (to be confirmed below) to
\begin{equation}\label{eq:primal-viscoelasticity-well-posedness-proof-1}
\varrho\,\langle\dot{\bv},\bw\rangle + \wt a_2(\bv,\bw) = \langle \cG[\wh \bv](t),\bw\rangle \qquad\forall\,\bw\in\bV \, ,\ \text{a.e. }t\in(0,T] \, ,
\end{equation}
and $\bv(0)=\bv_0$, where $\langle \cG[\wh \bv](t),\bw \rangle := \langle\sF(t),\bw\rangle  - \int_0^t \wt a_1(\wh \bv(s),\bw)\ds$. It is straightforward to verify that $\cG[\wh \bv] \in \LL^2(0,T;\bV')$, since $\wh \bv \in \LL^2(0,T;\bV)$ and $\sF \in \LL^2(0,T;\bV')$. Combined with the coercivity of $\wt a_2$ (cf.~\eqref{eq:coercivity-a}), this ensures that~\eqref{eq:primal-viscoelasticity-well-posedness-proof-1} admits a unique solution, as it corresponds to the standard weak formulation of a linear second-order parabolic problem.  Moreover, the solution $\bv$ belongs to $\LL^2(0,T;\bV)$ with $\dot{\bv}\in \LL^2(0,T;\bV')$. Hence, $\sS$ is indeed well-defined. 

Next, given $\wh \bv_1, \wh \bv_2\in \LL^2(0,T;\bV)$, we let $\bv_i := \sS(\wh \bv_i)$, $i\in \{1,2\}$. Since each $\bv_i$ is the solution to~\eqref{eq:primal-viscoelasticity-well-posedness-proof-1} with $\cG[\wh \bv_i]$ as right-hand side, we have that the difference $\delta \bv := \bv_1-\bv_2$ satisfies
\begin{equation}\label{eq:primal-viscoelasticity-well-posedness-proof-2}
\varrho\,\langle\delta \dot{\bv},\bw\rangle + \wt a_2(\delta\bv,\bw) = \langle \cG[\delta\wh \bv](t),\bw\rangle \qquad\forall\,\bw\in\bV \, ,\ \text{a.e. }t\in(0,T] \, ,
\end{equation}
and $\delta \bv(0) = \cero$, where $\delta\wh \bv := \wh \bv_1- \wh \bv_2$. By testing~\eqref{eq:primal-viscoelasticity-well-posedness-proof-2} against $\delta \bv(t)$ and using the coercivity of $\wt a_2$ (cf.~\eqref{eq:coercivity-a}), we find that
\begin{equation}\label{eq:primal-viscoelasticity-well-posedness-proof-3}
\frac{\varrho}{2} \, \partial_t \|\delta \bv(t)\|_{\bL^2(\Omega)}^2 + \kappa_2\, \|\delta \bv(t)\|_{\bV}^2 \leq - \int_0^t \wt a_1( \delta \wh \bv(s),\delta \bv(t)) \ds \, .
\end{equation}
The right-hand side is estimated by exploiting the definition of $\wt a_1$ and applying the Cauchy--Schwarz and Young inequalities. In fact,
\begin{gather*}
\left| \int_0^t \wt a_1( \delta \wh \bv(s),\delta \bv(t)) \ds \right| \leq \|\cC_1\|\,\|\delta \bv(t)\|_{\bV}\, \int_0^t  \|\delta \wh \bv(s)\|_{\bV} \ds \leq  \sqrt{t}\, \|\cC_1\|\,\|\delta \bv(t)\|_{\bV}\, \left\{ \int_0^t \|\delta \wh \bv(s)\|_{\bV}^2\ds \right\}^{1/2} \\
\leq \fracd{\kappa_2}{2}\, \|\delta \bv(t)\|_{\bV}^2 + \fracd{\|\cC_1\|^2\, t}{2\kappa_2}\, \int_0^t \|\delta \wh \bv(s)\|_{\bV}^2\ds\,.
\end{gather*}
Substituting this estimate into~\eqref{eq:primal-viscoelasticity-well-posedness-proof-3}, integrating from $0$ to some $t\in (0,T]$, using that $\delta \bv(0)=\cero$, and applying the definition of the Bochner norms, we arrive at
\begin{equation*}
\|\delta \bv\|_{\LL^\infty(0,T;\bL^2(\Omega))}^2 + \|\delta \bv\|_{\LL^2(0,T;\bV)}^2 \leq C\, T^2 \, \|\delta \wh \bv\|_{\LL^2(0,T;\bV)}^2\, ,
\end{equation*} 
with $C>0$ depending only on $\kappa_2$, $\varrho$, and $\|\cC_1\|$. In particular, this estimate gives us that
\begin{equation*}
\|\sS(\wh \bv_1)-\sS(\wh \bv_2)\|_{\LL^2(0,T;\bV)} \leq \sqrt{C}\, T\, \|\wh \bv_1-\wh \bv_2\|_{\LL^2(0,T;\bV)}\, ,
\end{equation*}
that is, $\sS$ is Lipschitz continuous. Moreover, choosing $T= \min\{1/(2\sqrt{C}),t_{\rm end}\}$, we deduce that $\sS$ is a contractive map in the Banach space $\LL^2(0,T;\bV)$. Consequently, by the Banach fixed-point theorem, we conclude that $\sS$ has a unique fixed point, and, hence,~\eqref{eq:primal-viscoelasticity-well-posedness} has a unique solution. This concludes the proof.
\end{proof}

The following result establishes a Lipschitz continuity property of the solution to~\eqref{eq:primal-viscoelasticity-well-posedness} with respect to the prescribed load functionals.

\begin{lemma}\label{lem:difference-of-velocities}
Let $\sF_1, \sF_2\in \LL^2(0,T;\bV')$ with $T$ given by Lemma~\ref{lem:primal-viscoelasticity-well-posedness}. For $i\in \{1,2\}$, let $\bv_i$ be the unique solution to~\eqref{eq:primal-viscoelasticity-well-posedness} with the respective functional $\sF_i$. Then, there exist $0<T'\leq T$ and a positive constant $C$, depending only on $\varrho$, $\kappa_2$, $\mu_1$, $\lambda_1$, but independent of the initial data, such that
\begin{subequations}
\begin{align}\label{eq:difference-of-velocities}
\|\bv_1-\bv_2\|_{\LL^\infty(0,T';\bL^2(\Omega))}^2 + \|\bv_1-\bv_2\|_{\LL^2(0,T';\bV)}^2 &\leq C\, \|\sF_1-\sF_2\|_{\LL^2(0,T';\bV')}^2 \, ,
\quad \text{and} \\
\label{eq:difference-of-displacements}
\|\bu_1-\bu_2\|_{\LL^2(0,T';\bH^1(\Omega))}^2 &\leq C\, (T')^2\, \|\sF_1-\sF_2\|_{\LL^2(0,T';\bV')}^2 \, ,
\end{align}\end{subequations}
where $\bu_i(t) := \bu_0 + \int_0^t \bv_i(s)\ds$.
\end{lemma}
\begin{proof}
Let $\delta \bv := \bv_1-\bv_2$. By subtracting equation~\eqref{eq:primal-viscoelasticity-well-posedness} with each prescribed load and testing against $\bw = \delta \bv$, we readily obtain that
\begin{equation}\label{eq:difference-of-velocities-proof-1}
\frac{\varrho}{2}\, \partial_t \|\delta \bv(t)\|_{\bL^2(\Omega)}^2 + \wt a_2(\delta \bv(t),\delta \bv(t)) = \langle \sF_1(t)-\sF_2(t), \delta \bv \rangle - \int_0^t \wt a_1(\delta \bv(s),\delta \bv(t))\ds\, .
\end{equation}
Since $\bu_i = \bu_0 + \int_0^t \bv_i(s)\ds$, it follows that $\delta \bu := \bu_1-\bu_2 = \int_0^t \delta \bv(s)\ds$. Hence,
\begin{equation*}
\left| \int_0^t \wt a_1(\delta \bv(s),\delta \bv(t)) \ds \right| = \left| \wt a_1(\delta \bu(t),\delta \bv(t)) \right| \leq \|\cC_1\|\, \|\delta \bu\|_{\bH^1(\Omega)}\, \|\delta \bv\|_{\bV}\, .
\end{equation*}
Substituting this estimate into~\eqref{eq:difference-of-velocities-proof-1}, and using the coercivity of $\wt a_2$ (cf.~\eqref{eq:coercivity-a}) together with Young's inequality to bound the right-hand side, we obtain
\begin{equation*}
\begin{gathered}
\frac{\varrho}{2}\, \partial_t \|\delta \bv\|_{\bL^2(\Omega)}^2 + \kappa_2\, \|\delta \bv\|_{\bV}^2 \leq \|\sF_1-\sF_2\|_{\bV'}\, \|\delta \bv\|_{\bV} + \|\cC_1\|\, \|\delta \bu\|_{\bH^1(\Omega)}\, \|\delta \bv\|_{\bV} \\
\leq \fracd{1}{2\kappa_2}\, \|\sF_1-\sF_2\|_{\bV'}^2 + \fracd{\kappa_2}{2}\, \|\delta \bv\|_{\bV}^2 + \frac{\|\cC_1\|^2}{\kappa_2}\, \|\delta \bu\|_{\bH^1(\Omega)}^2 + \frac{\kappa_2}{4} \|\delta \bv\|_{\bV}^2  \, ,
\end{gathered}
\end{equation*}
which becomes
\begin{equation}\label{eq:difference-of-velocities-proof-2}
\begin{gathered}
\frac{\varrho}{2}\, \partial_t \|\delta \bv\|_{\bL^2(\Omega)}^2 + \frac{\kappa_2}{4}\, \|\delta \bv\|_{\bV}^2 \leq \fracd{1}{2\kappa_2}\, \|\sF_1-\sF_2\|_{\bV'}^2 + \frac{\|\cC_1\|^2}{\kappa_2}\, \|\delta \bu\|_{\bH^1(\Omega)}^2  \, .
\end{gathered}
\end{equation}
We now integrate from $0$ to $t$ in~\eqref{eq:difference-of-velocities-proof-2} using that $\delta \bv(0)=\cero$, and use the definition of the Bochner norms to get
\begin{equation}\label{eq:difference-of-velocities-proof-3}
\|\delta \bv\|_{\LL^\infty(0,t;\bL^2(\Omega))}^2 + \|\delta \bv\|_{\LL^2(0,t;\bV)}^2 \leq C\, \left(\|\sF_1-\sF_2\|_{\LL^2(0,t;\bV')}^2 + \|\delta \bu\|_{\LL^2(0,t;\bH^1(\Omega))}^2 \right)\, ,
\end{equation}
with $C>0$ depending only on $\varrho$, $\kappa_2$ and $\|\cC_1\|$. Then, by definition of $\delta \bu$, we readily see that $\|\delta \bu\|_{\LL^2(0,t;\bH^1(\Omega))}^2\leq t^2\, \|\delta \bv\|_{\LL^2(0,t;\bV)}^2$, so by restricting ourselves to $0<t\leq T'$, with $T'\leq T$ such that $C\, (T')^2 \leq \frac{1}{2}$, we find that~\eqref{eq:difference-of-velocities-proof-3} implies
\begin{equation*}
\|\delta \bv\|_{\LL^\infty(0,T';\bL^2(\Omega))}^2 + \frac{1}{2}\, \|\delta \bv\|_{\LL^2(0,T';\bV)}^2 \leq C\, \|\sF_1-\sF_2\|_{\LL^2(0,T';\bV')}^2 \, ,
\end{equation*}
thereby proving~\eqref{eq:difference-of-velocities}. Finally, using once again that $\|\delta \bu\|_{\LL^2(0,T';\bH^1(\Omega))}^2 \leq (T')^2\, \|\delta \bv\|_{\LL^2(0,T';\bV)}^2$ along with~\eqref{eq:difference-of-velocities}, it follows~\eqref{eq:difference-of-displacements}. This completes the proof.
\end{proof}

We next record a Gr\"onwall-type inequality, whose proof can be found in~\cite[Lemma 3.3]{cwy2026}.
\begin{lemma}\label{lem:improved-Gronwall}
Let $\chi,R,A,B:[0,t_{\rm end}]\to [0,+\infty)$ measurable functions such that $A\in \LL^\infty(0,t_{\rm end})$ and $B\chi \in \LL^1(0,t_{\rm end})$. Suppose that, for a.e. $t\in (0,t_{\rm end}]$,
\begin{equation*}
\chi^2(t) + R(t) \leq A(t) + 2\, \int_0^t B(s)\, \chi(s)\ds.
\end{equation*}
Then, for all $t\in (0,t_{\rm end}]$,
\begin{equation*}
\sqrt{\chi^2(t) + R(t)} \leq \sup_{0\leq s\leq t_{\rm end}} \sqrt{A(s)} + \int_0^t B(s)\ds.
\end{equation*}
\end{lemma}

We are now in a position to prove the existence and uniqueness of a solution to~\eqref{eq:reduced-weak} on a short time interval.

\begin{theorem}[Local well-posedness]\label{thm:local}
There exists $T_\star\in(0,t_{\mathrm{end}}]$, depending only on the coercivity constants (cf.~\eqref{eq:coercivity-a}--\eqref{eq:coercivity-D}), Lipschitz constants (cf.~\eqref{eq:alpha-Lipschitz},~\eqref{eq:assumption-f},~\eqref{eq:assumption-g}), $\lambda_1$, $\mu_1$, $\varrho$, $\beta$, and $d$, but independent of the initial data, such that problem~\eqref{eq:reduced-weak} admits a unique solution on $[0,T_\star]$.
\end{theorem}
\begin{proof}
Let $0< T \leq t_{\rm end}$ sufficiently small so that Lemmas~\ref{lem:primal-viscoelasticity-well-posedness} and~\ref{lem:difference-of-velocities} hold. Define $X_{T} := \LL^2(0,T;\LL^2(\Omega)) \times \LL^2(0,T;\LL^2(\Omega))$, which is a Banach space endowed with the product norm
\begin{equation}\label{eq:local-well-posedness-proof-norm}
\| (\wh \varphi, \wh \theta) \|_{X_{T}}^2 := \|\wh \varphi\|_{\LL^2(0,T;\LL^2(\Omega))}^2 + \|\wh \theta\|_{\LL^2(0,T;\LL^2(\Omega))}^2 \, .
\end{equation}
Let $\sS:X_{T} \to X_{T}$ be the operator defined so that for each $(\wh \varphi, \wh \theta) \in X_{T}$, $\sS(\wh \varphi, \wh \theta) := (\varphi,\theta)$, where $(\varphi,\theta)$ is the unique solution (to be confirmed below) to
\begin{equation}\label{eq:local-well-posedness-proof-1}
\begin{aligned}
\langle\dot{\varphi},\psi\rangle + (\diff\nabla\varphi,\nabla\psi)_\Omega  &= (f(\wh\varphi,\wh\theta),\psi)_\Omega + \beta \, (\vdiv (\bu_{\wh \varphi}) ,\psi)_\Omega  \qquad \forall\, \psi \in \rQ  \, , \\[1.1ex]
(\dot{\theta},\vartheta)_\Omega &= (g(\wh\varphi,\wh\theta),\vartheta)_\Omega \qquad \forall\, \vartheta \in \LL^2(\Omega) \, ,
\end{aligned}
\end{equation}
for a.e. $t\in (0,T]$, $\varphi(0)=\varphi_0$ and $\theta(0)=\theta_0$. Here, we set $\bu_{\wh \varphi}(t) := \bu_0 + \int_0^t \bv_{\wh \varphi}(s)\ds$, where $\bv_{\wh \varphi}\in \LL^2(0,T;\bV)$, with $\dot \bv_{\wh \varphi} \in \LL^2(0,T;\bV')$, is the unique solution to
\begin{equation}\label{eq:local-well-posedness-proof-1.5}
\varrho\,\langle\dot{\bv}_{\wh \varphi},\bw\rangle + \wt a_2(\bv_{\wh \varphi},\bw) + \int_0^t \wt a_1(\bv_{\wh \varphi}(s),\bw)\ds = \langle\sF_{\wh \varphi}(t),\bw\rangle \qquad\forall\,\bw\in\bV \, ,\ \text{a.e. }t\in(0,T] \, ,
\end{equation}
with $\bv_{\wh \varphi}(0) = \bv_0$. Note that this solution is guaranteed by Lemma~\ref{lem:primal-viscoelasticity-well-posedness} with the right-hand side functional $\sF_{\wh \varphi}\in \LL^2(0,T;\bV')$ given by
\begin{equation*}
\langle \sF_{\wh \varphi}(t), \bw \rangle := (\varrho\,\bb,\bw)_\Omega - \wt a_1(\bu_0,\bw) + (\alpha(\wh \varphi),\vdiv\, \bw)_\Omega \, .
\end{equation*}
Notice that the regularity required by $\sF_{\wh \varphi}$ is satisfied owing to the fact that $\bb\in \LL^2(0,T;\bL^2(\Omega))$, $\wt a_1$ is a bounded bilinear form, and $\alpha$ is bounded.
We also point out that~\eqref{eq:local-well-posedness-proof-1} is obtained after freezing the coupling terms in~\eqref{eq:reduced-weak}.

A fixed point of $\sS$ yields a solution of~\eqref{eq:reduced-weak}: if $(\varphi,\theta)=\sS(\varphi,\theta)$, then setting $\bv:=\bv_{\varphi}$ as in~\eqref{eq:local-well-posedness-proof-1.5}, the triple $(\bv,\varphi,\theta)$ satisfies~\eqref{eq:rw-v}--\eqref{eq:rw-theta}. Conversely, every solution of~\eqref{eq:reduced-weak} arises in this way: if $(\wt\bv,\wt\varphi,\wt\theta)$ solves~\eqref{eq:reduced-weak}, then~\eqref{eq:rw-v} shows that $\wt\bv$ solves~\eqref{eq:local-well-posedness-proof-1.5} with load $\sF_{\wt\varphi}$, whence $\wt\bv=\bv_{\wt\varphi}$ by the uniqueness in Lemma~\ref{lem:primal-viscoelasticity-well-posedness}; substituting into~\eqref{eq:rw-phi}--\eqref{eq:rw-theta} gives $(\wt\varphi,\wt\theta)=\sS(\wt\varphi,\wt\theta)$. Thus solutions of~\eqref{eq:reduced-weak} and fixed points of $\sS$ are in one-to-one correspondence, and uniqueness of the latter yields uniqueness of the former.

\smallskip
\noindent
\textbf{The fixed-point operator is well-defined.} To prove that~\eqref{eq:local-well-posedness-proof-1} is indeed well-posed and, then, $\sS$ is well-defined, we first observe that the right-hand side functionals belong to $\LL^2(0,T;\rQ')$ and $\LL^2(0,T;\LL^2(\Omega))$, respectively. Furthermore, the diffusion term in the first equation in~\eqref{eq:local-well-posedness-proof-1} is coercive (cf.~\eqref{eq:coercivity-D}). Hence, by standard arguments for linear second-order parabolic equations, we deduce that~\eqref{eq:local-well-posedness-proof-1} is well-posed with $\varphi \in \LL^2(0,T;\rQ)$, $\dot{\varphi}\in \LL^2(0,T;\rQ')$, $\theta \in \HH^1(0,T;\LL^2(\Omega))$, and such that $\varphi(0)=\varphi_0$ and $\theta(0)=\theta_0$. In particular, $(\varphi,\theta) \in X_{T}$, which implies that $\sS$ is well-defined.

\smallskip
\noindent
\textbf{The fixed-point operator is Lipschitz continuous.} For each $i\in \{1,2\}$, let $(\wh \varphi_i, \wh \theta_i) \in X_{T}$ and set $\sS(\wh \varphi_i, \wh \theta_i) = (\varphi_i,\theta_i)$. Define $(\delta \varphi,\delta \theta) := (\varphi_1-\varphi_2,\theta_1-\theta_2)$ and note that, by subtracting the corresponding systems for each input, it satisfies
\begin{equation}\label{eq:local-well-posedness-proof-2}
\begin{aligned}
\langle\delta\dot{\varphi},\psi\rangle + (\diff\nabla\delta\varphi,\nabla\psi)_\Omega  &= (f(\wh\varphi_1,\wh\theta_1)-f(\wh\varphi_2,\wh\theta_2),\psi)_\Omega + \beta \, (\vdiv (\bu_{\wh \varphi_1}-\bu_{\wh \varphi_2}) ,\psi)_\Omega \qquad \forall\, \psi \in \rQ  \, , \\[1.1ex]
(\delta\dot{\theta},\vartheta)_\Omega &= (g(\wh\varphi_1,\wh\theta_1)-g(\wh\varphi_2,\wh\theta_2),\vartheta)_\Omega \qquad \forall\, \vartheta \in \LL^2(\Omega) \, ,
\end{aligned}
\end{equation}
and, since both $(\varphi_i,\theta_i)$ satisfy the same initial condition, it follows that $\delta \varphi(0)= 0$ and $\delta \theta(0)=0$. 

\smallskip
\noindent
\textbf{Lipschitz continuity of the calcium concentration part.} We test the first equation in~\eqref{eq:local-well-posedness-proof-2} with $\psi = \delta\varphi$, use the coercivity of $\diff$ (cf.~\eqref{eq:coercivity-D}), the Cauchy--Schwarz inequality and the Lipschitz continuity of $f$ (cf.~\eqref{eq:assumption-f}) to obtain that
\begin{equation}\label{eq:local-well-posedness-proof-3}
\partial_t \|\delta \varphi\|_{\LL^2(\Omega)}^2 + 2\, \kappa_{D} \|\delta \varphi\|_{\rQ}^2 \leq C_1 \, \left( \|\delta \wh \varphi\|_{\LL^2(\Omega)} + \|\delta \wh \theta\|_{\LL^2(\Omega)} + \|\bu_{\wh \varphi_1}-\bu_{\wh \varphi_2}\|_{\bH^1(\Omega)} \right)\, \|\delta \varphi\|_{\LL^2(\Omega)} \, ,
\end{equation}
with $C_1>0$ depending only on $L_{f1}$, $L_{f2}$, $\beta$, and $d$. Using that $\|\delta \varphi\|_{\LL^2(\Omega)} \leq \|\delta \varphi\|_{\rQ}$, we apply Young's inequality on the right-hand side of~\eqref{eq:local-well-posedness-proof-3} and integrate from $0$ to $t$ using that $\delta \varphi(0) = 0$ to deduce that
\begin{equation}\label{eq:local-well-posedness-proof-4}
\begin{aligned}
&\|\delta \varphi\|_{\LL^\infty(0,T;\LL^2(\Omega))}^2 + \|\delta \varphi\|_{\LL^2(0,T;\rQ)}^2 \\
&\quad \leq C_2 \, \left( \|\delta \wh \varphi\|_{\LL^2(0,T;\LL^2(\Omega))}^2 + \|\delta \wh \theta\|_{\LL^2(0,T;\LL^2(\Omega))}^2 + \|\bu_{\wh \varphi_1}-\bu_{\wh \varphi_2}\|_{\LL^2(0,T;\bH^1(\Omega))}^2 \right) \, ,
\end{aligned}
\end{equation}
with $C_2>0$ depending on $C_1$ and $\kappa_{D}$. To bound the last term in~\eqref{eq:local-well-posedness-proof-4}, we use Lemma~\ref{lem:difference-of-velocities} with right-hand side functionals given by $\sF_{\wh \varphi_1}$ and $\sF_{\wh \varphi_2}$, the definition of each $\sF_{\wh \varphi_i}$, the Cauchy--Schwarz inequality, and the Lipschitz continuity of $\alpha$ (cf.~\eqref{eq:alpha-Lipschitz}), thereby arriving at
\begin{equation*}
\|\bu_{\wh \varphi_1}-\bu_{\wh \varphi_2}\|_{\LL^2(0,T;\bH^1(\Omega))}^2 \leq C\, T^2\, \|\sF_{\wh \varphi_1} - \sF_{\wh \varphi_2}\|_{\LL^2(0,T;\bV')}^2 \leq C\, d\, L_{\alpha}^2 T^2 \, \,\|\wh \varphi_1-\wh \varphi_2\|_{\LL^2(0,T;\LL^2(\Omega))}^2\, .
\end{equation*}
Substituting this into~\eqref{eq:local-well-posedness-proof-4} yields
\begin{equation*}
\|\delta \varphi\|_{\LL^\infty(0,T;\LL^2(\Omega))}^2 + \|\delta \varphi\|_{\LL^2(0,T;\rQ)}^2 \leq C_2 \, \left( (1+C\, d\, L_{\alpha}^2\, T^2) \|\delta \wh \varphi\|_{\LL^2(0,T;\LL^2(\Omega))}^2 + \|\delta \wh \theta\|_{\LL^2(0,T;\LL^2(\Omega))}^2 \right) \, ,
\end{equation*}
which, combined with the fact that $\|\delta \varphi \|_{\LL^2(0,T;\LL^2(\Omega))}^2 \leq T\, \|\delta \varphi \|_{\LL^\infty(0,T;\LL^2(\Omega))}^2$, implies
\begin{equation}\label{eq:local-well-posedness-proof-7}
\|\delta \varphi \|_{\LL^2(0,T;\LL^2(\Omega))}^2 \leq  C_3\, (T^3 + T) \, \left( \|\delta \wh \varphi\|_{\LL^2(0,T;\LL^2(\Omega))}^2 + \|\delta \wh \theta\|_{\LL^2(0,T;\LL^2(\Omega))}^2  \right) \, ,
\end{equation}
with $C_3>0$ depending only on $C_2$ and $C$.

\smallskip
\noindent
\textbf{Lipschitz continuity of the iIT part.}
We now turn our attention to the second equation in~\eqref{eq:local-well-posedness-proof-2}. Testing with $\vartheta = \delta \theta$, we then use the Cauchy--Schwarz inequality, and the Lipschitz continuity of $g$ (cf.~\eqref{eq:assumption-g}) to arrive at
\begin{align*}
\partial_t \|\delta \theta\|_{\LL^2(\Omega)}^2 \leq 2 \left( L_{g1} \, \|\delta \wh \varphi\|_{\LL^2(\Omega)} + L_{g2} \, \|\delta \wh \theta\|_{\LL^2(\Omega)} \right)\, \|\delta \theta\|_{\LL^2(\Omega)}  \, ,
\end{align*}
which upon integrating from $0$ to $t$ and using that $\delta \theta(0) = 0$, yields
\begin{equation*}
\|\delta \theta(t)\|_{\LL^2(\Omega)}^2 \leq 2\, \int_0^t \left( L_{g1} \, \|\delta \wh \varphi(s)\|_{\LL^2(\Omega)} + L_{g2} \, \|\delta \wh \theta(s)\|_{\LL^2(\Omega)} \right)\, \|\delta \theta(s)\|_{\LL^2(\Omega)}\ds \, .
\end{equation*}
Then, we use Lemma~\ref{lem:improved-Gronwall} to get
\begin{align*}
&\|\delta \theta(t)\|_{\LL^2(\Omega)} \leq \int_0^t \left( L_{g1} \, \|\delta \wh \varphi(s)\|_{\LL^2(\Omega)} + L_{g2} \, \|\delta \wh \theta(s)\|_{\LL^2(\Omega)} \right)  \ds \\
&\quad \leq \max\{L_{g1},L_{g2}\}\, \sqrt{t}\, \left( \|\delta \wh \varphi\|_{\LL^2(0,t;\LL^2(\Omega))} + \|\delta \wh \theta\|_{\LL^2(0,t;\LL^2(\Omega))} \right) \, ,
\end{align*}
and, by squaring, using Young's inequality, bearing in mind the definition of the Bochner norms, and that $\|\bullet\|_{\LL^2(0,T;\LL^2(\Omega))} \leq \sqrt{T}\, \|\bullet\|_{\LL^\infty(0,T;\LL^2(\Omega))}$, we obtain
\begin{equation}\label{eq:local-well-posedness-proof-8}
\|\delta \theta\|_{\LL^2(0,T;\LL^2(\Omega))}^2 \leq T\, \|\delta \theta\|_{\LL^\infty(0,T;\LL^2(\Omega))}^2 \leq C_4\, T^2\, \left( \|\delta \wh \varphi\|_{\LL^2(0,T;\LL^2(\Omega))}^2 + \|\delta \wh \theta\|_{\LL^2(0,T;\LL^2(\Omega))}^2  \right)\, ,
\end{equation} 
with $C_4>0$ depending only on $L_{g1}$ and $L_{g2}$.

\smallskip
\noindent
\textbf{Contractivity of the Lipschitz continuous operator.}
By summing~\eqref{eq:local-well-posedness-proof-7} and~\eqref{eq:local-well-posedness-proof-8}, and recalling the definition of the norm~\eqref{eq:local-well-posedness-proof-norm}, we deduce that
\begin{equation}\label{eq:local-well-posedness-proof-9}
\|(\delta\varphi,\delta\theta)\|_{X_{T}}^2 \leq C_5\, (T^3 + T^2 + T)\, \|(\delta \wh \varphi,\delta \wh \theta)\|_{X_{T}}^2\, .
\end{equation}
Then, by taking $T$ so small that $C_5\, (T^3 + T^2 + T) \leq 1/4$, we deduce from~\eqref{eq:local-well-posedness-proof-9} that
\begin{equation*}
\|\sS(\wh\varphi_1,\wh \theta_1)-\sS(\wh\varphi_2,\wh \theta_2)\|_{X_{T}} \leq \fracd{1}{2}\, \|(\wh \varphi_1-\wh \varphi_2 , \wh \theta_1-\wh \theta_2)\|_{X_{T}}\, ,
\end{equation*}
which shows that $\sS$ is a contraction map defined on the Banach space $X_T$. Therefore, the Banach fixed-point theorem guarantees the existence of a unique fixed point. Consequently, \eqref{eq:reduced-weak} admits a unique solution on the time interval $[0,T]$.
\end{proof}

\begin{theorem}[Global existence and uniqueness]\label{thm:global}
There exists a unique solution $(\bv,\varphi,\theta)$ satisfying~\eqref{eq:reduced-weak} on the whole interval $(0,t_{\rm end}]$, and such that $\bv \in \LL^2(0,t_{\rm end};\bV)$, $\varphi \in \LL^2(0,t_{\rm end};\rQ)$, $\theta \in \HH^1(0,t_{\rm end};\LL^2(\Omega))$, with $\dot{\bv} \in \LL^2(0,t_{\rm end};\bV')$ and $\dot{\varphi} \in \LL^2(0,t_{\rm end};\rQ')$, $\bv(0)=\bv_0$, $\varphi(0)=\varphi_0$, and $\theta(0)=\theta_0$.
\end{theorem}
\begin{proof}
Since Theorem~\ref{thm:local} yields a unique solution on $[0,T_\star]$ and the existence time $T_\star$ is independent of the initial data, the local argument can be iterated. Indeed, the endpoint values $\bv(T_\star)\in\bL^2(\Omega)$, $\varphi(T_\star),\theta(T_\star)\in\LL^2(\Omega)$, and $\bu(T_\star)\in\bH^1(\Omega)$ satisfy the hypotheses of Theorem~\ref{thm:local}. Moreover, for $t\geq T_\star$ the viscoelastic memory splits as $\int_0^t \wt a_1(\bv(s),\bw)\ds = \wt a_1(\bu(T_\star)-\bu_0,\bw) + \int_{T_\star}^t \wt a_1(\bv(s),\bw)\ds$, so that~\eqref{eq:reduced-weak} on $[T_\star,2T_\star]$ takes again the form analysed in Theorem~\ref{thm:local}, now with offset $\bu(T_\star)$ and the memory measured from $T_\star$. Therefore, they may be taken as new initial data and offset to construct a unique solution on $[T_\star,2T_\star]$. Repeating this procedure finitely many times, and noting that the length of the existence interval remains unchanged at each step, one covers the entire interval $(0,t_{\rm end}]$ in at most $\lceil t_{\rm end}/T_\star\rceil$ iterations. Since each piece is continuous in time at the junctions, the concatenation belongs to the spaces asserted above. This establishes existence of a global solution. To prove uniqueness, we let $(\bv_i,\varphi_i,\theta_i)$, $i\in\{1,2\}$, be two solutions on $(0,t_{\rm end}]$. By Theorem~\ref{thm:local} they coincide on $[0,T_\star]$; in particular their values at $T_\star$ agree, so, again by Theorem~\ref{thm:local}, they coincide on $[T_\star,2T_\star]$, and inductively on each $[k\, T_\star,(k+1)\, T_\star]$. Hence they agree on $(0,t_{\rm end}]$. This completes the proof.
\end{proof}

\begin{theorem}[Enhanced regularity]\label{thm:enhanced-regularity}
Let $(\bv,\varphi,\theta)$ be the solution to~\eqref{eq:reduced-weak}, guaranteed by Theorem~\ref{thm:global}. In addition to Assumption~\ref{assum:data-regularity}, suppose that $\bu_0,\bv_0\in \bV$ and $\varphi_0 \in \rQ$. Then,
\begin{equation*}
\dot{\bv} \in \LL^2(0,t_{\rm end};\bL^2(\Omega)) \qan \dot{\varphi} \in \LL^2(0,t_{\rm end};\LL^2(\Omega))\, .
\end{equation*}
\end{theorem}
\begin{proof}
The a priori estimates below are formal in that they require testing with $\dot{\varphi}$ and $\dot{\bv}$, which a priori belong only to $\rQ'$ and $\bV'$; they are made rigorous by performing the computation on the Galerkin approximations (or a difference-quotient argument in time) underlying the parabolic theory invoked in Theorem~\ref{thm:local} and passing to the limit by weak lower semicontinuity.

\smallskip
\noindent\textbf{Enhanced regularity of the velocity.}
Since $\bu(t)=\bu_0+\int_0^t\bv(s)\ds$, the memory term and the offset combine as $\int_0^t \wt a_1(\bv(s),\bw)\ds + \wt a_1(\bu_0,\bw) = \wt a_1(\bu(t),\bw)$, so that~\eqref{eq:rw-v} reads
\begin{equation*}
\varrho\,\langle\dot{\bv},\bw\rangle + \wt a_2(\bv,\bw) + \wt a_1(\bu(t),\bw) = (\varrho\,\bb,\bw)_\Omega + (\alpha(\varphi),\vdiv\,\bw)_\Omega \qquad\forall\,\bw\in\bV\, .
\end{equation*}
Testing against $\bw=\dot{\bv}$ gives
\begin{equation}\label{eq:enhanced-regularity-proof-1}
\varrho\,\|\dot{\bv}\|_{\bL^2(\Omega)}^2 + \fracd{1}{2}\, \partial_t \wt a_2(\bv,\bv) + \wt a_1(\bu,\dot{\bv}) = (\varrho\,\bb,\dot{\bv})_\Omega + (\alpha(\varphi),\vdiv\,\dot{\bv})_\Omega\, .
\end{equation}
For the active term, integration by parts in space gives $(\alpha(\varphi),\vdiv\,\dot{\bv})_\Omega = -(\nabla\alpha(\varphi),\dot{\bv})_\Omega$, where the boundary contribution vanishes because $\dot{\bv}=\cero$ on $\Gamma^{\bu}$ and $\alpha(\varphi)=\alpha(0)=0$ on $\Gamma^{\bsigma}$, and $\nabla\alpha(\varphi)=\alpha'(\varphi)\nabla\varphi$ satisfies $\|\nabla\alpha(\varphi)\|_{\bL^2(\Omega)}\leq L_\alpha\,\|\nabla\varphi\|_{\bL^2(\Omega)}$. Hence, by the Cauchy--Schwarz and Young inequalities, the right-hand side of~\eqref{eq:enhanced-regularity-proof-1} is bounded as
\begin{equation*}
(\varrho\,\bb,\dot{\bv})_\Omega + (\alpha(\varphi),\vdiv\,\dot{\bv})_\Omega \leq \fracd{\varrho}{2}\,\|\dot{\bv}\|_{\bL^2(\Omega)}^2 + \varrho\,\|\bb\|_{\bL^2(\Omega)}^2 + \fracd{L_\alpha^2}{\varrho}\,\|\nabla\varphi\|_{\bL^2(\Omega)}^2\, .
\end{equation*}
Collecting the above, integrating from $0$ to $\tau\in(0,t_{\rm end}]$ in~\eqref{eq:enhanced-regularity-proof-1}, using $\bv(0)=\bv_0$ and $\bu(0)=\bu_0$, and using that $\wt a_1(\bu,\dot{\bv}) = \partial_t \wt a_1(\bu,\bv) - \wt a_1(\bv,\bv)$ since $\dot{\bu}=\bv$, we arrive at
\begin{equation}\label{eq:enhanced-regularity-proof-2}
\begin{aligned}
&\fracd{\varrho}{2}\int_0^\tau\!\|\dot{\bv}\|_{\bL^2(\Omega)}^2\ds + \fracd{1}{2}\,\wt a_2(\bv(\tau),\bv(\tau)) + \wt a_1(\bu(\tau),\bv(\tau))\\
&\quad \leq \fracd{1}{2}\,\wt a_2(\bv_0,\bv_0) + \wt a_1(\bu_0,\bv_0) + \int_0^\tau\!\wt a_1(\bv,\bv)\ds + \varrho\int_0^\tau\|\bb\|_{\bL^2(\Omega)}^2\ds + \fracd{L_\alpha^2}{\varrho}\int_0^\tau\|\nabla\varphi\|_{\bL^2(\Omega)}^2\ds \, .
\end{aligned}
\end{equation}
The cross term on the left-hand side is controlled by the Cauchy--Schwarz inequality for the symmetric positive semidefinite form $\wt a_1$, and Young's inequality (with some $\epsilon>0$), to obtain
\begin{equation*}
|\wt a_1(\bu(\tau),\bv(\tau))| \leq \wt a_1(\bu(\tau),\bu(\tau))^{1/2}\, \wt a_1(\bv(\tau),\bv(\tau))^{1/2} \leq \fracd{1}{2\epsilon}\,\wt a_1(\bu(\tau),\bu(\tau)) + \fracd{\epsilon}{2}\,\wt a_1(\bv(\tau),\bv(\tau)) \, .
\end{equation*}
We then use the comparison
\begin{equation*}
\wt a_1(\bw,\bw)\leq \frac{2\mu_1+d\lambda_1}{2\mu_2}\,\wt a_2(\bw,\bw) \, ,
\end{equation*}
which follows from $2\mu_2\,\|\beps(\bw)\|_{\bbL^2(\Omega)}^2\leq \wt a_2(\bw,\bw)$ and $\wt a_1(\bw,\bw)\leq (2\mu_1+d\lambda_1)\|\beps(\bw)\|_{\bbL^2(\Omega)}^2$, and choosing $\epsilon=\tfrac{\mu_2}{2\mu_1+d\lambda_1}$, yields
\begin{equation*}
|\wt a_1(\bu(\tau),\bv(\tau))| \leq \fracd{2\mu_1+d\lambda_1}{2\mu_2}\,\wt a_1(\bu(\tau),\bu(\tau)) + \fracd{1}{4}\,\wt a_2(\bv(\tau),\bv(\tau))\, .
\end{equation*}
Inserting this bound into~\eqref{eq:enhanced-regularity-proof-2} and then neglecting the term $\wt a_2(\bv(\tau),\bv(\tau))\geq 0$,
\begin{equation}\label{eq:enhanced-regularity-proof-3}
\int_0^\tau\!\|\dot{\bv}\|_{\bL^2(\Omega)}^2\ds \leq C\, \left( \Theta(\tau) + \int_0^\tau\!\wt a_1(\bv,\bv)\ds + \wt a_1(\bu(\tau),\bu(\tau)) + \int_0^\tau\|\nabla\varphi\|_{\bL^2(\Omega)}^2\ds \right) \, ,
\end{equation}
with $C>0$ depending only on $\varrho$, $\mu_1$, $\lambda_1$, $\mu_2$, $L_{\alpha}$, and $d$, where
\begin{equation*}
\Theta(\tau) := \wt a_2(\bv_0,\bv_0) + \wt a_1(\bu_0,\bv_0)  + \int_0^\tau\|\bb\|_{\bL^2(\Omega)}^2\ds  \, .
\end{equation*}
Each term on the right-hand side of the above estimate is finite, uniformly in $\tau\in(0,t_{\rm end}]$. In fact, $\Theta(\tau)$ is finite as $\bu_0,\bv_0 \in \bV$ and $\bb \in \LL^2(0,t_{\rm end};\bL^2(\Omega))$. Additionally,
\begin{gather*}
\int_0^\tau \wt a_1(\bv,\bv)\ds\leq\|\cC_1\|\,\|\bv\|_{\LL^2(0,t_{\rm end};\bV)}^2 \, , \quad \wt a_1(\bu(\tau),\bu(\tau))\leq\|\cC_1\|\,\|\bu(\tau)\|_{\bV}^2 \, ,\\
\qan \|\bu(\tau)\|_{\bV}\leq\|\bu_0\|_{\bV}+\sqrt{t_{\rm end}}\,\|\bv\|_{\LL^2(0,t_{\rm end};\bV)} \, ,
\end{gather*}
and $\varphi \in \LL^2(0,t_{\rm end};\rQ)$. Therefore, taking $\tau=t_{\rm end}$ in~\eqref{eq:enhanced-regularity-proof-3}, we obtain that $\|\dot{\bv}\|_{\LL^2(0,t_{\rm end};\bL^2(\Omega))}^2 \leq \wt C < +\infty$, which proves that $\dot{\bv}\in\LL^2(0,t_{\rm end};\bL^2(\Omega))$.

\smallskip
\noindent\textbf{Enhanced regularity of the calcium concentration.}
Recalling that $\vdiv\,\bu = \vdiv\,\bu_0 + \int_0^t\vdiv\,\bv\ds$, equation~\eqref{eq:rw-phi} tested against $\psi=\dot{\varphi}$ yields
\begin{equation}\label{eq:enhanced-regularity-proof-4}
\|\dot{\varphi}\|_{\LL^2(\Omega)}^2 + \fracd{1}{2}\, \partial_t (\diff \nabla \varphi,\nabla\varphi)_\Omega = (f(\varphi,\theta),\dot\varphi)_\Omega + \beta\, (\vdiv(\bu),\dot\varphi)_\Omega\, .
\end{equation}
Bounding the right-hand side by the Cauchy--Schwarz and Young inequalities,
\begin{equation*}
(f(\varphi,\theta),\dot\varphi)_\Omega + \beta\, (\vdiv\,\bu,\dot\varphi)_\Omega \leq \fracd{1}{2}\, \|\dot{\varphi}\|_{\LL^2(\Omega)}^2 + \|f(\varphi,\theta)\|_{\LL^2(\Omega)}^2 + \beta^2\, \|\vdiv\,\bu\|_{\LL^2(\Omega)}^2\, ,
\end{equation*}
so that, integrating~\eqref{eq:enhanced-regularity-proof-4} from $0$ to $t_{\rm end}$, using that $(\diff\nabla\varphi,\nabla\varphi)_\Omega\geq 0$ and $\varphi(0)=\varphi_0$, we obtain
\begin{equation*}
\fracd{1}{2}\, \|\dot{\varphi}\|_{\LL^2(0,t_{\rm end};\LL^2(\Omega))}^2 \leq \fracd{1}{2}\,(\diff\nabla\varphi_0,\nabla\varphi_0)_\Omega + \|f(\varphi,\theta)\|_{\LL^2(0,t_{\rm end};\LL^2(\Omega))}^2 + \beta^2\, \|\vdiv\,\bu\|_{\LL^2(0,t_{\rm end};\LL^2(\Omega))}^2\, .
\end{equation*}
The right-hand side is finite. Indeed, $(\diff\nabla\varphi_0,\nabla\varphi_0)_\Omega\leq \overline{D}\,\|\nabla\varphi_0\|_{\bL^2(\Omega)}^2<\infty$ since $\varphi_0\in\rQ$. In turn, the Lipschitz continuity of $f$ (cf.~\eqref{eq:assumption-f}) together with $\varphi,\theta\in\LL^2(0,t_{\rm end};\LL^2(\Omega))$ gives $f(\varphi,\theta)\in\LL^2(0,t_{\rm end};\LL^2(\Omega))$. In addition, $\bu\in\LL^\infty(0,t_{\rm end};\bH^1(\Omega))$ yields $\vdiv\,\bu\in\LL^2(0,t_{\rm end};\LL^2(\Omega))$. Hence $\dot{\varphi}\in\LL^2(0,t_{\rm end};\LL^2(\Omega))$. This completes the proof.
\end{proof}

\subsection{Well-posedness of the mixed system}

Let us now reconstruct, from the solutions to the integro-differential equations~\eqref{eq:reduced-weak}, the mixed variables and prove that they satisfy~\eqref{eq:mixed-weak}. Throughout this section, we assume the additional regularity on the data required by Theorem~\ref{thm:enhanced-regularity}, namely, $\bu_0,\bv_0 \in \bV$ and $\varphi_0 \in \rQ$. We denote by $(\bv,\varphi,\theta)$ the solution of~\eqref{eq:reduced-weak} guaranteed by Theorem~\ref{thm:global}. Furthermore, as before, we set
\begin{equation}\label{eq:reconstructed-displacement}
\bu(t) := \bu_0 + \int_0^t \bv(s)\ds \, .
\end{equation}
Since $\bv$ is continuous (see Remark~\ref{rem:lions-magenes}), it follows that $\bu$ is as well. More precisely, from the regularity of $\bv$ (cf. Theorem~\ref{thm:global}), we have that $\bu \in \LL^2(0,t_{\rm end};\bV)$ and $\dot{\bu} = \bv$, so $\dot{\bu} \in \LL^2(0,t_{\rm end};\bV)$.
The fact that $\bu(t)\in \bV$ means two things: namely, that $\bu(t)\in \bH^1(\Omega)$ and that $\bu(t)$ vanishes on $\Gamma^\bu$. Both follow from $\bu_0\in \bV$ and $\bv(s)\in \bV$.

Then, define the mixed variables as in their strong form, but now understanding that they are taken in the sense of distributions:
\begin{equation}\label{eq:reconstructed-mixed-variables}
\bsigma_1 := \cC_1\, \beps(\bu)\, , \quad \bsigma_2 := \cC_2\, \beps(\bv) - \alpha(\varphi)\, \bbI \, , \quad \bomega := \fracd{1}{2}\, (\bnabla \bv - (\bnabla \bv)^\top)\, , \qan \bz := \diff \, \nabla \varphi \, .
\end{equation}
The regularity of the primal variables ensured by Theorem~\ref{thm:global} guarantees that each mixed variable is well defined for a.e. $t\in(0,t_{\rm end}]$.

The following result asserts that these mixed variables, together with the primal ones, solve~\eqref{eq:mixed-weak}. This gives us also the spatial regularity that problem~\eqref{eq:mixed-weak} requires.

\begin{theorem}[Existence of a solution]\label{thm:mixed-reconstructions}
Assume the hypotheses of Theorem~\ref{thm:enhanced-regularity}. Let $(\bv,\varphi,\theta)$ be the solution to~\eqref{eq:reduced-weak} given by Theorem~\ref{thm:global}, $\bu$ defined as in~\eqref{eq:reconstructed-displacement}, and $(\bsigma_1,\bsigma_2,\bomega,\bz)$ as in~\eqref{eq:reconstructed-mixed-variables}. Then, for a.e. $t\in (0,t_{\rm end}]$, 
\begin{equation*}
\bsigma_1(t),\, \bsigma_2(t) \in \bbL_{\mathrm{sym}}^2(\Omega) \, , \quad (\bsigma_1+\bsigma_2)(t) \in \bbH_\star(\bdiv;\Omega)\, , \quad \bomega(t) \in \bbL^2_{\sk}(\Omega)\, , \quad \bz(t)\in \bH_\star(\vdiv;\Omega)\, .
\end{equation*}
Moreover, $(\bsigma_1,\bsigma_2,\bv,\bomega,\bz,\varphi,\theta)$ is a solution to~\eqref{eq:mixed-weak}.
\end{theorem}
\begin{proof}
From the definitions of the mixed variables (cf.~\eqref{eq:reconstructed-mixed-variables}) and the regularity obtained from Theorem~\ref{thm:global}, we immediately see that $\bsigma_1(t),\bsigma_2(t) \in \bbL_{\mathrm{sym}}^2(\Omega)$, $\bomega(t) \in \bbL_{\sk}^2(\Omega)$, and $\bz(t) \in \bL^2(\Omega)$.

We now establish the remaining regularity properties and verify that each equation in~\eqref{eq:mixed-weak} is satisfied. We begin by using that $\cA_i=\cC_i^{-1}$ to obtain
\begin{equation}\label{eq:mixed-reconstructions-proof-1}
\cA_1\dot{\bsigma}_1 = \beps(\bv) \qan \cA_2\bsigma_2 = \beps(\bv) - \cA_2(\alpha(\varphi)\,\bbI) = \beps(\bv) - \frac{\alpha(\varphi)}{2\mu_2+d\lambda_2}\, \bbI\, .
\end{equation}
The first expression in~\eqref{eq:mixed-reconstructions-proof-1} combined with the fact that $\bomega = \bnabla \bv - \beps(\bv)$, yields
\begin{gather*}
(\cA_1\dot{\bsigma}_1,\btau_1)_{\Omega} + (\bdiv\, \btau_1,\bv)_{\Omega} + (\btau_1,\bomega)_{\Omega} = (\beps(\bv),\btau_1)_\Omega + (\bdiv (\btau_1),\bv)_\Omega + (\btau_1,\bnabla \bv - \beps(\bv))_\Omega\\
=(\bdiv(\btau_1),\bv)_\Omega + (\btau_1,\bnabla \bv)_\Omega = 0\, ,
\end{gather*}
for all $\btau_1\in \bbH_\star(\bdiv;\Omega)$, where the last equality comes from an application of integration by parts, where the boundary contribution vanishes owing to that $\btau_1\in \bbH_\star(\bdiv;\Omega)$ and $\bv(t) \in \bV$ for a.e. $t\in (0,t_{\rm end}]$. This proves that the first equality in~\eqref{eq:mixed-weak} holds.

Likewise, using the second expression in~\eqref{eq:mixed-reconstructions-proof-1}, integration by parts, and once again the identity of the spin tensor, we can assert that 
\begin{align*}
&(\cA_2\bsigma_2,\btau_2)_{\Omega} + (\bdiv\, \btau_2,\bv)_{\Omega} + (\btau_2,\bomega)_{\Omega} \\
&\quad = (\beps(\bv),\btau_2)_\Omega - \frac{1}{2\mu_2+d\lambda_2}\, (\alpha(\varphi),\tr(\btau_2))_\Omega + (\bdiv(\btau_2),\bv)_\Omega + (\btau_2,\bnabla \bv - \beps(\bv))_\Omega \\
&\quad = - \frac{1}{2\mu_2+d\lambda_2}\, (\alpha(\varphi),\tr(\btau_2))_\Omega \, ,
\end{align*} 
for all $\btau_2\in \bbH_\star(\bdiv;\Omega)$. This gives us the second row in~\eqref{eq:mixed-weak}.

Since $(\bv,\varphi,\theta)$ is a solution of~\eqref{eq:reduced-weak}, it satisfies~\eqref{eq:reduced-strong} in the sense of distributions. Taking this equality and the definitions of the mixed variables (cf.~\eqref{eq:reconstructed-mixed-variables}) and that of the displacement (cf.~\eqref{eq:reconstructed-displacement}), it holds
\begin{equation*}
\bdiv(\bsigma_1+\bsigma_2) = \varrho\, \dot{\bv} - \varrho\, \bb\, .
\end{equation*}
Since $\dot{\bv}\in \LL^2(0,t_{\rm end};\bL^2(\Omega))$ (cf. Theorem~\ref{thm:enhanced-regularity}) and $\bb\in \LL^2(0,t_{\rm end};\bL^2(\Omega))$, we deduce that $\bdiv(\bsigma_1+\bsigma_2)(t) \in \bL^2(\Omega)$ for a.e. $t\in  (0,t_{\rm end}]$. Hence, $\bsigma_1(t)+\bsigma_2(t) \in \bbH(\bdiv;\Omega)$. Moreover, from the natural boundary conditions in~\eqref{eq:reduced-weak}, it follows that $(\bsigma_1+\bsigma_2)\bn = \cero$ on $\Gamma^\bsigma$. Thus, $\bsigma_1(t)+\bsigma_2(t) \in \bbH_\star(\bdiv;\Omega)$. Then, by letting $\bw\in \bL^2(\Omega)$ and $\eeta \in \bbL_{\sk}^2(\Omega)$, and using that $\bsigma_1$ and $\bsigma_2$ are symmetric,
\begin{equation*}
(\varrho\, \dot{\bv},\bw)_\Omega - (\bdiv(\bsigma_1+\bsigma_2),\bw)_\Omega - (\bsigma_1+\bsigma_2,\eeta)_\Omega = (\varrho\, \dot{\bv},\bw)_\Omega - (\varrho\, \dot{\bv} - \varrho\, \bb,\bw)_\Omega = (\varrho\, \bb,\bw)_\Omega \, .
\end{equation*}
This shows that the third row in~\eqref{eq:mixed-weak} holds.

By the definition of $\bz$ (cf.~\eqref{eq:reconstructed-mixed-variables}) and that $\varphi(t)\in \rQ$, an application of integration by parts shows that
\begin{equation*}
(\diff^{-1}\bz,\br)_\Omega + (\vdiv(\br),\varphi)_\Omega = (\nabla\varphi,\br)_\Omega - (\br,\nabla \varphi)_\Omega = 0\, ,
\end{equation*}
for all $\br \in \bH_\star(\vdiv;\Omega)$, thereby proving that the fourth equation in~\eqref{eq:mixed-weak} holds.

On the other hand, by the definition of $\bz$ (cf.~\eqref{eq:reconstructed-mixed-variables}) and $\bu$ (cf.~\eqref{eq:reconstructed-displacement}), equation~\eqref{eq:rw-phi}, which is the weak version of the second row of~\eqref{eq:reduced-strong}, and the definition of $\bsigma_1$ (cf.~\eqref{eq:reconstructed-mixed-variables}) together with the fact that $\cA_1=\cC_1^{-1}$, it follows that
\begin{equation*}
\dot{\varphi} - \vdiv(\bz) = f(\varphi,\theta) + \beta \, \vdiv(\bu) = f(\varphi,\theta) + \frac{\beta}{2\mu_1+d\lambda_1}\, \tr(\bsigma_1) \, ,
\end{equation*}
in the sense of distributions. This combined with the fact that $\dot{\varphi}\in \LL^2(0,t_{\rm end};\LL^2(\Omega))$ (cf. Theorem~\ref{thm:enhanced-regularity}), $f$ is globally Lipschitz continuous, $\varphi(t),\theta(t)\in \LL^2(\Omega)$, and $\bsigma_1(t)\in \bbL^2(\Omega)$, proves that $\bz(t) \in \bH(\vdiv;\Omega)$ and the fifth equation in~\eqref{eq:mixed-weak}. Furthermore, the boundary conditions in~\eqref{eq:reduced-weak} imply that $\bz(t) \in \bH_\star(\vdiv;\Omega)$.

Finally, the last equation in~\eqref{eq:mixed-weak} readily follows from~\eqref{eq:rw-theta}. This completes the proof.
\end{proof}

The following result shows that the solution obtained from the integro-differential system analysed in Sections~\ref{sec:integro-diff-primal} and~\ref{sec:well-posedness-reduced-problem}, together with the mixed variables introduced at the beginning of this section, not only solves~\eqref{eq:mixed-weak} (cf. Theorem~\ref{thm:mixed-reconstructions}), but is also the unique solution to the mixed system.

\begin{theorem}[Stability and uniqueness]\label{thm:mixed-stability}
Assume the hypotheses of Theorem~\ref{thm:mixed-reconstructions}. For $i\in\{1,2\}$, let $(\bsigma_1^i,\bsigma_2^i,\bv^i,\bomega^i,\bz^i,\varphi^i,\theta^i)$ be a solution to~\eqref{eq:mixed-weak} associated with the body force $\bb^i$ and the initial data $\bsigma_1^i(0),\bv_0^i,\varphi_0^i,\theta_0^i$. Then there exists $C>0$, depending only on $\cA_1$, $\cA_2$, $\diff$, $\varrho$, the Lipschitz constants of $\alpha$, $f$, $g$, the Lam\'e parameters, $\beta$, $d$, and $t_{\rm end}$, but independent of the data, such that, writing $\delta(\bullet):=(\bullet)^1-(\bullet)^2$, for a.e. $t\in(0,t_{\rm end}]$
\begin{equation}\label{eq:mixed-stability}
\begin{aligned}
&\|\delta\bsigma_1\|_{\LL^\infty(0,t;\bbL^2(\Omega))}^2 + \|\delta\bv\|_{\LL^\infty(0,t;\bL^2(\Omega))}^2 + \|\delta\varphi\|_{\LL^\infty(0,t;\LL^2(\Omega))}^2 + \|\delta\theta\|_{\LL^\infty(0,t;\LL^2(\Omega))}^2\\
&\qquad + \|\delta\bsigma_2\|_{\LL^2(0,t;\bbL^2(\Omega))}^2 + \|\delta\bz\|_{\LL^2(0,t;\bL^2(\Omega))}^2 \\
&\quad \leq C\,\Big( \|\delta\bsigma_1(0)\|_{\bbL^2(\Omega)}^2 + \|\delta\bv_0\|_{\bL^2(\Omega)}^2 + \|\delta\varphi_0\|_{\LL^2(\Omega)}^2 + \|\delta\theta_0\|_{\LL^2(\Omega)}^2 + \|\delta\bb\|_{\LL^2(0,t;\bL^2(\Omega))}^2 \Big)\, .
\end{aligned}
\end{equation}
In particular, problem~\eqref{eq:mixed-weak} admits at most one solution.
\end{theorem}
\begin{proof}
Subtracting the two systems~\eqref{eq:mixed-weak}, adding the first two resulting equations having as test function the pair $(\btau_1,\btau_2)$, and testing the resulting equations with their own unknowns, namely with $(\btau_1,\btau_2)=(\delta\bsigma_1,\delta\bsigma_2)$,  $(\bw,\eeta)=(\delta\bv,\delta\bomega)$, $\br=\delta\bz$, $\psi=\delta\varphi$, and $\vartheta=\delta\theta$, the coupling contributions cancel pairwise: the terms $(\bdiv\,(\delta\btau_1+\delta \btau_2),\delta\bv)$ and $(\delta\btau_1 + \delta \btau_2,\delta\bomega)$ are removed by $-(\bdiv(\delta\bsigma_1+\delta\bsigma_2),\delta\bv)$ and $-(\delta\bsigma_1+\delta\bsigma_2,\delta\bomega)$, and likewise $(\vdiv\,\delta\bz,\delta\varphi)$ from the fourth cancels with $-(\vdiv\,\delta\bz,\delta\varphi)$ from the fifth. Introducing the energy
\begin{equation*}
\sE(t):=\fracd{1}{2}\,(\cA_1\delta\bsigma_1,\delta\bsigma_1)_\Omega + \fracd{\varrho}{2}\,\|\delta\bv\|_{\bL^2(\Omega)}^2 + \fracd{1}{2}\,\|\delta\varphi\|_{\LL^2(\Omega)}^2 + \fracd{1}{2}\,\|\delta\theta\|_{\LL^2(\Omega)}^2\, ,
\end{equation*}
and using that $\cA_1$ is symmetric and time-independent, the surviving identities add up to
\begin{equation*}
\partial_t\sE + (\cA_2\delta\bsigma_2,\delta\bsigma_2)_\Omega + (\diff^{-1}\delta\bz,\delta\bz)_\Omega = \sR\, ,
\end{equation*}
where, abbreviating $\delta f := f(\varphi^1,\theta^1)-f(\varphi^2,\theta^2)$ and $\delta g := g(\varphi^1,\theta^1)-g(\varphi^2,\theta^2)$,
\begin{align*}
\sR &:= -\fracd{1}{2\mu_2+d\lambda_2}\,(\alpha(\varphi^1)-\alpha(\varphi^2),\tr\delta\bsigma_2)_\Omega + (\varrho\,\delta\bb,\delta\bv)_\Omega + (\delta f,\delta\varphi)_\Omega \\
&\qquad + \fracd{\beta}{2\mu_1+d\lambda_1}\,(\tr\delta\bsigma_1,\delta\varphi)_\Omega + (\delta g,\delta\theta)_\Omega\, .
\end{align*}
We bound $\sR$ term by term using the Cauchy--Schwarz and Young inequalities, the bound $\|\tr\btau\|_{\LL^2(\Omega)}\leq\sqrt{d}\,\|\btau\|_{\bbL^2(\Omega)}$, the Lipschitz continuity of $\alpha$, $f$, $g$ (cf.~\eqref{eq:alpha-Lipschitz},~\eqref{eq:assumption-f},~\eqref{eq:assumption-g}), and the uniform positive definiteness of $\cA_1$, which gives $\|\delta\bsigma_1\|_{\bbL^2(\Omega)}^2\leq \underline{a}_1^{-1}(\cA_1\delta\bsigma_1,\delta\bsigma_1)_\Omega\leq 2\,\underline{a}_1^{-1}\sE$, with $\underline{a}_1>0$ depending only on $\mu_1$ and $\lambda_1$. The first term is absorbed into the dissipation $(\cA_2\delta\bsigma_2,\delta\bsigma_2)_\Omega\geq\underline{a}_2\|\delta\bsigma_2\|_{\bbL^2(\Omega)}^2$, while all remaining contributions are controlled by $\sE$ and the load. Hence
\begin{equation}\label{eq:mixed-stability-proof-2}
\partial_t\sE + \fracd{\underline{a}_2}{2}\,\|\delta\bsigma_2\|_{\bbL^2(\Omega)}^2 + \overline{D}^{-1}\|\delta\bz\|_{\bL^2(\Omega)}^2 \leq C_0\,\sE + \fracd{\varrho}{2}\,\|\delta\bb\|_{\bL^2(\Omega)}^2\, ,
\end{equation}
with $C_0>0$ depending on the same quantities as $C$ but independent of $t_{\rm end}$. Dropping the (nonnegative) dissipation terms and integrating, Gr\"onwall's inequality yields, for a.e. $t\in(0,t_{\rm end}]$,
\begin{equation*}
\sE(t) \leq e^{C_0 t}\,\Big( \sE(0) + \fracd{\varrho}{2}\,\|\delta\bb\|_{\LL^2(0,t;\bL^2(\Omega))}^2 \Big)\, ,
\end{equation*}
which, recalling the definition of $\sE$ and that $\sE(0)$ is bounded by the initial data on the right-hand side of~\eqref{eq:mixed-stability}, controls the four $\LL^\infty$-norms therein. Integrating~\eqref{eq:mixed-stability-proof-2} once more over $(0,t)$ without discarding the dissipative terms, and bounding $\int_0^t\sE(s)\ds$ by the previous estimate, recovers the bounds on $\|\delta\bsigma_2\|_{\LL^2(0,t;\bbL^2(\Omega))}^2$ and $\|\delta\bz\|_{\LL^2(0,t;\bL^2(\Omega))}^2$, thereby proving~\eqref{eq:mixed-stability}.

Finally, if the two solutions share the same data, then the right-hand side of~\eqref{eq:mixed-stability} vanishes, whence $\delta\bsigma_1=\delta\bsigma_2=\delta\bv=\delta\varphi=\delta\theta=\cero$ and $\delta\bz=\cero$. The remaining variable follows from the first row of~\eqref{eq:mixed-weak}: with $\delta\bsigma_1=\cero$ and $\delta\bv=\cero$ it reduces to $(\btau_1,\delta\bomega)_\Omega=0$ for all $\btau_1\in\bbH_\star(\bdiv;\Omega)$, and the density of $\bbH_\star(\bdiv;\Omega)$ in $\bbL^2(\Omega)$ forces $\delta\bomega=\cero$. This proves uniqueness.
\end{proof}

\section{Finite element discretisation}\label{sec:FE}
 
In this section, we construct a conforming finite element discretisation of~\eqref{eq:mixed-weak} and establish its well-posedness in the semidiscrete (continuous-in-time) setting. To this end, we first introduce some preliminary notation. Let $\{\cT_h\}_{h>0}$ be a shape-regular family of meshes partitioning $\Omega$ into triangles or tetrahedra (in 2D and 3D, respectively) of diameter $h_K$, with mesh size $h := \max\{h_K: K\in \cT_h\}$.
 
We next introduce the semidiscrete continuous-in-time formulation associated with~\eqref{eq:mixed-weak}. We define the bilinear forms $\va^i:\bbL^2(\Omega)\times\bbL^2(\Omega)\to\R$, $\vb:\bbH_\star(\bdiv;\Omega)\times[\bL^2(\Omega)\times\bbL^2_{\sk}(\Omega)]\to\R$, $\ca:\bL^2(\Omega)\times\bL^2(\Omega)\to\R$, and $\cb:\bH_\star(\vdiv;\Omega)\times\LL^2(\Omega)\to\R$ as
\begin{gather*}
\va^i(\bzeta,\btau) := (\cA_i\,\bzeta,\btau)_\Omega\, , \qquad
\vb(\btau,[\bw,\eeta]) := (\bdiv\,\btau,\bw)_\Omega + (\btau,\eeta)_\Omega\, , \\[1ex]
\ca(\bz,\br) := (\diff^{-1}\bz,\br)_\Omega\, , \qan
\cb(\br,\psi) := (\vdiv\,\br,\psi)_\Omega\, .
\end{gather*}
Furthermore, we set conforming finite element subspaces
\begin{gather*}
\bbS_h \subseteq \bbH_\star(\bdiv;\Omega)\, , \quad \bV_h \subseteq \bL^2(\Omega)\, , \quad \bbK_h \subseteq \bbL^2_{\mathrm{skew}}(\Omega)\, , \\[1ex]
\bZ_h \subseteq \bH_\star(\vdiv;\Omega)\, , \quad \Phi_h \subseteq \LL^2(\Omega)\, , \qan \Theta_h \subseteq \LL^2(\Omega)\, ,
\end{gather*}
and assume the following hypothesis:

\noindent \textbf{(H1)} There exists a constant $\beta_{\rd}>0$, independent of $h$, such that
\begin{equation*}
\sup_{\cero \neq \btau_h \in \bbS_h}\frac{\vb(\btau_h, [\bw_h,\eeta_h])}{\|\btau_h\|_{\bbH(\bdiv;\Omega)}} \geq \beta_{\rd}\, \| (\bw_h, \eeta_h) \|_{\bL^2(\Omega)\times \bbL^2(\Omega)}  \qquad \forall\, [\bw_h, \eeta_h] \in \bV_h \times \bbK_h \, .
\end{equation*}
 

In this way, the semidiscrete scheme consists of finding
\begin{equation}\label{eq:semidiscrete-unknowns}
\begin{aligned}
&(\bsigma_{1,h},\bsigma_{2,h},\bv_h,\bomega_h,\bz_h,\varphi_h,\theta_h) \\
&\quad \in C^1([0,t_{\rm end}];\bbS_h) \times C^0([0,t_{\rm end}];\bbS_h) \times C^1([0,t_{\rm end}];\bV_h) \times C^0([0,t_{\rm end}];\bbK_h) \\
&\qquad \times C^0([0,t_{\rm end}];\bZ_h) \times C^1([0,t_{\rm end}];\Phi_h) \times C^1([0,t_{\rm end}];\Theta_h)\, ,
\end{aligned}
\end{equation}
satisfying the initial conditions $\bsigma_{1,h}(0)=\bsigma_{1,h}^0$, $\bv_h(0)=\bv_h^0$, $\varphi_h(0)=\varphi_h^0$, $\theta_h(0)=\theta_h^0$, where $\bsigma_{1,h}^0$, $\bv_h^0$, $\varphi_h^0$, and $\theta_h^0$ are suitable approximations of the corresponding initial data (specified below), such that, for all $t\in(0,t_{\rm end}]$,
\begin{equation}\label{eq:semidiscrete}
\begin{array}{rcl}
\va^1(\dot{\bsigma}_{1,h},\btau_{1,h}) + \vb(\btau_{1,h},[\bv_h,\bomega_h]) & = & 0 \, , \\[1.25ex]
\va^2(\bsigma_{2,h},\btau_{2,h}) + \vb(\btau_{2,h},[\bv_h,\bomega_h]) & = & -\frac{1}{2\mu_2+d\lambda_2}\, (\alpha(\varphi_h),\tr(\btau_{2,h}))_\Omega \, , \\[1.25ex]
(\varrho\,\dot{\bv}_h,\bw_h)_\Omega - \vb(\bsigma_{1,h}+\bsigma_{2,h},[\bw_h,\eeta_h]) & = & (\varrho\,\bb,\bw_h)_\Omega \, , \\[1.25ex]
\ca(\bz_h,\br_h) + \cb(\br_h,\varphi_h) & = & 0 \, , \\[1.25ex]
(\dot{\varphi}_h,\psi_h)_\Omega - \cb(\bz_h,\psi_h) & = & \left(f(\varphi_h,\theta_h),\psi_h\right)_\Omega + \tfrac{\beta}{2\mu_1+d\lambda_1}(\tr(\bsigma_{1,h}),\psi_h)_\Omega \, , \\[1.25ex]
(\dot{\theta}_h,\vartheta_h)_\Omega & = & (g(\varphi_h,\theta_h),\vartheta_h)_\Omega \, ,
\end{array}
\end{equation}
for all $(\btau_{1,h},\btau_{2,h},[\bw_h,\eeta_h],\br_h,\psi_h,\vartheta_h)\in \bbS_h\times\bbS_h\times[\bV_h\times\bbK_h]\times\bZ_h\times\Phi_h\times\Theta_h$.
 
The discrete initial data are chosen as projections of the corresponding continuous ones. Since $\va^1$ is an inner product on $\bbL^2(\Omega)$, we let $\bsigma_{1,h}^0\in\bbS_h$ be the $\va^1$-orthogonal projection of $\bsigma_1(0)=\cC_1\beps(\bu_0)$, namely
\begin{equation}\label{eq:discrete-initial-data}
\va^1\big(\bsigma_{1,h}^0,\btau_h\big) = \va^1\big(\cC_1\beps(\bu_0),\btau_h\big) \qquad \forall\, \btau_h \in \bbS_h\, ,
\end{equation}
which is well defined because $\bu_0\in\bH^1(\Omega)$ implies $\cC_1\beps(\bu_0)\in\bbL^2(\Omega)$, and we let $\bv_h^0$, $\varphi_h^0$, and $\theta_h^0$ be the $\bL^2(\Omega)$- and $\LL^2(\Omega)$-orthogonal projections of $\bv_0$, $\varphi_0$, and $\theta_0$ onto $\bV_h$, $\Phi_h$, and $\Theta_h$, respectively. No initial datum is prescribed for $\bsigma_{2,h}$, $\bomega_h$, and $\bz_h$, as these are determined by the remaining unknowns through the algebraic relations in~\eqref{eq:semidiscrete}.
 
We now rewrite~\eqref{eq:semidiscrete} as a differential-algebraic system for the time-dependent coordinate vectors of the discrete unknowns. Let
\begin{gather*}
\big\{\btau^{(j)}\big\}_{j=1}^{n_\bsigma} \ \text{be a basis of } \bbS_h\, , \quad
\big\{\bw^{(j)}\big\}_{j=1}^{n_\bv} \ \text{a basis of } \bV_h\, , \quad
\big\{\eeta^{(j)}\big\}_{j=1}^{n_\bomega} \ \text{a basis of } \bbK_h\, , \\[0.4ex]
\big\{\br^{(j)}\big\}_{j=1}^{n_\bz} \ \text{a basis of } \bZ_h\, , \quad
\big\{\psi^{(j)}\big\}_{j=1}^{n_\varphi} \ \text{a basis of } \Phi_h\, , \quad
\big\{\vartheta^{(j)}\big\}_{j=1}^{n_\theta} \ \text{a basis of } \Theta_h\, .
\end{gather*}
Note that $\bsigma_{1,h}$ and $\bsigma_{2,h}$ are expanded in the same basis of $\bbS_h$. We expand each discrete field in its corresponding basis, with time-dependent coefficients,
\begin{gather*}
\bsigma_{i,h}(t) = \sum_{j=1}^{n_\bsigma} \sigma_{i,j}(t)\, \btau^{(j)}\quad (i\in \{1,2\})\, , \quad 
\bv_h(t) = \sum_{j=1}^{n_\bv} v_{j}(t)\, \bw^{(j)}\, , \quad \bomega_h(t) = \sum_{j=1}^{n_\bomega} \omega_j(t)\, \eeta^{(j)}\, ,\\
\bz_h(t) = \sum_{j=1}^{n_\bz} z_j(t)\, \br^{(j)}\, , \quad
\varphi_h(t) = \sum_{j=1}^{n_\varphi} \varphi_j(t)\, \psi^{(j)}\, , \qan
\theta_h(t) = \sum_{j=1}^{n_\theta} \theta_j(t)\, \vartheta^{(j)}\, , 
\end{gather*}
where the scalar coefficients are collected into the time-dependent vectors $\coef{\bsigma}_i=(\sigma_{i,j})_{j=1}^{n_\bsigma}\in\R^{n_\bsigma}$ $(i\in\{1,2\})$, $\coef{\bv}=(v_j)_{j=1}^{n_\bv}\in\R^{n_\bv}$, $\coef{\bomega}=(\omega_j)_{j=1}^{n_\bomega}\in\R^{n_\bomega}$, $\coef{\bz}=(z_j)_{j=1}^{n_\bz}\in\R^{n_\bz}$, $\coef{\varphi}=(\varphi_j)_{j=1}^{n_\varphi}\in\R^{n_\varphi}$, and $\coef{\theta}=(\theta_j)_{j=1}^{n_\theta}\in\R^{n_\theta}$.
 
Testing~\eqref{eq:semidiscrete} against the basis functions and invoking the bilinearity of $\va^i$, $\vb$, $\ca$, $\cb$ and of the $\LL^2$-inner products, every linear term becomes a matrix acting on a coefficient vector. Recalling the definitions of the bilinear forms, we introduce, for $i\in\{1,2\}$, the stress mass matrices and the two components of the pairing $\vb$,
\begin{equation*}
(\bA_i)_{k\ell} := \va^i\big(\btau^{(\ell)},\btau^{(k)}\big) \, , \quad
(\bB)_{k\ell} := \vb\big(\btau^{(\ell)},[\bw^{(k)},\cero]\big)\, , \quad
(\bE)_{k\ell} := \vb\big(\btau^{(\ell)},[\cero,\eeta^{(k)}]\big)\, ,
\end{equation*}
so that, by bilinearity, $\vb\big(\btau_h,[\bw_h,\eeta_h]\big) = \coef{\bw}^{\top}\bB\,\coef{\btau} + \coef{\eeta}^{\top}\bE\,\coef{\btau}$ for all $\btau_h\in\bbS_h$ and $[\bw_h,\eeta_h]\in\bV_h\times\bbK_h$, where $\coef{\btau}$, $\coef{\bw}$, and $\coef{\eeta}$ denote the corresponding coefficient vectors. For the reaction-diffusion block we set
\begin{equation*}
(\bC)_{k\ell} := \ca\big(\br^{(\ell)},\br^{(k)}\big) \, , \qquad
(\bK)_{k\ell} := \cb\big(\br^{(\ell)},\psi^{(k)}\big) \, ,
\end{equation*}
whence $\cb(\br_h,\psi_h) = \coef{\psi}^{\top}\bK\,\coef{\br}$.
Next, we introduce the symmetric and positive definite mass matrices
\begin{equation*}
(\bMv)_{k\ell} := (\bw^{(\ell)},\bw^{(k)})_\Omega\, , \quad
(\bMf)_{k\ell} := (\psi^{(\ell)},\psi^{(k)})_\Omega\, , \quad
(\bMt)_{k\ell} := (\vartheta^{(\ell)},\vartheta^{(k)})_\Omega\, ,
\end{equation*}
and the matrix coupling the elastic-stress trace into the calcium equation,
\begin{equation*}
(\bTc)_{k\ell} := \frac{\beta}{2\mu_1+d\lambda_1}\,(\tr(\btau^{(\ell)}),\psi^{(k)})_\Omega\, .
\end{equation*}
Finally, the model data and the nonlinear terms are collected, componentwise, in the vectors
\begin{equation*}
\begin{array}{ll}
\big(\coef{\bb}(t)\big)_k := (\varrho\,\bb(t),\bw^{(k)})_\Omega\, , &
\big(\coef{\ba}(\coef{\varphi})\big)_k := -\dfrac{1}{2\mu_2+d\lambda_2}\,(\alpha(\varphi_h),\tr(\btau^{(k)}))_\Omega\, , \\[2.2ex]
\big(\coef{\ff}(\coef{\varphi},\coef{\theta})\big)_k := (f(\varphi_h,\theta_h),\psi^{(k)})_\Omega\, , &
\big(\coef{\gg}(\coef{\varphi},\coef{\theta})\big)_k := (g(\varphi_h,\theta_h),\vartheta^{(k)})_\Omega\, .
\end{array}
\end{equation*}
By construction, $\bA_1$, $\bA_2$, $\bC$, $\bMv$, $\bMf$, and $\bMt$ are symmetric and positive definite (hence invertible), owing to the coercivity of $\va^i$, $\ca$, and of the $\LL^2$-inner products on the respective finite element spaces. Moreover, the maps $\coef{\ba}$, $\coef{\ff}$, $\coef{\gg}$ inherit the global Lipschitz continuity of $\alpha$, $f$, $g$ (cf.~\eqref{eq:alpha-Lipschitz},~\eqref{eq:assumption-f},~\eqref{eq:assumption-g}).
 
Requiring~\eqref{eq:semidiscrete} to hold for all basis functions is then equivalent to seeking the coefficient vectors $\big(\coef{\bsigma}_1,\coef{\bsigma}_2,\coef{\bv},\coef{\bomega},\coef{\bz},\coef{\varphi},\coef{\theta}\big):[0,t_{\rm end}]\to\R^{N}$, with $N=2n_\bsigma+n_\bv+n_\bomega+n_\bz+n_\varphi+n_\theta$, satisfying the differential-algebraic system
\begin{subequations}\label{eq:fem-dae}
\begin{align}
\bA_1\,\dot{\coef{\bsigma}}_1 + \bB^{\top}\coef{\bv} + \bE^{\top}\coef{\bomega} &= \cero\, , \label{eq:fem-dae-1}\\[0.3ex]
\bA_2\,\coef{\bsigma}_2 + \bB^{\top}\coef{\bv} + \bE^{\top}\coef{\bomega} &= \coef{\ba}(\coef{\varphi})\, , \label{eq:fem-dae-2}\\[0.3ex]
\varrho\,\bMv\,\dot{\coef{\bv}} - \bB\,(\coef{\bsigma}_1+\coef{\bsigma}_2) &= \coef{\bb}(t)\, , \label{eq:fem-dae-3}\\[0.3ex]
\bE\,(\coef{\bsigma}_1+\coef{\bsigma}_2) &= \cero\, , \label{eq:fem-dae-4}\\[0.3ex]
\bC\,\coef{\bz} + \bK^{\top}\coef{\varphi} &= \cero\, , \label{eq:fem-dae-5}\\[0.3ex]
\bMf\,\dot{\coef{\varphi}} - \bK\,\coef{\bz} - \bTc\,\coef{\bsigma}_1 &= \coef{\ff}(\coef{\varphi},\coef{\theta})\, , \label{eq:fem-dae-6}\\[0.3ex]
\bMt\,\dot{\coef{\theta}} &= \coef{\gg}(\coef{\varphi},\coef{\theta})\, , \label{eq:fem-dae-7}
\end{align}
\end{subequations}
subject to $\coef{\bsigma}_1(0)=\coef{\bsigma}_1^{\,0}$, $\coef{\bv}(0)=\coef{\bv}^{\,0}$, $\coef{\varphi}(0)=\coef{\varphi}^{\,0}$, $\coef{\theta}(0)=\coef{\theta}^{\,0}$, the coordinate representations of $\bsigma_{1,h}^0$, $\bv_h^0$, $\varphi_h^0$, $\theta_h^0$. Here~\eqref{eq:fem-dae-3} stems from testing the third row of~\eqref{eq:semidiscrete} against $[\bw_h,\cero]$, while~\eqref{eq:fem-dae-4} gathers its weak-symmetry component $[\cero,\eeta_h]$.
 
Note that~\eqref{eq:fem-dae-1},~\eqref{eq:fem-dae-3},~\eqref{eq:fem-dae-6}, and~\eqref{eq:fem-dae-7} are differential (in time), and~\eqref{eq:fem-dae-2},~\eqref{eq:fem-dae-4}, and~\eqref{eq:fem-dae-5} are algebraic and determine $\coef{\bsigma}_2$, $\coef{\bomega}$, and $\coef{\bz}$. Accordingly, we gather the differential unknowns into $\yvec:=(\coef{\bsigma}_1,\coef{\bv},\coef{\varphi},\coef{\theta})$ and the algebraic ones into $\xivec:=(\coef{\bsigma}_2,\coef{\bomega},\coef{\bz})$, and define the block-diagonal mass matrix
\begin{equation}\label{eq:fem-block-mass}
\blkM := \mathrm{diag}\big(\bA_1,\ \varrho\,\bMv,\ \bMf,\ \bMt\big)\, ,
\end{equation}
so that~\eqref{eq:fem-dae} adopts the compact block form
\begin{equation}\label{eq:fem-block-form}
\blkM\,\dot{\yvec}(t) + \blkH\,\yvec(t) + \blkS\,\xivec(t) = \coef{\bF}\big(t,\yvec(t)\big)\, , \qquad
\blkB\,\yvec(t) + \blkC\,\xivec(t) = \coef{\bG}\big(\yvec(t)\big)\, ,
\end{equation}
whose first (differential) block assembles~\eqref{eq:fem-dae-1},~\eqref{eq:fem-dae-3},~\eqref{eq:fem-dae-6},~\eqref{eq:fem-dae-7} and whose second (algebraic) block assembles~\eqref{eq:fem-dae-2},~\eqref{eq:fem-dae-4},~\eqref{eq:fem-dae-5}. Ordering $\yvec=(\coef{\bsigma}_1,\coef{\bv},\coef{\varphi},\coef{\theta})$ and $\xivec=(\coef{\bsigma}_2,\coef{\bomega},\coef{\bz})$, the coupling matrices read
\begin{align*}
\blkH &= \begin{bmatrix} \cero & \bB^{\top} & \cero & \cero \\ -\bB & \cero & \cero & \cero \\ -\bTc & \cero & \cero & \cero \\ \cero & \cero & \cero & \cero \end{bmatrix}, \qquad
\blkS = \begin{bmatrix} \cero & \bE^{\top} & \cero \\ -\bB & \cero & \cero \\ \cero & \cero & -\bK \\ \cero & \cero & \cero \end{bmatrix}, \\[1ex]
\blkB &= \begin{bmatrix} \cero & \bB^{\top} & \cero & \cero \\ \bE & \cero & \cero & \cero \\ \cero & \cero & \bK^{\top} & \cero \end{bmatrix}, \qquad
\blkC = \begin{bmatrix} \bA_2 & \bE^{\top} & \cero \\ \bE & \cero & \cero \\ \cero & \cero & \bC \end{bmatrix},
\end{align*}
and the right-hand sides gather the data and the nonlinearities,
\begin{equation*}
\coef{\bF}(t,\yvec) := \begin{bmatrix} \cero \\ \coef{\bb}(t) \\ \coef{\ff}(\coef{\varphi},\coef{\theta}) \\ \coef{\gg}(\coef{\varphi},\coef{\theta}) \end{bmatrix}, \qquad
\coef{\bG}(\yvec) := \begin{bmatrix} \coef{\ba}(\coef{\varphi}) \\ \cero \\ \cero \end{bmatrix}.
\end{equation*}
Note that the algebraic block of~\eqref{eq:fem-block-form} decouples into the saddle-point subsystem associated with the pair $(\coef{\bsigma}_2,\coef{\bomega})$ and the symmetric positive definite subsystem associated with $\coef{\bz}$. This structure is what allows the algebraic unknowns to be eliminated, thereby reducing~\eqref{eq:fem-block-form} to a system of ordinary differential equations, as we now show.
 
\begin{theorem}[Well-posedness of the semidiscrete problem]\label{thm:semidiscrete-wellposedness}
Let {\bf (H1)} hold and assume, in addition to Assumption~\ref{assum:data-regularity}, that $\bb \in C^0([0,t_{\rm end}];\bL^2(\Omega))$. Then problem~\eqref{eq:semidiscrete}, complemented with the initial data~\eqref{eq:discrete-initial-data}, admits a unique solution on $[0,t_{\rm end}]$, with the regularity stated in~\eqref{eq:semidiscrete-unknowns}.
\end{theorem}
\begin{proof}
With the ordering $\xivec=(\coef{\bsigma}_2,\coef{\bomega},\coef{\bz})$, the algebraic block of~\eqref{eq:fem-block-form} amounts to
\begin{equation}\label{eq:semidiscrete-wp-proof-1}
\begin{bmatrix} \bA_2 & \bE^{\top} \\ \bE & \cero \end{bmatrix}
\begin{bmatrix} \coef{\bsigma}_2 \\ \coef{\bomega} \end{bmatrix}
= \begin{bmatrix} \coef{\ba}(\coef{\varphi}) - \bB^{\top}\coef{\bv} \\ -\,\bE\,\coef{\bsigma}_1 \end{bmatrix}\, ,
\qquad
\bC\,\coef{\bz} = -\,\bK^{\top}\coef{\varphi}\, .
\end{equation}
The second identity in~\eqref{eq:semidiscrete-wp-proof-1} determines $\coef{\bz}$ uniquely, since $\bC$ is symmetric and positive definite. As for the first one, we observe that $\bA_2$ is symmetric and positive definite, because $\va^2$ is an inner product on $\bbS_h$, and that $\bE$ has full row rank. The latter follows from {\bf (H1)}: choosing $\bw_h=\cero$ therein, and recalling that $\vb(\btau_h,[\cero,\eeta_h]) = \coef{\eeta}^{\top}\bE\,\coef{\btau}$, we obtain
\begin{equation*}
\sup_{\cero \neq \btau_h \in \bbS_h}\frac{\coef{\eeta}^{\top}\bE\,\coef{\btau}}{\|\btau_h\|_{\bbH(\bdiv;\Omega)}} \geq \beta_{\rd}\, \|\eeta_h\|_{\bbL^2(\Omega)} \qquad \forall\, \eeta_h \in \bbK_h\, ,
\end{equation*}
so that $\bE^{\top}\coef{\eeta}\neq\cero$ whenever $\coef{\eeta}\neq\cero$. These two properties render the matrix in~\eqref{eq:semidiscrete-wp-proof-1} invertible. Indeed, if $\bA_2\,\coef{\bx} + \bE^{\top}\coef{\bq} = \cero$ and $\bE\,\coef{\bx} = \cero$, then multiplying the first identity by $\coef{\bx}^{\top}$ and using the second one yields $\coef{\bx}^{\top}\bA_2\,\coef{\bx} = 0$, whence $\coef{\bx}=\cero$, and therefore $\bE^{\top}\coef{\bq}=\cero$, which in turn gives $\coef{\bq}=\cero$. Consequently, there exists a map $\xivec = \Upsilon(\yvec)$ furnishing the algebraic unknowns in terms of the differential ones, which is globally Lipschitz continuous, as it depends linearly and boundedly on $(\coef{\bsigma}_1,\coef{\bv},\coef{\varphi})$ and on $\coef{\varphi}$ through the Lipschitz map $\coef{\ba}$.

Substituting $\xivec = \Upsilon(\yvec)$ into the differential block of~\eqref{eq:fem-block-form}, and recalling that the mass matrix $\blkM$ in~\eqref{eq:fem-block-mass} is symmetric and positive definite, and hence invertible, we arrive at the initial-value problem
\begin{equation}\label{eq:semidiscrete-wp-proof-2}
\dot{\yvec}(t) = \blkM^{-1}\Big( \coef{\bF}\big(t,\yvec(t)\big) - \blkH\,\yvec(t) - \blkS\,\Upsilon\big(\yvec(t)\big) \Big) =: \boldsymbol{\Psi}\big(t,\yvec(t)\big)\, , \qquad \yvec(0)=\yvec^{\,0}\, ,
\end{equation}
where $\yvec^{\,0}$ collects the coordinates of the discrete initial data. Since $\alpha$, $f$, and $g$ are globally Lipschitz continuous (cf.~\eqref{eq:alpha-Lipschitz},~\eqref{eq:assumption-f},~\eqref{eq:assumption-g}), the map $\Upsilon$ is Lipschitz, and all norms are equivalent on the finite-dimensional spaces at hand, the function $\boldsymbol{\Psi}(t,\bullet)$ is globally Lipschitz continuous, uniformly in $t$. Moreover, $\boldsymbol{\Psi}(\bullet,\yvec)$ is continuous because $\bb$ is. The Picard--Lindel\"of theorem then yields a unique maximal solution of~\eqref{eq:semidiscrete-wp-proof-2} on an interval $[0,T_{\max})$, with $T_{\max}\leq t_{\rm end}$.

In order to obtain an \textit{a priori} energy bound, we now test the first row of~\eqref{eq:semidiscrete} with $\btau_{1,h}=\bsigma_{1,h}$, the second with $\btau_{2,h}=\bsigma_{2,h}$, the third with $[\bw_h,\eeta_h]=[\bv_h,\bomega_h]$, the fourth with $\br_h=\bz_h$, the fifth with $\psi_h=\varphi_h$, and the sixth with $\vartheta_h=\theta_h$, so that the terms involving $\vb$ and $\cb$ cancel pairwise, exactly as in the proof of Theorem~\ref{thm:mixed-stability}. Introducing the discrete energy
\begin{equation*}
\sE_h(t) := \fracd{1}{2}\,\va^1(\bsigma_{1,h},\bsigma_{1,h}) + \fracd{\varrho}{2}\,\|\bv_h\|_{\bL^2(\Omega)}^2 + \fracd{1}{2}\,\|\varphi_h\|_{\LL^2(\Omega)}^2 + \fracd{1}{2}\,\|\theta_h\|_{\LL^2(\Omega)}^2\, ,
\end{equation*}
we obtain
\begin{align*}
\partial_t\sE_h + \va^2(\bsigma_{2,h},\bsigma_{2,h}) + \ca(\bz_h,\bz_h) = {}& -\fracd{1}{2\mu_2+d\lambda_2}\,(\alpha(\varphi_h),\tr(\bsigma_{2,h}))_\Omega + (\varrho\,\bb,\bv_h)_\Omega + (f(\varphi_h,\theta_h),\varphi_h)_\Omega \\
&  + (g(\varphi_h,\theta_h),\theta_h)_\Omega   + \fracd{\beta}{2\mu_1+d\lambda_1}\,(\tr(\bsigma_{1,h}),\varphi_h)_\Omega\, .
\end{align*}
The first term on the right-hand side is absorbed into $\va^2(\bsigma_{2,h},\bsigma_{2,h})$ by the Cauchy--Schwarz and Young inequalities, since $0\leq\alpha\leq\alpha_0$. The remaining ones are bounded by means of the Cauchy--Schwarz and Young inequalities, the bound $\|\tr\btau\|_{\LL^2(\Omega)}\leq\sqrt{d}\,\|\btau\|_{\bbL^2(\Omega)}$, the uniform positive definiteness of $\cA_1$, and the linear growth of $f$ and $g$ implied by~\eqref{eq:assumption-f}--\eqref{eq:assumption-g}. Hence, there exist constants $C_1,C_2>0$, independent of $t$, such that
\begin{equation*}
\partial_t\sE_h \leq C_1\,\sE_h + C_2\,\big(1+\|\bb(t)\|_{\bL^2(\Omega)}^2\big)\, ,
\end{equation*}
and Gr\"onwall's inequality yields that $\sE_h$ remains bounded on $[0,T_{\max})$, uniformly with respect to $t$.

Since $\va^1$, $\bMv$, $\bMf$ and $\bMt$ induce inner products on the corresponding discrete spaces, there is $c_h>0$ with $\sE_h \geq c_h\,|\yvec|^2$. The bound established above therefore prevents $|\yvec|$ from blowing up as $t\to T_{\max}^-$, whence $T_{\max}=t_{\rm end}$ and the solution of~\eqref{eq:semidiscrete-wp-proof-2} is defined on the whole of $[0,t_{\rm end}]$. Uniqueness follows from that of the Picard--Lindel\"of theorem together with the unique solvability of~\eqref{eq:semidiscrete-wp-proof-1}. Finally, $\dot{\yvec}=\boldsymbol{\Psi}(\bullet,\yvec)$ is continuous, so that $\yvec\in C^1([0,t_{\rm end}])$, while $\xivec=\Upsilon(\yvec)$ is continuous, which gives precisely the regularity asserted for $(\bsigma_{1,h},\bsigma_{2,h},\bv_h,\bomega_h,\bz_h,\varphi_h,\theta_h)$. This completes the proof.
\end{proof}
 
\begin{remark}\label{rem:semidiscrete-caratheodory}
The continuity assumption on $\bb$ is only used to guarantee that $\yvec\in C^1([0,t_{\rm end}])$. Under the weaker regularity $\bb\in\LL^2(0,t_{\rm end};\bL^2(\Omega))$ of Assumption~\ref{assum:data-regularity}, the same argument based on Carath\'eodory's existence theorem provides a unique absolutely continuous solution with $\dot{\yvec}\in\LL^2(0,t_{\rm end})$, and the regularity of the semidiscrete solution has to be relaxed accordingly.
\end{remark}

\section{A priori error estimates}
\label{sec:errors}

In this section we derive the error equations associated with the semidiscretisation~\eqref{eq:semidiscrete}. Throughout, $(\bsigma_1,\bsigma_2,\bv,\bomega,\bz,\varphi,\theta)$ denotes the solution of~\eqref{eq:mixed-weak} provided by Theorems~\ref{thm:mixed-reconstructions} and~\ref{thm:mixed-stability}, and $(\bsigma_{1,h},\bsigma_{2,h},\bv_h,\bomega_h,\bz_h,\varphi_h,\theta_h)$ denotes the solution of~\eqref{eq:semidiscrete} given by Theorem~\ref{thm:semidiscrete-wellposedness}.

\subsection{Projection operators and error splitting}\label{sec:projections}

Besides {\bf (H1)}, the forthcoming analysis relies on some structural properties of the finite element spaces, as well as the existence of projection operators onto the stress and flux spaces which are compatible with the divergence operators. We state these properties in an abstract form. Specific families of finite elements fulfilling them are discussed in Section~\ref{sec:FE-examples}.

\noindent \textbf{(H2)} $\bdiv(\bbS_h) = \bV_h$.
 
\noindent \textbf{(H3)} $\vdiv(\bZ_h) = \Phi_h$.

\noindent \textbf{(H4)} There exist $s>1/2$ and a bounded linear operator $\bPi_h^e : \bbH_\star(\bdiv;\Omega)\cap\bbH^{s}(\Omega) \to \bbS_h$ such that
\begin{equation}\label{eq:commuting-stress}
\bdiv\big(\bPi_h^e\btau\big) = \bsP_h^{\bV}\big(\bdiv\,\btau\big) \qquad \forall\, \btau \in \bbH_\star(\bdiv;\Omega)\cap\bbH^{s}(\Omega)\, ,
\end{equation}
where $\bsP_h^{\bV}:\bL^2(\Omega)\to\bV_h$ denotes the $\bL^2(\Omega)$-orthogonal projection onto $\bV_h$, and
\begin{equation}\label{eq:weak-symmetry-property-interpolation-operator}
(\bPi_h^e\btau,\eeta_h)_\Omega = (\btau,\eeta_h)_\Omega \qquad \forall\, \eeta_h \in \bbK_h\, , \quad \forall\, \btau \in \bbH_\star(\bdiv;\Omega)\cap\bbH^{s}(\Omega) \, .
\end{equation}

\noindent \textbf{(H5)} There exist $s>1/2$ and a bounded linear operator $\wt\bPi_h : \bH_\star(\vdiv;\Omega)\cap\bH^{s}(\Omega) \to \bZ_h$ such that
\begin{equation}\label{eq:commuting-flux}
\vdiv\big(\wt\bPi_h\br\big) = \sP_h^{\Phi}\big(\vdiv\,\br\big) \qquad \forall\, \br \in \bH_\star(\vdiv;\Omega)\cap\bH^{s}(\Omega)\, ,
\end{equation}
where $\sP_h^{\Phi}:\LL^2(\Omega)\to\Phi_h$ denotes the $\LL^2(\Omega)$-orthogonal projection onto $\Phi_h$.

In addition, we let $\bsP_h^{\bbK}:\bbL^2(\Omega)\to\bbK_h$ and $\sP_h^{\Theta}:\LL^2(\Omega)\to\Theta_h$ be the corresponding $\LL^2$-orthogonal projections. We assume from now on that the solution of~\eqref{eq:mixed-weak} is regular enough for $\bPi_h^e\bsigma_i$ and $\wt\bPi_h\bz$ to be well defined for a.e. $t\in(0,t_{\rm end}]$.

As usual, we split each error into an approximation (or projection) part and a discrete part. Setting
\begin{equation}\label{eq:error-splitting}
\begin{array}{lll}
\be_{\bsigma_i} := \bsigma_i - \bsigma_{i,h} = \be^I_{\bsigma_i} + \be^h_{\bsigma_i}\, , &
\be_{\bv} := \bv - \bv_h = \be^I_{\bv} + \be^h_{\bv}\, , &
\be_{\bomega} := \bomega - \bomega_h = \be^I_{\bomega} + \be^h_{\bomega}\, , \\[1.5ex]
\be_{\bz} := \bz - \bz_h = \be^I_{\bz} + \be^h_{\bz}\, , &
e_{\varphi} := \varphi - \varphi_h = e^I_{\varphi} + e^h_{\varphi}\, , &
e_{\theta} := \theta - \theta_h = e^I_{\theta} + e^h_{\theta}\, ,
\end{array}
\end{equation}
where the projection errors and the discrete errors are given, respectively, by
\begin{equation}\label{eq:error-splitting-def-1}
\begin{gathered}
\be^I_{\bsigma_i} := \bsigma_i - \bPi_h^e\bsigma_i\, , \qquad
\be^I_{\bv} := \bv - \bsP_h^{\bV}\bv\, , \qquad
\be^I_{\bomega} := \bomega - \bsP_h^{\bbK}\bomega\, , \\
\be^I_{\bz} := \bz - \wt\bPi_h\bz\, , \qquad
e^I_{\varphi} := \varphi - \sP_h^{\Phi}\varphi\, , \qquad
e^I_{\theta} := \theta - \sP_h^{\Theta}\theta\, ,
\end{gathered}
\end{equation}
and
\begin{equation}\label{eq:error-splitting-def-2}
\begin{gathered}
\be^h_{\bsigma_i} := \bPi_h^e\bsigma_i - \bsigma_{i,h}\, , \qquad
\be^h_{\bv} := \bsP_h^{\bV}\bv - \bv_h\, , \qquad \be^h_{\bomega} := \bsP_h^{\bbK}\bomega - \bomega_h\, , \\
\be^h_{\bz} := \wt\bPi_h\bz - \bz_h\, , \qquad
e^h_{\varphi} := \sP_h^{\Phi}\varphi - \varphi_h\, , \qquad
e^h_{\theta} := \sP_h^{\Theta}\theta - \theta_h\, .
\end{gathered}
\end{equation}
Note that, by construction, the discrete errors are discrete functions, that is
\begin{equation*}
\be^h_{\bsigma_i}(t) \in \bbS_h\, , \quad \be^h_{\bv}(t) \in \bV_h\, , \quad \be^h_{\bomega}(t) \in \bbK_h\, , \quad \be^h_{\bz}(t) \in \bZ_h\, , \quad e^h_{\varphi}(t) \in \Phi_h\, , \qan e^h_{\theta}(t) \in \Theta_h\, ,
\end{equation*}
so they are admissible test functions in~\eqref{eq:semidiscrete}. We shall repeatedly use the following three consequences of the hypotheses above. First, since $\bdiv(\bbS_h)=\bV_h$ (cf.~{\bf (H2)}) and $\vdiv(\bZ_h)=\Phi_h$ (cf.~{\bf (H3)}), the definitions of $\bsP_h^{\bV}$ and $\sP_h^{\Phi}$ yield
\begin{equation}\label{eq:orthogonality-div-test}
(\bdiv\,\btau_h,\be^I_{\bv})_\Omega = 0 \qan (\vdiv\,\br_h,e^I_{\varphi})_\Omega = 0 \qquad \forall\, \btau_h \in \bbS_h\, ,\ \forall\, \br_h \in \bZ_h\, .
\end{equation}
Second, the commuting properties~\eqref{eq:commuting-stress} and~\eqref{eq:commuting-flux}, combined with the property~\eqref{eq:weak-symmetry-property-interpolation-operator} give
\begin{equation}\label{eq:orthogonality-div-proj}
\vb(\be^I_{\bsigma_1}+\be^I_{\bsigma_2},[\bw_h,\eeta_h]) = 0 \qan (\vdiv\,\be^I_{\bz},\psi_h)_\Omega = 0 \qquad \forall\, [\bw_h,\eeta_h] \in \bV_h\times \bbK_h\, ,\ \forall\, \psi_h \in \Phi_h\, .
\end{equation}
Finally, since $\bsP_h^{\bV}$, $\sP_h^{\Phi}$, and $\sP_h^{\Theta}$ are time-independent, they commute with $\partial_t$, whence
\begin{equation}\label{eq:orthogonality-time}
(\dot \be^I_\bv,\bw_h)_\Omega = 0\, , \quad (\dot e^I_\varphi,\psi_h)_\Omega = 0\, , \qan (\dot e^I_\theta,\vartheta_h)_\Omega = 0 \qquad \forall\, \bw_h \in \bV_h\, ,\ \forall\, \psi_h \in \Phi_h\, ,\ \forall\, \vartheta_h \in \Theta_h\, .
\end{equation}

\subsection{The error equations}\label{sec:error-equations}

Since the finite element spaces are conforming, the discrete test functions are admissible in~\eqref{eq:mixed-weak}. Hence, subtracting~\eqref{eq:semidiscrete} from~\eqref{eq:mixed-weak} tested against $(\btau_{1,h},\btau_{2,h},[\bw_h,\eeta_h],\br_h,\psi_h,\vartheta_h)\in \bbS_h\times\bbS_h\times[\bV_h\times\bbK_h]\times\bZ_h\times\Phi_h\times\Theta_h$, we obtain, for a.e. $t\in(0,t_{\rm end}]$,
\begin{equation}\label{eq:error-eqs}
\begin{array}{rcl}
\va^1(\dot{\be}_{\bsigma_1},\btau_{1,h}) + \vb(\btau_{1,h},[\be_{\bv},\be_{\bomega}]) & = & 0 \, , \\[1.25ex]
\va^2(\be_{\bsigma_2},\btau_{2,h}) + \vb(\btau_{2,h},[\be_{\bv},\be_{\bomega}]) & = & -\frac{1}{2\mu_2+d\lambda_2}\,\big(\alpha(\varphi)-\alpha(\varphi_h),\tr(\btau_{2,h})\big)_\Omega \, , \\[1.25ex]
(\varrho\,\dot{\be}_{\bv},\bw_h)_\Omega - \vb(\be_{\bsigma_1}+\be_{\bsigma_2},[\bw_h,\eeta_h]) & = & 0 \, , \\[1.25ex]
\ca(\be_{\bz},\br_h) + \cb(\br_h,e_{\varphi}) & = & 0 \, , \\[1.25ex]
(\dot{e}_{\varphi},\psi_h)_\Omega - \cb(\be_{\bz},\psi_h) & = & \big(f(\varphi,\theta)-f(\varphi_h,\theta_h),\psi_h\big)_\Omega + \tfrac{\beta}{2\mu_1+d\lambda_1}\big(\tr(\be_{\bsigma_1}),\psi_h\big)_\Omega \, , \\[1.25ex]
(\dot{e}_{\theta},\vartheta_h)_\Omega & = & \big(g(\varphi,\theta)-g(\varphi_h,\theta_h),\vartheta_h\big)_\Omega \, .
\end{array}
\end{equation}

We now insert the splitting~\eqref{eq:error-splitting} into~\eqref{eq:error-eqs} and move the projection errors to the right-hand side. This yields a system involving only $(\be^h_{\bsigma_1},\be^h_{\bsigma_2},\be^h_{\bv},\be^h_{\bomega},\be^h_{\bz},e^h_{\varphi},e^h_{\theta})$ on the left-hand side. Several terms vanish by the orthogonality relations~\eqref{eq:orthogonality-div-test}--\eqref{eq:orthogonality-time}. In particular, recalling the definition of $\vb$ and using~\eqref{eq:orthogonality-div-test}, the first two rows of~\eqref{eq:error-eqs} give
\begin{equation*}
\vb\big(\btau_{i,h},[\be^I_{\bv},\be^I_{\bomega}]\big) = (\bdiv\,\btau_{i,h},\be^I_{\bv})_\Omega + (\btau_{i,h},\be^I_{\bomega})_\Omega = (\btau_{i,h},\be^I_{\bomega})_\Omega\, .
\end{equation*}
In turn, by~\eqref{eq:orthogonality-div-proj} and~\eqref{eq:orthogonality-time}, in the third row we have $\vb\big(\be^I_{\bsigma_1}+\be^I_{\bsigma_2},[\bw_h,\eeta_h]\big) = 0$ and $(\varrho\,\dot \be^I_\bv,\bw_h)_\Omega = 0$, respectively. In the fourth and fifth rows, $\cb(\br_h,e^I_{\varphi}) = 0$ and $\cb(\be^I_{\bz},\psi_h) = 0$ by~\eqref{eq:orthogonality-div-test} and~\eqref{eq:orthogonality-div-proj}, respectively, while $(\dot e^I_\varphi,\psi_h)_\Omega = 0$ and $(\dot e^I_\theta,\vartheta_h)_\Omega = 0$ by~\eqref{eq:orthogonality-time}.

Consequently, the discrete errors satisfy, for a.e. $t\in(0,t_{\rm end}]$,
\begin{equation}\label{eq:discrete-error-eqs}
\begin{array}{rcl}
\va^1(\dot \be^h_{\bsigma_1},\btau_{1,h}) + \vb(\btau_{1,h},[\be^h_{\bv},\be^h_{\bomega}]) & = & -\,\va^1(\dot\be^I_{\bsigma_1},\btau_{1,h}) - (\btau_{1,h},\be^I_{\bomega})_\Omega \, , \\[1.25ex]
\va^2(\be^h_{\bsigma_2},\btau_{2,h}) + \vb(\btau_{2,h},[\be^h_{\bv},\be^h_{\bomega}]) & = & -\,\va^2(\be^I_{\bsigma_2},\btau_{2,h}) - (\btau_{2,h},\be^I_{\bomega})_\Omega \\[0.6ex]
& & \quad -\ \frac{1}{2\mu_2+d\lambda_2}\,\big(\alpha(\varphi)-\alpha(\varphi_h),\tr(\btau_{2,h})\big)_\Omega \, , \\[1.25ex]
(\varrho\,\dot{\be}^h_{\bv},\bw_h)_\Omega - \vb(\be^h_{\bsigma_1}+\be^h_{\bsigma_2},[\bw_h,\eeta_h]) & = & 0 \, , \\[1.25ex]
\ca(\be^h_{\bz},\br_h) + \cb(\br_h,e^h_{\varphi}) & = & -\,\ca(\be^I_{\bz},\br_h) \, , \\[1.25ex]
(\dot{e}^h_{\varphi},\psi_h)_\Omega - \cb(\be^h_{\bz},\psi_h) & = & \big(f(\varphi,\theta)-f(\varphi_h,\theta_h),\psi_h\big)_\Omega \\[0.6ex]
& & \quad +\ \tfrac{\beta}{2\mu_1+d\lambda_1}\big(\tr(\be^I_{\bsigma_1}+\be^h_{\bsigma_1}),\psi_h\big)_\Omega \, , \\[1.25ex]
(\dot{e}^h_{\theta},\vartheta_h)_\Omega & = & \big(g(\varphi,\theta)-g(\varphi_h,\theta_h),\vartheta_h\big)_\Omega \, ,
\end{array}
\end{equation}
for all $(\btau_{1,h},\btau_{2,h},[\bw_h,\eeta_h],\br_h,\psi_h,\vartheta_h)\in \bbS_h\times\bbS_h\times[\bV_h\times\bbK_h]\times\bZ_h\times\Phi_h\times\Theta_h$.

Finally, concerning the initial discrete errors, the choice of the discrete initial data made in~\eqref{eq:discrete-initial-data} gives
\begin{equation}\label{eq:initial-discrete-errors}
\be^h_{\bsigma_1}(0) = \bPi_h^e\bsigma_1(0) - \bsigma_{1,h}^0\, , \quad
\be^h_{\bv}(0) = \cero\, , \quad
e^h_{\varphi}(0) = 0\, , \qan
e^h_{\theta}(0) = 0\, ,
\end{equation}
since $\bv_h^0$, $\varphi_h^0$, and $\theta_h^0$ were defined as the $\LL^2$-projections of $\bv_0$, $\varphi_0$, and $\theta_0$ onto $\bV_h$, $\Phi_h$, and $\Theta_h$, respectively.

\subsection{A priori error estimate}\label{sec:error-estimate}

Following the structure of~\eqref{eq:discrete-error-eqs}, we measure the discrete errors by means of the energy and dissipation functionals associated to the error equations:
\begin{equation}\label{eq:error-energy-functionals}
\begin{gathered}
\sE_h^{\sf e}(t) := \fracd{1}{2}\, \va^1\big(\be^h_{\bsigma_1},\be^h_{\bsigma_1}\big) + \fracd{\varrho}{2}\,\|\be^h_{\bv}\|_{\bL^2(\Omega)}^2 + \fracd{1}{2}\, \|e^h_{\varphi}\|_{\LL^2(\Omega)}^2 + \fracd{1}{2}\, \|e^h_{\theta}\|_{\LL^2(\Omega)}^2 \, , \\[1ex]
\sD_h^{\sf e}(t) := \va^2\big(\be^h_{\bsigma_2},\be^h_{\bsigma_2}\big) + \ca\big(\be^h_{\bz},\be^h_{\bz}\big)\, ,
\end{gathered}
\end{equation}
and we collect the initial discrete error together with the consistency contributions stemming from the projection errors into
\begin{equation}\label{eq:error-consistency-functional}
\begin{aligned}
\Xi(t)^2 &:= \|\be^h_{\bsigma_1}(0)\|_{\bbL^2(\Omega)}^2 + \|\dot{\be}^I_{\bsigma_1}\|_{\LL^2(0,t;\bbL^2(\Omega))}^2 + \|\be^I_{\bsigma_1}\|_{\LL^\infty(0,t;\bbL^2(\Omega))}^2 + \|\be^I_{\bv}\|_{\LL^\infty(0,t;\bL^2(\Omega))}^2 \\
&\quad + \|\be^I_{\bomega}\|_{\LL^2(0,t;\bbL^2(\Omega))}^2 + \|\be^I_{\bsigma_2}\|_{\LL^2(0,t;\bbL^2(\Omega))}^2 + \|\be^I_{\bz}\|_{\LL^2(0,t;\bL^2(\Omega))}^2 \\
&\quad + \|e^I_{\varphi}\|_{\LL^\infty(0,t;\LL^2(\Omega))}^2 + \|e^I_{\theta}\|_{\LL^\infty(0,t;\LL^2(\Omega))}^2 \, .
\end{aligned}
\end{equation}

\begin{theorem}[A priori error estimate]\label{thm:error-estimate}
Let the hypotheses of Theorem~\ref{thm:semidiscrete-wellposedness} hold, assume {\bf (H2)}--{\bf (H5)}, and suppose in addition that $\bsigma_1 \in \HH^1(0,t_{\rm end};\bbH_\star(\bdiv;\Omega)\cap\bbH^s(\Omega))$. Then there exists $C>0$, depending only on $\cA_1$, $\cA_2$, $\diff$, $\varrho$, the Lipschitz constants of $\alpha$, $f$, $g$, the Lam\'e parameters, $\beta$, $d$, and $t_{\rm end}$, but independent of $h$, such that, for all $t\in(0,t_{\rm end}]$,
\begin{equation}\label{eq:error-estimate}
\begin{aligned}
&\|\be_{\bsigma_1}\|_{\LL^\infty(0,t;\bbL^2(\Omega))}^2 + \|\be_{\bv}\|_{\LL^\infty(0,t;\bL^2(\Omega))}^2 + \|e_{\varphi}\|_{\LL^\infty(0,t;\LL^2(\Omega))}^2 + \|e_{\theta}\|_{\LL^\infty(0,t;\LL^2(\Omega))}^2 \\
&\qquad + \|\be_{\bsigma_2}\|_{\LL^2(0,t;\bbL^2(\Omega))}^2 + \|\be_{\bz}\|_{\LL^2(0,t;\bL^2(\Omega))}^2 \leq C\, \Xi(t)^2\, .
\end{aligned}
\end{equation}
\end{theorem}

\begin{proof}
We first record that, by the uniform positive definiteness and boundedness of $\cA_1$, $\cA_2$ and $\diff$, the quantities introduced in~\eqref{eq:error-energy-functionals} control the discrete errors in the sense that
\begin{equation}\label{eq:error-estimate-proof-control}
\begin{gathered}
\|\be^h_{\bsigma_1}\|_{\bbL^2(\Omega)}^2 + \|\be^h_{\bv}\|_{\bL^2(\Omega)}^2 + \|e^h_{\varphi}\|_{\LL^2(\Omega)}^2 + \|e^h_{\theta}\|_{\LL^2(\Omega)}^2 \leq c_1\, \sE_h^{\sf e}\, , \\[1ex]
\|\be^h_{\bsigma_2}\|_{\bbL^2(\Omega)}^2 + \|\be^h_{\bz}\|_{\bL^2(\Omega)}^2 \leq c_2\, \sD_h^{\sf e}\, ,
\end{gathered}
\end{equation}
where $c_1>0$ depends only on $\varrho$, $\mu_1$ and $\lambda_1$, and $c_2>0$ only on $\mu_2$, $\lambda_2$ and $\overline{D}$.

We now take $\btau_{i,h}=\be^h_{\bsigma_i}$, $[\bw_h,\eeta_h]=[\be^h_{\bv},\be^h_{\bomega}]$, $\br_h=\be^h_{\bz}$, $\psi_h=e^h_{\varphi}$, and $\vartheta_h=e^h_{\theta}$ in~\eqref{eq:discrete-error-eqs}, and add the first three rows and, separately, the fourth and fifth ones. Note that the contributions of $\vb$ cancel pairwise, and likewise with $\cb$. Recalling that $\va^1$, $\va^2$, $\ca$, and the $\LL^2$-inner products are symmetric and time-independent, and adding the resulting identities together with the sixth row, we arrive at
\begin{equation}\label{eq:error-estimate-proof-1}
\partial_t \sE_h^{\sf e}(t) + \sD_h^{\sf e}(t) = \sR_1(t) + \sR_2(t) \, ,
\end{equation}
where the two groups of consistency terms are given by
\begin{equation*}
\begin{aligned}
\sR_1 := {}& -\,\va^1\big(\dot{\be}^I_{\bsigma_1},\be^h_{\bsigma_1}\big) - \va^2\big(\be^I_{\bsigma_2},\be^h_{\bsigma_2}\big) - \big(\be^h_{\bsigma_1}+\be^h_{\bsigma_2},\be^I_{\bomega}\big)_\Omega - \ca\big(\be^I_{\bz},\be^h_{\bz}\big) \, , \\
\sR_2 := {}& -\fracd{1}{2\mu_2+d\lambda_2}\big(\alpha(\varphi)-\alpha(\varphi_h),\tr(\be^h_{\bsigma_2})\big)_\Omega + \big(f(\varphi,\theta)-f(\varphi_h,\theta_h),e^h_{\varphi}\big)_\Omega \\
& + \big(g(\varphi,\theta)-g(\varphi_h,\theta_h),e^h_{\theta}\big)_\Omega + \fracd{\beta}{2\mu_1+d\lambda_1}\big(\tr(\be^I_{\bsigma_1}+\be^h_{\bsigma_1}),e^h_{\varphi}\big)_\Omega \, .
\end{aligned}
\end{equation*}

The four terms defining $\sR_1$ are estimated by the Cauchy--Schwarz inequality, the boundedness of $\cA_1$, $\cA_2$ and $\diff^{-1}$, and~\eqref{eq:error-estimate-proof-control}, which yield
\begin{equation}\label{eq:error-estimate-proof-2}
|\sR_1| \leq c_3\, \Big( \|\dot{\be}^I_{\bsigma_1}\|_{\bbL^2(\Omega)} + \|\be^I_{\bomega}\|_{\bbL^2(\Omega)} \Big)\, \sqrt{\sE_h^{\sf e}} + c_4\, \Big( \|\be^I_{\bsigma_2}\|_{\bbL^2(\Omega)} + \|\be^I_{\bomega}\|_{\bbL^2(\Omega)} + \|\be^I_{\bz}\|_{\bL^2(\Omega)} \Big)\, \sqrt{\sD_h^{\sf e}} \, ,
\end{equation}
with $c_3>0$ depending only on $c_1$ and the Lam\'e parameters, and $c_4>0$ only on $c_2$, the Lam\'e parameters, $\underline{D}$ and $\overline{D}$.

In order to bound $\sR_2$, we use the Lipschitz continuity of $\alpha$, $f$ and $g$ (cf.~\eqref{eq:alpha-Lipschitz},~\eqref{eq:assumption-f},~\eqref{eq:assumption-g}) together with the splitting~\eqref{eq:error-splitting}, which give
\begin{equation*}
\|\alpha(\varphi)-\alpha(\varphi_h)\|_{\LL^2(\Omega)} \leq L_\alpha \Big( \|e^I_{\varphi}\|_{\LL^2(\Omega)} + \|e^h_{\varphi}\|_{\LL^2(\Omega)} \Big)\, ,
\end{equation*}
and analogously for the differences of $f$ and $g$, with the additional contributions $\|e^I_{\theta}\|_{\LL^2(\Omega)}+\|e^h_{\theta}\|_{\LL^2(\Omega)}$ and the corresponding Lipschitz constants. Using in addition that $\|\tr\btau\|_{\LL^2(\Omega)}\leq\sqrt{d}\,\|\btau\|_{\bbL^2(\Omega)}$, together with~\eqref{eq:error-estimate-proof-control}, we obtain
\begin{equation}\label{eq:error-estimate-proof-3}
|\sR_2| \leq c_5\, \sE_h^{\sf e} + c_6\, \big( \|\be^I_{\bsigma_1}\|_{\bbL^2(\Omega)} + \|e^I_{\varphi}\|_{\LL^2(\Omega)} + \|e^I_{\theta}\|_{\LL^2(\Omega)} \big)\, \sqrt{\sE_h^{\sf e}} + c_7\, \big( \|e^I_{\varphi}\|_{\LL^2(\Omega)} + \sE_h^{\sf e} \big)\, \sqrt{\sD_h^{\sf e}} \, ,
\end{equation}
where $c_5$, $c_6$ and $c_7$ are positive constants depending only on $c_1$, $c_2$, $d$, $\beta$, the Lam\'e parameters, and the Lipschitz constants of $\alpha$, $f$ and $g$.

Gathering~\eqref{eq:error-estimate-proof-2} and~\eqref{eq:error-estimate-proof-3} in~\eqref{eq:error-estimate-proof-1}, and collecting the projection errors into
\begin{equation}\label{eq:error-estimate-proof-FG}
\begin{gathered}
F(t) := \|\dot{\be}^I_{\bsigma_1}\|_{\bbL^2(\Omega)} + \|\be^I_{\bsigma_1}\|_{\bbL^2(\Omega)} + \|\be^I_{\bomega}\|_{\bbL^2(\Omega)} + \|e^I_{\varphi}\|_{\LL^2(\Omega)} + \|e^I_{\theta}\|_{\LL^2(\Omega)}\, , \\[1ex]
G(t) := \|\be^I_{\bsigma_2}\|_{\bbL^2(\Omega)} + \|\be^I_{\bomega}\|_{\bbL^2(\Omega)} + \|\be^I_{\bz}\|_{\bL^2(\Omega)} + \|e^I_{\varphi}\|_{\LL^2(\Omega)}\, ,
\end{gathered}
\end{equation}
we arrive at the differential inequality
\begin{equation}\label{eq:error-estimate-proof-4}
\partial_t \sE_h^{\sf e} + \sD_h^{\sf e} \leq c_5\, \sE_h^{\sf e} + c_8\, F(t)\,\sqrt{\sE_h^{\sf e}} + c_9\, G(t)\,\sqrt{\sD_h^{\sf e}} + c_7\,\sqrt{\sE_h^{\sf e}\, \sD_h^{\sf e}} \, ,
\end{equation}
with $c_8 := \max\{c_3,c_6\}$ and $c_9 := \max\{c_4,c_7\}$. The last two terms on the right-hand side of~\eqref{eq:error-estimate-proof-4} are absorbed by means of Young's inequality, namely
\begin{equation*}
c_9\,G\,\sqrt{\sD_h^{\sf e}} \leq \fracd{1}{4}\,\sD_h^{\sf e} + c_9^2\,G^2 \qan c_7\,\sqrt{\sE_h^{\sf e}\,\sD_h^{\sf e}} \leq \fracd{1}{4}\,\sD_h^{\sf e} + c_7^2\,\sE_h^{\sf e}\, ,
\end{equation*}
whereas $c_8\,F\,\sqrt{\sE_h^{\sf e}} \leq \tfrac{1}{2}\,c_8^2\,F^2 + \tfrac{1}{2}\,\sE_h^{\sf e}$. Substituting these bounds into~\eqref{eq:error-estimate-proof-4}, we deduce that
\begin{equation}\label{eq:error-estimate-proof-5}
\partial_t \sE_h^{\sf e} + \fracd{1}{2}\, \sD_h^{\sf e} \leq C_1\, \sE_h^{\sf e} + C_2\,\big( F(t)^2 + G(t)^2 \big) \, ,
\end{equation}
where $C_1$ and $C_2$ depend only on $c_5$, $c_7$, $c_8$ and $c_9$, and hence, ultimately, only on $\varrho$, $d$, $\beta$, the Lam\'e parameters, $\underline{D}$, $\overline{D}$, and the Lipschitz constants of $\alpha$, $f$ and $g$. In particular, neither $C_1$ nor $C_2$ depends on $h$ or on $t_{\rm end}$.

Integrating~\eqref{eq:error-estimate-proof-5} from $0$ to $t$ and discarding, for the moment, the nonnegative dissipative term, Gr\"onwall's inequality yields
\begin{equation}\label{eq:error-estimate-proof-6}
\sE_h^{\sf e}(t) \leq e^{C_1 t}\, \left( \sE_h^{\sf e}(0) + C_2 \int_0^t \big( F(s)^2 + G(s)^2 \big) \ds \right) \qquad \forall\, t \in (0,t_{\rm end}]\, .
\end{equation}
In turn, integrating~\eqref{eq:error-estimate-proof-5} without discarding $\sD_h^{\sf e}$ and bounding $\int_0^t \sE_h^{\sf e}$ by means of~\eqref{eq:error-estimate-proof-6}, we also obtain
\begin{equation}\label{eq:error-estimate-proof-7}
\|\sE_h^{\sf e}\|_{\LL^\infty(0,t)} + \|\sD_h^{\sf e}\|_{\LL^1(0,t)} \leq C_3\, \left( \sE_h^{\sf e}(0) + \int_0^t \big( F(s)^2 + G(s)^2 \big) \ds \right) \, ,
\end{equation}
where $C_3$ depends only on $C_1$, $C_2$ and $t_{\rm end}$.

It remains to identify the two terms on the right-hand side of~\eqref{eq:error-estimate-proof-7}. On the one hand, since $\be^h_{\bv}(0)=\cero$, $e^h_{\varphi}(0)=0$ and $e^h_{\theta}(0)=0$ (cf.~\eqref{eq:initial-discrete-errors}), the definition of $\sE_h^{\sf e}$ and the boundedness of $\cA_1$ give $\sE_h^{\sf e}(0) \leq c_{10}\,\|\be^h_{\bsigma_1}(0)\|_{\bbL^2(\Omega)}^2$, with $c_{10}$ depending only on $\mu_1$ and $\lambda_1$. On the other hand, squaring~\eqref{eq:error-estimate-proof-FG} and recalling the definition~\eqref{eq:error-consistency-functional} of $\Xi$, we readily see that $\int_0^t (F^2+G^2)\ds \leq c_{11}\,\Xi(t)^2$, with $c_{11}$ depending on $t_{\rm end}$. Hence,
\begin{equation}\label{eq:error-estimate-proof-8}
\|\sE_h^{\sf e}\|_{\LL^\infty(0,t)} + \|\sD_h^{\sf e}\|_{\LL^1(0,t)} \leq C_4\, \Xi(t)^2\, ,
\end{equation}
with $C_4$ depending only on $C_3$, $c_{10}$ and $c_{11}$.

Finally, by~\eqref{eq:error-estimate-proof-control}, estimate~\eqref{eq:error-estimate-proof-8} controls $\be^h_{\bsigma_1}$, $\be^h_{\bv}$, $e^h_{\varphi}$ and $e^h_{\theta}$ in the $\LL^\infty$-in-time norms, and $\be^h_{\bsigma_2}$ and $\be^h_{\bz}$ in the $\LL^2$-in-time norms, all of them with a constant depending only on $C_4$, $c_1$ and $c_2$. The assertion~\eqref{eq:error-estimate} then follows from the splitting~\eqref{eq:error-splitting} and the triangle inequality, the norms of the projection errors thus arising being themselves part of $\Xi(t)^2$. Tracking the dependencies of $c_1,\ldots,c_{11}$ and $C_1,\ldots,C_4$ collected along the proof, the resulting constant depends only on $\cA_1$, $\cA_2$, $\diff$, $\varrho$, the Lipschitz constants of $\alpha$, $f$ and $g$, the Lam\'e parameters, $\beta$, $d$ and $t_{\rm end}$, but not on $h$. This completes the proof.
\end{proof}

\begin{remark}\label{rem:initial-error-term}
The term $\|\be^h_{\bsigma_1}(0)\|_{\bbL^2(\Omega)}$ in~\eqref{eq:error-estimate} is itself of approximation order. Indeed, by~\eqref{eq:initial-discrete-errors} we have $\be^h_{\bsigma_1}(0) = \bPi_h^e\bsigma_1(0) - \bsigma_{1,h}^0$, and since $\bsigma_{1,h}^0$ is the $\va^1$-orthogonal projection of $\bsigma_1(0)$ onto $\bbS_h$, hence quasi-optimal in the $\bbL^2(\Omega)$-norm, the triangle inequality yields $\|\be^h_{\bsigma_1}(0)\|_{\bbL^2(\Omega)} \leq c\,\|\be^I_{\bsigma_1}(0)\|_{\bbL^2(\Omega)}$, with $c>0$ depending only on $\mu_1$ and $\lambda_1$. Alternatively, choosing $\bsigma_{1,h}^0 := \bPi_h^e\big(\cC_1\beps(\bu_0)\big)$ in~\eqref{eq:discrete-initial-data} gives $\be^h_{\bsigma_1}(0)=\cero$ and removes this term altogether.
\end{remark}

\begin{theorem}[Error estimate for the spin tensor]\label{thm:error-estimate-spin}
Under the hypotheses of Theorem~\ref{thm:error-estimate}, there exists $C_{\rm spin}>0$, depending only on the same quantities as the constant in~\eqref{eq:error-estimate} and on the inf-sup constant $\beta_{\rd}$ from {\bf (H1)}, but independent of $h$, such that, for all $t\in(0,t_{\rm end}]$,
\begin{equation}\label{eq:error-estimate-spin}
\|\be_{\bomega}\|_{\LL^2(0,t;\bbL^2(\Omega))}^2 \leq C_{\rm spin}\, \Xi(t)^2\, .
\end{equation}
\end{theorem}
\begin{proof}
By using {\bf (H1)} with $[\bw_h,\eeta_h] = [\be_\bv^h,\be_\bomega^h]$ and the second row in the system~\eqref{eq:discrete-error-eqs},
\begin{align*}
&\beta_{\rd} \, \|\be_\bomega^h\|_{\bbL^2(\Omega)} \leq \beta_{\rd} \, \|(\be_\bv^h,\be_\bomega^h)\|_{\bL^2(\Omega)\times \bbL^2(\Omega)} \leq \sup_{\cero\neq \btau_h \in \bbS_h} \frac{\vb(\btau_h,[\be_\bv^h,\be_\bomega^h])}{\|\btau_h\|_{\bbH(\bdiv;\Omega)}} \\
&\quad = \sup_{\cero\neq \btau_h \in \bbS_h} \frac{-\va^2(\be_{\bsigma_2},\btau_h) - (\btau_h,\be_\bomega^I)_\Omega - \frac{1}{2\mu_2+d\lambda_2}(\alpha(\varphi)-\alpha(\varphi_h),\tr(\btau_h))_\Omega}{\|\btau_h\|_{\bbH(\bdiv;\Omega)}} \, .
\end{align*}
We then bound each term on the right-hand side using the boundedness of $\va^2$, the Lipschitz continuity of $\alpha$ (cf.~\eqref{eq:alpha-Lipschitz}), and the fact that $\|\btau_h\|_{\bbL^2(\Omega)} \leq \|\btau_h\|_{\bbH(\bdiv;\Omega)}$, and take squares in the resulting inequality, thus arriving at
\begin{equation*}
\|\be_\bomega^h\|_{\bbL^2(\Omega)}^2 \leq c \, \left\{ \|\be_{\bsigma_2}\|_{\bbL^2(\Omega)}^2 + \|\be_\bomega^I\|_{\bbL^2(\Omega)}^2 + \|e_\varphi\|_{\LL^2(\Omega)}^2 \right\} \, ,\qquad \text{a.e. in }(0,t_{\rm end}]\, ,
\end{equation*}
with $c>0$ depending only on $\beta_{\rd}$, $\cA_2$, $d$, $\mu_2$, $\lambda_2$, and $L_\alpha$. Integrating from $0$ to $t\in(0,t_{\rm end}]$, and bounding $\|e_\varphi\|_{\LL^2(0,t;\LL^2(\Omega))} \leq \sqrt{t_{\rm end}}\, \|e_\varphi\|_{\LL^\infty(0,t;\LL^2(\Omega))}$, the estimate~\eqref{eq:error-estimate} yields
\begin{equation*}
\|\be_\bomega^h\|_{\LL^2(0,t;\bbL^2(\Omega))}^2 \leq \tilde c \; \Xi(t)^2 \, .
\end{equation*}
The assertion~\eqref{eq:error-estimate-spin} then follows from the splitting $\be_{\bomega} = \be^I_{\bomega} + \be^h_{\bomega}$ and the triangle inequality, since $\|\be^I_{\bomega}\|_{\LL^2(0,t;\bbL^2(\Omega))}^2$ is itself part of $\Xi(t)^2$. This completes the proof.
\end{proof}

For the statement of the next result, we let $\Xi_1(t)$ stand for the quantity obtained from $\Xi(t)$ in~\eqref{eq:error-consistency-functional} upon discarding the initial term $\|\be^h_{\bsigma_1}(0)\|_{\bbL^2(\Omega)}^2$ and replacing every projection error by its time derivative, and $\sE_h^{\sf{e}\,,(1)}(t)$, $\sD_h^{\sf{e}\,,(1)}(t)$ for the functionals in~\eqref{eq:error-energy-functionals} with each discrete error replaced by its time derivative. We also abbreviate
\begin{equation*}
J := f(\varphi,\theta) + \fracd{\beta}{2\mu_1+d\lambda_1}\,\tr(\bsigma_1)\, , \qquad
J_h := f(\varphi_h,\theta_h) + \fracd{\beta}{2\mu_1+d\lambda_1}\,\tr(\bsigma_{1,h})\, .
\end{equation*}

\begin{theorem}[Error estimates in the divergence norms]\label{thm:error-estimate-div}
In addition to the hypotheses of Theorem~\ref{thm:error-estimate}, assume that
\begin{enumerate}
\item[\rm (i)] $f,g\in C^1(\R^2)$, with bounded and globally Lipschitz continuous first-order partial derivatives;
\item[\rm (ii)] $\bb \in \HH^1(0,t_{\rm end};\bL^2(\Omega))$;
\item[\rm (iii)] the solution of~\eqref{eq:mixed-weak} satisfies
\begin{gather*}
\bsigma_1 \in \HH^2\big(0,t_{\rm end};\bbH_\star(\bdiv;\Omega)\cap\bbH^s(\Omega)\big)\, , \quad
\bsigma_2 \in \HH^1\big(0,t_{\rm end};\bbH_\star(\bdiv;\Omega)\cap\bbH^s(\Omega)\big)\, , \\
\bz \in \HH^1\big(0,t_{\rm end};\bH_\star(\vdiv;\Omega)\cap\bH^s(\Omega)\big)\, , \quad
\bomega \in \HH^1\big(0,t_{\rm end};\bbL^2(\Omega)\big)\, , \quad
\bv \in \HH^2\big(0,t_{\rm end};\bL^2(\Omega)\big)\, , \\
\varphi,\theta \in \HH^2\big(0,t_{\rm end};\LL^2(\Omega)\big) \cap \rW^{1,\infty}\big(0,t_{\rm end};\LL^\infty(\Omega)\big)\, .
\end{gather*}
\end{enumerate}
Then $\Xi_1(t)<\infty$ and there exists $C_{\rm div}>0$, independent of $h$ and depending only on the same quantities as the constant in~\eqref{eq:error-estimate}, on the bounds and Lipschitz constants of $\alpha'$ and of the first-order partial derivatives of $f$ and $g$, and on $\|\dot{\varphi}\|_{\LL^\infty(0,t_{\rm end};\LL^\infty(\Omega))}+\|\dot{\theta}\|_{\LL^\infty(0,t_{\rm end};\LL^\infty(\Omega))}$, such that, for all $t\in(0,t_{\rm end}]$,
\begin{equation}\label{eq:error-estimate-div}
\begin{aligned}
&\|\bdiv(\be_{\bsigma_1}+\be_{\bsigma_2})\|_{\LL^2(0,t;\bL^2(\Omega))}^2 + \|\vdiv\,\be_{\bz}\|_{\LL^2(0,t;\LL^2(\Omega))}^2
\\
&\quad \leq C_{\rm div}\,\Big\{ \sE_h^{\sf{e}\,,(1)}(0) + \Xi(t)^2 + \Xi_1(t)^2 + \|(\mathrm{I}-\bsP_h^{\bV})\bb\|_{\LL^2(0,t;\bL^2(\Omega))}^2 + \|(\mathrm{I}-\sP_h^{\Phi})\, J \|_{\LL^2(0,t;\LL^2(\Omega))}^2 \Big\}\, .
\end{aligned}
\end{equation}
\end{theorem}
\begin{proof}
First, since $\dot{\bv}_h\in\bV_h$ and $\bdiv(\bbS_h)=\bV_h$ (cf.~{\bf (H2)}), the third row of~\eqref{eq:semidiscrete}, tested with $[\bw_h,\cero]$, gives $\bdiv(\bsigma_{1,h}+\bsigma_{2,h}) = \varrho\,\dot{\bv}_h - \varrho\,\bsP_h^{\bV}\bb$ in $\bV_h$. Together with the (weak) identity $\bdiv(\bsigma_1+\bsigma_2)=\varrho\,\dot{\bv}-\varrho\,\bb$ satisfied by the exact solution, this yields
\begin{equation}\label{eq:error-estimate-div-proof-1}
\bdiv(\be_{\bsigma_1}+\be_{\bsigma_2}) = \varrho\,\dot{\be}_{\bv} - \varrho\,(\mathrm{I}-\bsP_h^{\bV})\bb \qin \bL^2(\Omega) \, .
\end{equation}
Analogously, since $\dot{\varphi}_h\in\Phi_h$ and $\vdiv(\bZ_h)=\Phi_h$ (cf.~{\bf (H3)}), the fifth row of~\eqref{eq:semidiscrete} gives $\vdiv\,\bz_h = \dot{\varphi}_h - \sP_h^{\Phi}J_h$ in $\Phi_h$, whereas the exact solution satisfies $\vdiv\,\bz = \dot{\varphi}-J$ weakly, whence
\begin{equation}\label{eq:error-estimate-div-proof-2}
\vdiv\,\be_{\bz} = \dot{e}_{\varphi} - (\mathrm{I}-\sP_h^{\Phi})J - \sP_h^{\Phi}(J-J_h) \qin \LL^2(\Omega) \, .
\end{equation}
Regarding the last term in~\eqref{eq:error-estimate-div-proof-2}, we note that, by the Lipschitz continuity of $f$ (cf.~\eqref{eq:assumption-f}), the stability of $\sP_h^{\Phi}$ in $\LL^2(\Omega)$, and the bound $\|\tr\btau\|_{\LL^2(\Omega)}\leq\sqrt{d}\,\|\btau\|_{\bbL^2(\Omega)}$, we have
\begin{equation*}
\|\sP_h^{\Phi}(J-J_h)\|_{\LL^2(\Omega)} \leq c\,\big( \|e_{\varphi}\|_{\LL^2(\Omega)}+\|e_{\theta}\|_{\LL^2(\Omega)}+\|\be_{\bsigma_1}\|_{\bbL^2(\Omega)} \big)\, ,
\end{equation*}
so that Theorem~\ref{thm:error-estimate} bounds its $\LL^2(0,t;\LL^2(\Omega))$-norm by $c\,\Xi(t)$. Consequently, it only remains to estimate $\dot{\be}_{\bv}$ and $\dot{e}_{\varphi}$ in the $\LL^2(0,t;\bullet)$-norms in~\eqref{eq:error-estimate-div-proof-1}--\eqref{eq:error-estimate-div-proof-2}.

To this end we differentiate the system~\eqref{eq:discrete-error-eqs} with respect to $t$, which is legitimate by virtue of the assumptions {\rm (i)}--{\rm (iii)}: hypothesis~{\rm (iii)} provides the time derivatives of the exact solution entering the projection errors, while~{\rm (ii)} guarantees, through the differentiability in time of the right-hand side of~\eqref{eq:semidiscrete-wp-proof-2}, that the semidiscrete solution is twice differentiable in time. Since the projection operators are time-independent, they commute with $\partial_t$, and therefore the differentiated system retains exactly the structure of~\eqref{eq:discrete-error-eqs}, with each error replaced by its time derivative; in particular, its third row again has a vanishing right-hand side, which is what makes the argument of Theorem~\ref{thm:error-estimate} applicable verbatim. The only new contributions are those coming from the nonlinear terms, which are handled by means of the identity
\begin{equation*}
\partial_t\big[\alpha(\varphi)-\alpha(\varphi_h)\big] = \alpha'(\varphi_h)\,\dot{e}_{\varphi} + \big(\alpha'(\varphi)-\alpha'(\varphi_h)\big)\,\dot{\varphi}\, ,
\end{equation*}
and its analogues for $f$ and $g$. Recalling that $\alpha'$ is bounded and Lipschitz continuous (cf.~\eqref{eq:def-Hill-function}), and using~{\rm (i)} together with $\dot{\varphi},\dot{\theta}\in\LL^\infty(0,t_{\rm end};\LL^\infty(\Omega))$, we obtain
\begin{equation*}
\big\|\partial_t\big[\alpha(\varphi)-\alpha(\varphi_h)\big]\big\|_{\LL^2(\Omega)} \leq c\,\big( \|\dot{e}_{\varphi}\|_{\LL^2(\Omega)} + \|e_{\varphi}\|_{\LL^2(\Omega)} \big)\, ,
\end{equation*}
and similarly for $f$ and $g$, with the additional contributions $\|\dot{e}_{\theta}\|_{\LL^2(\Omega)}+\|e_{\theta}\|_{\LL^2(\Omega)}$. These are first-order terms of the same nature as those treated in~\eqref{eq:error-estimate-proof-3}, together with zeroth-order terms which Theorem~\ref{thm:error-estimate} controls by $\Xi(t)$.

We may then test the differentiated system with the time derivatives of the discrete errors and repeat the computations leading to~\eqref{eq:error-estimate-proof-8}, thus arriving at
\begin{equation}\label{eq:error-estimate-div-proof-3}
\|\sE_h^{\sf{e}\,,(1)}\|_{\LL^\infty(0,t)} + \|\sD_h^{\sf{e}\,,(1)}\|_{\LL^1(0,t)} \leq c\, \Big( \sE_h^{\sf{e}\,,(1)}(0) + \Xi(t)^2 + \Xi_1(t)^2 \Big)\, ,
\end{equation}
where $\Xi_1(t)^2$ collects the consistency terms and $\Xi(t)^2$ accounts for the zeroth-order contributions identified above. In particular, $\dot{\be}^h_{\bv}$ and $\dot{e}^h_{\varphi}$ are bounded in the $\LL^2(0,t;\bullet)$-norms by the right-hand side of~\eqref{eq:error-estimate-div-proof-3}, and so are $\dot{\be}_{\bv}$ and $\dot{e}_{\varphi}$, upon adding the projection errors $\dot{\be}^I_{\bv}$ and $\dot{e}^I_{\varphi}$, both of which are part of $\Xi_1(t)$.

The assertion~\eqref{eq:error-estimate-div} follows by taking $\LL^2(0,t;\bullet)$-norms in~\eqref{eq:error-estimate-div-proof-1} and~\eqref{eq:error-estimate-div-proof-2} and inserting the bounds just obtained.
\end{proof}

\begin{remark}\label{rem:initial-time-derivative}
The quantity $\sE_h^{\sf{e}\,,(1)}(0)$ is not a datum of the problem: the initial values $\dot{\be}^h_{\bsigma_1}(0)$, $\dot{\be}^h_{\bv}(0)$, $\dot{e}^h_{\varphi}(0)$ and $\dot{e}^h_{\theta}(0)$ are obtained by evaluating~\eqref{eq:mixed-weak} and~\eqref{eq:semidiscrete} at $t=0$. It is of approximation order provided the discrete initial data are chosen compatibly with the equations at $t=0$.
\end{remark}

\section{Specific finite element spaces}
\label{sec:FE-examples}

\subsection{Stable families of finite element spaces and convergence rates}

We now exhibit families of finite element spaces satisfying {\bf (H1)}--{\bf (H5)}, and translate the abstract estimates of Section~\ref{sec:errors} into rates of convergence. We begin by setting the local spaces to be employed. Given an integer $l\geq 0$ and an element $K\in\cT_h$, we denote by $\mathrm{P}_l(K)$ the space of polynomials of degree at most $l$ on $K$, and by $\mathbf{P}_l(K)$ and $\mathbb{P}_l(K)$ its vector- and tensor-valued counterparts. We denote by $\wt{\PP}_l(K)$ the space of homogeneous polynomials of degree $l$ on $K$, and we let
\begin{equation*}
\mathbf{RT}_l(K) := \bP_l(K) \oplus \wt{\PP}_l(K)\,\bx,
\end{equation*}
be the local Raviart--Thomas space, where $\bx$ is the position vector in $\R^d$. Its tensorial counterpart $\mathbb{RT}_l(K)$ consists of the tensors whose rows belong to $\mathbf{RT}_l(K)$. We let $b_K$ be the bubble function on $K$, which is given by the product of its $d+1$ barycentric coordinates. The local bubble space of order $l$ is then set as 
\begin{equation*}
\bB_l(K) := 
\begin{cases}
\bcurl(b_K\, \PP_l(K)) &\text{if }d=2\, , \\
\bcurl(b_K\, \bP_l(K)) &\text{if }d=3\, ,
\end{cases}
\end{equation*}
where the curl operators are defined as $\bcurl(v) := \big( \frac{\partial v}{\partial x_2}, -\frac{\partial v}{\partial x_1}\big)$ for $v:K\to \R$ (if $d=2$), and $\bcurl(\bv) := \nabla \times \bv$ for $\bv:K\to \R^3$ (if $d=3$). Finally, $\mathbb{B}_l(K)$ denotes the space of tensor functions in which each row belongs to $\bB_l(K)$.

\medskip
\noindent\textbf{The viscoelastic block.} Given $k\geq 0$, for the stress, the velocity and the spin we consider
\begin{gather*}
\bbS_h := \big\{ \btau \in \bbH_\star(\bdiv;\Omega): \ \btau|_K \in \bbP_{k+1}(K) \ \ \forall\, K \in \cT_h \big\}\, , \\[0.8ex]
\bbS_h^{\rm P} := \big\{ \btau \in \bbH_\star(\bdiv;\Omega): \ \btau|_K \in \mathbb{RT}_{k}(K)\oplus\mathbb{B}_{k}(K) \ \ \forall\, K \in \cT_h \big\}\, , \\[0.8ex]
\bV_h := \big\{ \bw \in \bL^2(\Omega): \ \bw|_K \in \bP_{k}(K) \ \ \forall\, K \in \cT_h \big\}\, , \\[0.8ex]
\bbK_h^{\rm c} := \big\{ \eeta \in \bbL^2_{\sk}(\Omega)\cap C^0(\overline\Omega)^{d\times d}: \ \eeta|_K \in \bbP_{k+1}(K) \ \ \forall\, K \in \cT_h \big\}\, , \\[0.8ex]
\bbK_h^{\rm d} := \big\{ \eeta \in \bbL^2_{\sk}(\Omega): \ \eeta|_K \in \bbP_{k}(K) \ \ \forall\, K \in \cT_h \big\}\, ,
\end{gather*}
where we note that $\bbS_h$ corresponds to the tensorial version of the $\mathrm{BDM}_{k+1}$ space. The following three choices are admissible:
\begin{enumerate}
\item[\rm (a)] The triplet $(\bbS_h,\bV_h,\bbK_h^{\rm d})$, which is the Arnold--Falk--Winther family~\cite{afw07}.
\item[\rm (b)] The triplet $(\bbS_h,\bV_h,\bbK_h^{\rm c})$, which is the Farhloul--Fortin family introduced for two dimensions in~\cite{farhloul97} and extended to three dimensions in~\cite{bbf09}.
\item[\rm (c)] The triplet $(\bbS_h^{\rm P},\bV_h,\bbK_h^{\rm c})$, which is the PEERS element~\cite{peers,Lonsing-Verfurth}.
\end{enumerate}
All of them satisfy {\bf (H2)} and the inf-sup condition of {\bf (H1)} with $\beta_{\rd}$ independent of $h$. We refer to~\cite{boffi13} and to the survey~\cite{lee16-ACM} for a unified treatment of these and other weakly symmetric elements. Moreover, for each of them, there exists an interpolation operator $\bPi_h^e$ satisfying \eqref{eq:commuting-stress} and, in addition, the weak symmetry property \eqref{eq:weak-symmetry-property-interpolation-operator}, whose existence was established for general classes of weakly symmetric stress elements in \cite{arnoldlee14}; hypothesis {\bf (H4)} is therefore fulfilled.

\begin{remark}\label{rem:FF-vs-PEERS}
The Farhloul--Fortin (or Boffi--Brezzi--Fortin) family {\rm (b)} may be regarded as an improvement of PEERS: the two elements share the same structure, but the former approximates the stress with one additional order of accuracy. This is what makes superconvergence of the displacement available, through a post-processing in the spirit of~\cite{stenberg88}, without invoking full elliptic regularity. We will discuss this in detail in Section~\ref{sec:superconvergence}.
\end{remark}

\medskip
\noindent\textbf{The reaction-diffusion block.} For the diffusive flux, the calcium concentration, and the amount of iIT, we take
\begin{gather*}
\bZ_h^{\mathrm{BDM}} := \big\{ \br \in \bH_\star(\vdiv;\Omega): \ \br|_K \in \bP_{k+1}(K) \ \ \forall\, K\in\cT_h \big\}\, , \\[0.8ex]
\bZ_h^{\mathrm{RT}} := \big\{ \br \in \bH_\star(\vdiv;\Omega): \ \br|_K \in \mathbf{RT}_{k}(K) \quad \forall\, K\in\cT_h \big\}\, , \\[0.8ex]
\Phi_h = \Theta_h := \big\{ \psi \in \LL^2(\Omega): \ \psi|_K \in \PP_{k}(K) \ \ \forall\, K\in\cT_h \big\}\, .
\end{gather*}
The two possible options are $(\bZ_h,\Phi_h,\Theta_h)$ with $\bZ_h=\bZ_h^{\rm BDM}$ or $\bZ_h^{\rm RT}$. In any case, $\vdiv(\bZ_h)=\Phi_h$ (cf. \textbf{(H3)}), and the associated $\mathrm{BDM}$ (or Raviart--Thomas) interpolation operator $\wt\bPi_h$ satisfies the commuting property~\eqref{eq:commuting-flux} (see, e.g.,~\cite{boffi13}), so {\bf (H5)} is fulfilled.

We now collect some approximation properties for the spaces listed above. Since $\bbK_h^{\rm c}$, $\bV_h$, $\Phi_h$ and $\Theta_h$ contain the local spaces $\bbP_{k+1}(K)$, $\bP_k(K)$, $\PP_k(K)$ and $\PP_k(K)$, respectively, the $\LL^2$-projections onto them satisfy the standard estimates
\begin{equation}\label{eq:approximation-properties-L2}
\begin{gathered}
\|\eeta - \bsP_h^{\bbK}\eeta\|_{\bbL^2(\Omega)} \leq C\,h^{\ell}\,|\eeta|_{\bbH^{\ell}(\Omega)} \qquad \forall\, \eeta \in \bbL^2_{\sk}(\Omega)\cap\bbH^{\ell}(\Omega)\, , \quad 0\leq \ell\leq k+2\, , \\[1ex]
\|\bw - \bsP_h^{\bV}\bw\|_{\bL^2(\Omega)} \leq C\,h^{\ell}\,|\bw|_{\bH^{\ell}(\Omega)} \qquad \forall\, \bw \in \bH^{\ell}(\Omega)\, , \quad 0\leq \ell\leq k+1\, , \\[1ex]
\|\psi - \sP_h^{\Phi}\psi\|_{\LL^2(\Omega)} \leq C\,h^{\ell}\,|\psi|_{\HH^{\ell}(\Omega)} \qquad \forall\, \psi \in \HH^{\ell}(\Omega)\, , \quad 0 \leq \ell\leq k+1\, , \\[1ex]
\|\vartheta - \sP_h^{\Theta}\vartheta\|_{\LL^2(\Omega)} \leq C\,h^{\ell}\,|\vartheta|_{\HH^{\ell}(\Omega)} \qquad \forall\, \vartheta \in \HH^{\ell}(\Omega)\, , \quad 0 \leq \ell\leq k+1\, ,
\end{gathered}
\end{equation}
with $C>0$ independent of $h$ (see, e.g.,~\cite{ern-guermond-1,gatica14,boffi13}). In turn, the interpolation operators of {\bf (H4)} and {\bf (H5)} are bounded on $\bbH^{\ell}(\Omega)$ and $\bH^{\ell}(\Omega)$ only for $\ell>s$, with $s\in(1/2,1)$ as in~\eqref{eq:commuting-stress} and~\eqref{eq:commuting-flux}, and they satisfy
\begin{equation}\label{eq:approximation-properties-Hdiv}
\begin{gathered}
\|\btau - \bPi_h^e\btau\|_{\bbL^2(\Omega)} \leq C\,h^{\ell}\,|\btau|_{\bbH^{\ell}(\Omega)} \qquad \forall\, \btau \in \bbH_\star(\bdiv;\Omega)\cap\bbH^{\ell}(\Omega)\, , \\[1ex]
\|\br - \wt\bPi_h\br\|_{\bL^2(\Omega)} \leq C\,h^{\ell}\,|\br|_{\bH^{\ell}(\Omega)} \qquad \forall\, \br \in \bH_\star(\vdiv;\Omega)\cap\bH^{\ell}(\Omega)\, ,
\end{gathered}
\end{equation}
for $s<\ell\leq k+1$ in the case of $\bbS_h^{\rm P}$ and $\bZ_h^{\rm RT}$, and for $s<\ell\leq k+2$ in the case of $\bbS_h$ and $\bZ_h^{\rm BDM}$. We refer to~\cite[Section 2.5.2]{boffi13} for the classical $\mathrm{BDM}$ and $\mathrm{RT}$ interpolants, and to~\cite{arnoldlee14} for the operator $\bPi_h^e$ associated with weakly symmetric stress elements. We stress that, since the divergences $\bdiv\,\btau$ and $\vdiv\,\br$ never enter the consistency functional $\Xi$ of~\eqref{eq:error-consistency-functional}, it is not necessary to require more regularity from them.

The following result establishes the convergence rates associated with the mixed scheme~\eqref{eq:semidiscrete} and the finite element spaces specified above.

\begin{theorem}[Rates of convergence]\label{thm:rates}
Let $k\geq 0$, let the discrete spaces be given by the family {\rm (b)} with either choice of $\bZ_h$ in the reaction-diffusion block, and let the hypotheses of Theorem~\ref{thm:error-estimate} hold. Furthermore, assume that the solution of~\eqref{eq:mixed-weak} satisfies
\begin{gather*}
\bsigma_1 \in \HH^1\big(0,t_{\rm end};\bbH^{k+1}(\Omega)\big)\, , \quad
\bsigma_2, \bomega \in \LL^2\big(0,t_{\rm end};\bbH^{k+1}(\Omega)\big)\, , \quad
\bz \in \LL^2\big(0,t_{\rm end};\bH^{k+1}(\Omega)\big)\, , \\
\bv \in \LL^\infty\big(0,t_{\rm end};\bH^{k+1}(\Omega)\big)\, , \qan
\varphi,\theta \in \LL^\infty\big(0,t_{\rm end};\HH^{k+1}(\Omega)\big)\, .
\end{gather*}
In addition, let the discrete initial datum for the elastic stress be chosen as $\bsigma_{1,h}^0 := \bPi_h^e\big(\cC_1\beps(\bu_0)\big)$. Then there exists $C>0$, independent of $h$, such that, for all $t\in(0,t_{\rm end}]$,
\begin{equation}\label{eq:rates}
\begin{aligned}
&\|\be_{\bsigma_1}\|_{\LL^\infty(0,t;\bbL^2(\Omega))} + \|\be_{\bv}\|_{\LL^\infty(0,t;\bL^2(\Omega))} + \|e_{\varphi}\|_{\LL^\infty(0,t;\LL^2(\Omega))} + \|e_{\theta}\|_{\LL^\infty(0,t;\LL^2(\Omega))} \\
&\qquad + \|\be_{\bsigma_2}\|_{\LL^2(0,t;\bbL^2(\Omega))} + \|\be_{\bz}\|_{\LL^2(0,t;\bL^2(\Omega))} + \|\be_{\bomega}\|_{\LL^2(0,t;\bbL^2(\Omega))} \leq C\,h^{k+1}\, .
\end{aligned}
\end{equation}
\end{theorem}
\begin{proof}
It suffices to insert~\eqref{eq:approximation-properties-L2} with $\ell=k+1$, and~\eqref{eq:approximation-properties-Hdiv} with $\ell=k+1$, into the definition~\eqref{eq:error-consistency-functional} of $\Xi$, which gives $\Xi(t)\leq C\,h^{k+1}$. Note that $\be^h_{\bsigma_1}(0)=\cero$ by the choice of $\bsigma_{1,h}^0$ (cf.~Remark~\ref{rem:initial-error-term}). The assertion then follows from Theorems~\ref{thm:error-estimate} and~\ref{thm:error-estimate-spin}.
\end{proof}

\begin{theorem}[Rates of convergence in the divergence norms]\label{thm:rates-div}
In addition to the hypotheses of Theorem~\ref{thm:rates}, let those of Theorem~\ref{thm:error-estimate-div} hold, and assume that the time derivatives $\dot{\bsigma}_1,\dot{\bsigma}_2,\dot{\bomega},\dot{\bz},\dot{\bv},\dot{\varphi},\dot{\theta}$ enjoy the regularity required from the corresponding unknowns in Theorem~\ref{thm:rates}, that the datum $\bb \in \LL^2(0,t_{\rm end};\bH^{k+1}(\Omega))$, and that $f\in C^{k+1}(\R^2)$ with bounded derivatives up to order $k+1$. Then, with $C>0$ independent of $h$, and for all $t\in(0,t_{\rm end}]$,
\begin{equation}\label{eq:rates-div}
\|\bdiv(\be_{\bsigma_1}+\be_{\bsigma_2})\|_{\LL^2(0,t;\bL^2(\Omega))} + \|\vdiv\,\be_{\bz}\|_{\LL^2(0,t;\LL^2(\Omega))} \leq C\,h^{k+1} \, .
\end{equation}
\end{theorem}
\begin{proof}
The additional regularity of the time derivatives yields $\Xi_1(t)\leq C\,h^{k+1}$, exactly as in the proof of Theorem~\ref{thm:rates}. Moreover, $\bb\in\LL^2(0,t_{\rm end};\bH^{k+1}(\Omega))$ and the assumption on $f$ give $J\in\LL^2(0,t_{\rm end};\HH^{k+1}(\Omega))$, whence~\eqref{eq:approximation-properties-L2} with $\ell=k+1$ bounds the two data approximation terms in~\eqref{eq:error-estimate-div} by $C\,h^{k+1}$. The assertion then follows from Theorem~\ref{thm:error-estimate-div}.
\end{proof}

\begin{remark}\label{rem:rates-other-families}
With the Farhloul--Fortin family, the rate in~\eqref{eq:rates}--\eqref{eq:rates-div} is dictated by the velocity and the concentrations, which are approximated by discontinuous piecewise polynomials of degree $k$, whereas the stress, and the flux if $\bZ_h = \bZ_h^{\rm BDM}$, are approximated with one extra order. This extra accuracy is not visible in~\eqref{eq:rates} because of the coupling, but the one gained by the stress is the reason why this family admits a post-processing yielding superconvergence of the displacement (cf.~Remark~\ref{rem:FF-vs-PEERS} and Section~\ref{sec:superconvergence}). Estimates~\eqref{eq:rates}--\eqref{eq:rates-div} hold verbatim for the family {\rm (a)}, the only modification being that the spin tensor is then approximated to order $h^{k+1}$ instead of $h^{k+2}$, which does not affect the resulting rate, and for the PEERS element {\rm (c)}, in which case the stress is approximated to order $h^{k+1}$.
\end{remark}

\subsection{Displacement superconvergence for specific finite element pairs}
\label{sec:superconvergence}
In this subsection, we prove a superconvergence estimate for the projected displacement and its local post-processing for special combinations of finite elements.

We first define the displacement reconstructed from the discrete velocity by
\begin{equation}\label{eq:FF-RT-displacement-recovery}
\bu_h(t):=\bu_{0,h}+\int_0^t\bv_h(s)\ds,
\qquad \bu_{0,h}:=\bsP_h^{\bV}\bu_0 \, .
\end{equation}
Because $\bsP_h^{\bV}$ is independent of time and
$\dot{\bu}=\bv$, we have the exact identity
\begin{equation*}
    \bsP_h^{\bV}\bu(t)-\bu_h(t) =\int_0^t\big(\bsP_h^{\bV}\bv(s)-\bv_h(s)\big)\ds.
\end{equation*}
For the local reconstruction, introduce
\begin{equation}\label{eq:FF-RT-integrated-spin}
\brho(t):=\sk(\bnabla\bu_0)+\int_0^t\bomega(s)\ds \, ,
\qquad
\brho_h(t):=\bsP_h^{\bbK}\, \sk(\bnabla\bu_0)+\int_0^t\bomega_h(s)\ds \, ,
\end{equation}
where $\sk(\bnabla\bu_0) = \tfrac{1}{2} \, (\bnabla\bu_0-\bnabla\bu_0^\top)$.
Then, at the continuous level one has $\bnabla\bu=\cA_1\bsigma_1+\brho$. At the discrete level, we define, for every $K\in\cT_h$, the reconstruction $\bu_h^\star|_K\in\bP_{k+2}(K)$ as the solution of
\begin{equation}\label{eq:FF-RT-displacement-postprocessing}
\begin{array}{rcll}
(\bnabla\bu_h^\star,\bnabla\bw)_K 
&=&(\cA_1\bsigma_{1,h}+\brho_h,\bnabla\bw)_K
&\forall\,\bw\in\bP_{k+2}(K) \, ,\\[1ex]
(\bu_h^\star,\bw)_K &=& (\bu_h,\bw)_K
&\forall\,\bw\in\bP_0(K) \, .
\end{array}
\end{equation}

\begin{theorem}[Superconvergence of displacement]\label{thm:displacement-superconvergence}
Let the viscoelasticity spaces be the Farhloul--Fortin spaces with $k\ge 0$ and let the diffusion finite element spaces be $(\bZ_h^{\rm RT}, \Phi_h, \Theta_h)$ with $k+1$. We choose the initial data so that
\begin{equation*}
\bsigma_{1,h}^0=\bPi_h^e\bsigma_1(0) \, ,\qquad
\bv_h^0=\bsP_h^{\bV}\bv_0\, ,\qquad
\varphi_h^0=\sP_h^{\Phi}\varphi_0\, ,\qquad
\theta_h^0=\sP_h^{\Theta}\theta_0\, .
\end{equation*} 
Then,
\begin{equation}\label{eq:FF-RT-displacement-postprocessing-rate}
\|\bu - \bu_h^\star \|_{\LL^{\infty}(0,t;\bL^2(\Omega))} \leq C h^{k+2} \, , 
\end{equation}
with $C>0$ depending on solution regularities.
\end{theorem}
\begin{proof}
For the FF--RT pair, the approximation operators in~\eqref{eq:error-splitting-def-1} satisfy, for a.e. $s\in(0,t_{\rm end})$,
\begin{equation}\label{eq:FF-RT-approximation-hierarchy}
\begin{aligned}
\norm{\dot{\be}^I_{\bsigma_1}}_{\bbL^2(\Omega)}
+\norm{\be^I_{\bsigma_1}}_{\bbL^2(\Omega)}
+\norm{\be^I_{\bsigma_2}}_{\bbL^2(\Omega)}
+\norm{\be^I_{\bomega}}_{\bbL^2(\Omega)}
+\norm{\be^I_{\bz}}_{\bL^2(\Omega)}
+\norm{e^I_{\varphi}}_{\LL^2(\Omega)}
+\norm{e^I_{\theta}}_{\LL^2(\Omega)}
\leq C h^{k+2}.
\end{aligned}
\end{equation}
The same estimate is assumed for the time derivatives. 

Note that $\be^h_{\bsigma_1}(0)=\cero$, $\be^h_{\bv}(0)=\cero$,
$e^h_{\varphi}(0)=0$, and $e^h_{\theta}(0)=0$. By inserting the error equations
~\eqref{eq:discrete-error-eqs} into the energy identity
~\eqref{eq:error-estimate-proof-1}, the terms independent of the nonlinear
residual are bounded as in~\eqref{eq:error-estimate-proof-2}. The Lipschitz bounds~\eqref{eq:assumption-f}--\eqref{eq:assumption-g} and~\eqref{eq:alpha-Lipschitz} give
\begin{align*}
\norm{\alpha(\varphi)-\alpha(\varphi_h)}_{\LL^2(\Omega)}
&\leq L_\alpha\, \big(\|e^I_{\varphi}\|_{\LL^2(\Omega)}
+\| e^h_{\varphi}\|_{\LL^2(\Omega)}\big) \, ,\\
\norm{f(\varphi,\theta)-f(\varphi_h,\theta_h)}_{\LL^2(\Omega)}
&\leq L_{f1} \, \big(\| e^I_{\varphi}\|_{\LL^2(\Omega)}
+\|e^h_{\varphi}\|_{\LL^2(\Omega)}\big) +L_{f2} \, \big(\|e^I_{\theta}\|_{\LL^2(\Omega)}
+\|e^h_{\theta}\|_{\LL^2(\Omega)}\big) \, ,
\end{align*}
and an analogous bound holds for $g$. Using these inequalities in~\eqref{eq:error-estimate-proof-3}, and then using Young's inequality in~\eqref{eq:error-estimate-proof-4}, gives constants $C_0,C_1>0$ independent
of $h$ such that
\begin{equation}\label{eq:FF-RT-special-energy-inequality}
\partial_t\sE_h^{\sf e}+\sD_h^{\sf e}
\leq C_0 \, \sE_h^{\sf e} + C_1\, \big(F(t)^2+G(t)^2\big)\, .
\end{equation}
By~\eqref{eq:FF-RT-approximation-hierarchy} and the definition of $F$ and
$G$ in~\eqref{eq:error-estimate-proof-FG},
\begin{equation*}
\norm{F}_{\LL^2(0,t)}+\norm{G}_{\LL^2(0,t)}\leq C h^{k+2} \, .
\end{equation*}
Consequently, integrating~\eqref{eq:FF-RT-special-energy-inequality} and applying
Gr\"onwall's inequality yields
\begin{equation*}
\norm{\sE_h^{\sf e}}_{\LL^\infty(0,t)}
+\norm{\sD_h^{\sf e}}_{\LL^1(0,t)}
\leq C h^{2k+4} \, .
\end{equation*}
In particular, by~\eqref{eq:error-estimate-proof-control},
\begin{equation}\label{eq:FF-RT-velocity-discrete-error}
\|\be^h_{\bv}\|_{\LL^\infty(0,t;\bL^2(\Omega))}
\leq C h^{k+2} \, .
\end{equation}
Since $\be^h_{\bv}=\bsP_h^{\bV}\bv-\bv_h$, this is the required
velocity supercloseness estimate. Notice that the usual approximation error
$\bv-\bsP_h^{\bV}\bv$ is not part of~\eqref{eq:FF-RT-velocity-discrete-error}.

It follows from the Cauchy--Schwarz inequality in time and~\eqref{eq:FF-RT-velocity-discrete-error} that
\begin{equation}\label{eq:FF-RT-displacement-projected-error}
\norm{\bsP_h^{\bV}\bu-\bu_h}_{\LL^\infty(0,t;\bL^2(\Omega))}
\leq t^{1/2}\, \|\be^h_{\bv}\|_{\LL^2(0,t;\bL^2(\Omega))}
\leq C h^{k+2} \, .
\end{equation}
Thus the projected displacement converges one order faster than the assumed
best approximation order $h^{k+1}$ of $\bu$ in $\bV_h$.
A scaling argument on the reference element followed by the triangle inequality (see \cite{Lee-thesis} for details), gives
\begin{equation}\label{eq:FF-RT-local-postprocessing-estimate}
\begin{aligned}
\norm{\bu(t)-\bu_h^\star(t)}_{\bL^2(K)}
\leq C\Big(&\norm{\bsP_h^{\bV}\bu(t)-\bu_h(t)}_{\bL^2(K)}
+h_K\norm{\bsigma_1(t)-\bsigma_{1,h}(t)}_{\bbL^2(K)}\\
&+h_K\norm{\brho(t)-\brho_h(t)}_{\bbL^2(K)}
+h_K^{k+2}|\bu(t)|_{\bH^{k+2}(K)}\Big) \, .
\end{aligned}
\end{equation}
The first term is $O(h^{k+2})$ by~\eqref{eq:FF-RT-displacement-projected-error}.
The stress term is no larger than $O(h^{k+2})$ by~\eqref{eq:FF-RT-approximation-hierarchy} and~\eqref{eq:error-estimate-proof-7}. $\norm{\brho(t)-\brho_h(t)}_{\bbL^2}$ is estimated by $\norm{\brho(0) - \brho_h(0)}_{\bbL^2} \le Ch^{k+2}\norm{\nabla \bu_0}_{\bH^{k+2}}$ and an estimate of the time-integration of $\be^I_{\bomega}$. Recalling the first equation in \eqref{eq:discrete-error-eqs} and adopting an argument similar to the proof of Theorem~\ref{thm:error-estimate-spin}, one can show that $\norm{\brho(t)-\brho_h(t)}_{\bbL^2} \le Ch^{k+2}$. Therefore, \eqref{eq:FF-RT-displacement-postprocessing-rate} follows. 
\end{proof}
\begin{remark}
    The polynomial space in \eqref{eq:FF-RT-displacement-postprocessing} and the estimates of $\bsigma_1$ and $\brho$ in \eqref{eq:FF-RT-local-postprocessing-estimate} allow $Ch^{k+3}$ approximation but the final estimate in \eqref{eq:FF-RT-local-postprocessing-estimate} is limited by $Ch^{k+2}$ due to \eqref{eq:FF-RT-displacement-projected-error}.
\end{remark}

\section{Numerical examples}\label{sec:results}
This section reports on the verification of the theoretical convergence rates and assessment of performance of the formulation. All simulations are implemented using the \texttt{FEniCS} finite element library \cite{Alnaes2015}. 
\subsection{Corroboration of convergence against smooth solutions}\label{sec:corroboration}
We assess the spatial convergence of the fully coupled viscoelasticity-reaction-diffusion problem using manufactured solutions on the unit square $\Omega=(0,1)^2$. The final time is $t_{\mathrm{end}}=0.5$ (adimensional), and backward Euler's time stepping is used to approximate $\dot{\bsigma}_1,\dot{\bv},\dot{\varphi},\dot{\theta}$, with the constant time step $\Delta t = 10^{-3}$. Since the same time step is used for all mesh levels, and it is sufficiently small, the observed rates effectively measure the spatial discretisation error. The sub-boundaries are specified as follows: $\Gamma^{\bu}$ consists of the left and bottom sides of the unit squares, and $\Gamma^{\bsigma}$ is the top and right edges.

\begin{table}[t!]
\setlength{\tabcolsep}{2.pt}
\centering
{\small\textbf{Component-wise errors}\\[1mm]
\begin{tabular}{rc|
  cc cc cc cc cc cc cc}
\toprule
DoFs & $h$
& \multicolumn{2}{c}{$\bsigma_1$}
& \multicolumn{2}{c}{$\bsigma_2$}
& \multicolumn{2}{c}{$\bv$}
& \multicolumn{2}{c}{$\bomega$}
& \multicolumn{2}{c}{$\bz$}
& \multicolumn{2}{c}{$\varphi$}
& \multicolumn{2}{c}{$\theta$} \\
\cmidrule(lr){3-16}
& & error & rate & error & rate & error & rate
& error & rate & error & rate & error & rate & error & rate \\
\midrule
\multicolumn{16}{c}{AFW family with $k=0$} \\
\midrule
  200 & 0.707 & 5.9e+02 & -- & 9.9e+03 & -- & 1.8e-01 & -- & 6.4e-01 & -- & 2.4e+00 & -- & 1.3e-01 & -- & 6.1e-02 & -- \\
   720 & 0.354 & 3.2e+02 & 0.89 & 5.1e+03 & 0.94 & 6.3e-02 & 1.48 & 1.8e-01 & 1.82 & 1.3e+00 & 0.92 & 6.7e-02 & 0.92 & 3.2e-02 & 0.92\\
  2720 & 0.177 & 1.6e+02 & 0.96 & 2.6e+03 & 0.99 & 2.8e-02 & 1.16 & 6.7e-02 & 1.42 & 6.4e-01 & 0.98 & 3.3e-02 & 1.02 & 1.6e-02 & 0.98\\
 10560 & 0.088 & 8.3e+01 & 0.99 & 1.3e+03 & 1.00 & 1.3e-02 & 1.07 & 3.0e-02 & 1.15 & 3.2e-01 & 1.00 & 1.6e-02 & 1.01 & 8.2e-03 & 0.99\\
 41600 & 0.044 & 4.1e+01 & 0.99 & 6.5e+02 & 1.00 & 6.6e-03 & 1.03 & 1.5e-02 & 1.05 & 1.6e-01 & 1.00 & 8.2e-03 & 1.00 & 4.2e-03 & 0.97\\
165120 & 0.022 & 2.1e+01 & 1.00 & 3.2e+02 & 1.00 & 3.3e-03 & 1.01 & 7.3e-03 & 1.01 & 8.1e-02 & 1.00 & 4.1e-03 & 1.00 & 2.1e-03 & 0.99\\
\midrule 
\multicolumn{16}{c}{AFW family with $k=1$} \\
\midrule 
  480 & 0.707 & 2.0e+02 & -- & 3.0e+03 & -- & 6.1e-02 & -- & 2.2e-01 & -- & 7.3e-01 & -- & 4.3e-02 & -- & 1.9e-02 & -- \\
  1800 & 0.354 & 5.1e+01 & 2.01 & 7.8e+02 & 1.93 & 1.2e-02 & 2.35 & 4.0e-02 & 2.42 & 1.9e-01 & 1.92 & 1.0e-02 & 2.08 & 4.9e-03 & 1.92\\
  6960 & 0.177 & 1.2e+01 & 2.02 & 2.0e+02 & 1.98 & 2.3e-03 & 2.36 & 6.7e-03 & 2.59 & 4.9e-02 & 1.98 & 2.5e-03 & 2.03 & 1.3e-03 & 1.96\\
 27360 & 0.088 & 3.1e+00 & 2.01 & 4.9e+01 & 2.00 & 5.3e-04 & 2.14 & 1.3e-03 & 2.33 & 1.2e-02 & 2.00 & 6.2e-04 & 2.01 & 3.6e-04 & 1.97\\
108480 & 0.044 & 9.3e-01 & 1.74 & 1.2e+01 & 2.00 & 1.3e-04 & 2.04 & 3.1e-04 & 2.12 & 3.1e-03 & 2.00 & 1.6e-04 & 2.00 & 9.1e-05 & 1.98\\
432000 & 0.022 & 2.4e-01 & 1.92 & 3.1e+00 & 2.00 & 3.2e-05 & 2.01 & 7.5e-05 & 2.04 & 7.7e-04 & 1.99 & 4.0e-05 & 1.96 & 2.3e-05 & 1.97\\
\midrule 
\multicolumn{16}{c}{FF family with $k=0$} \\
\midrule 
 201 & 0.707 & 5.9e+02 & -- & 9.9e+03 & -- & 2.1e-01 & -- & 8.8e-01 & -- & 2.4e+00 & -- & 1.3e-01 & -- & 6.1e-02 & -- \\
   713 & 0.354 & 3.2e+02 & 0.89 & 5.2e+03 & 0.94 & 6.5e-02 & 1.71 & 1.7e-01 & 2.34 & 1.3e+00 & 0.92 & 6.7e-02 & 0.92 & 3.2e-02 & 0.92\\
  2673 & 0.177 & 1.6e+02 & 0.96 & 2.6e+03 & 0.99 & 2.8e-02 & 1.19 & 3.6e-02 & 2.27 & 6.4e-01 & 0.98 & 3.3e-02 & 1.02 & 1.6e-02 & 0.98\\
 10337 & 0.088 & 8.3e+01 & 0.99 & 1.3e+03 & 1.00 & 1.3e-02 & 1.08 & 9.0e-03 & 2.01 & 3.2e-01 & 1.00 & 1.6e-02 & 1.01 & 8.2e-03 & 0.99\\
 40641 & 0.044 & 4.1e+01 & 0.99 & 6.5e+02 & 1.00 & 6.6e-03 & 1.03 & 2.3e-03 & 1.98 & 1.6e-01 & 1.00 & 8.2e-03 & 1.00 & 4.1e-03 & 1.00\\
161153 & 0.022 & 2.1e+01 & 1.00 & 3.3e+02 & 1.00 & 3.3e-03 & 1.01 & 5.7e-04 & 1.99 & 8.1e-02 & 1.00 & 4.1e-03 & 1.00 & 2.1e-03 & 0.99\\
\midrule 
\multicolumn{16}{c}{PEERS family with $k=0$} \\
\midrule 
   153 & 0.707 & 5.4e+02 & -- & 1.0e+04 & -- & 7.0e-01 & -- & 6.2e+00 & -- & 2.5e+00 & -- & 1.2e-01 & -- & 6.1e-02 & --\\
   561 & 0.354 & 2.9e+02 & 0.91 & 5.4e+03 & 0.92 & 1.9e-01 & 1.87 & 2.0e+00 & 1.61 & 1.3e+00 & 0.92 & 6.4e-02 & 0.93 & 3.2e-02 & 0.92\\
  2145 & 0.177 & 1.8e+02 & 0.69 & 2.8e+03 & 0.98 & 5.1e-02 & 1.91 & 5.3e-01 & 1.93 & 6.6e-01 & 0.98 & 3.3e-02 & 0.98 & 1.6e-02 & 0.98\\
  8385 & 0.088 & 1.0e+02 & 0.83 & 1.4e+03 & 1.00 & 1.7e-02 & 1.56 & 1.7e-01 & 1.69 & 3.3e-01 & 1.00 & 1.6e-02 & 1.00 & 8.2e-03 & 0.99\\
 33153 & 0.044 & 5.1e+01 & 0.94 & 6.9e+02 & 1.00 & 7.2e-03 & 1.28 & 5.8e-02 & 1.51 & 1.6e-01 & 1.00 & 8.2e-03 & 1.00 & 4.1e-03 & 1.00\\
131841 & 0.022 & 2.6e+01 & 0.97 & 3.5e+02 & 1.00 & 3.3e-03 & 1.10 & 2.1e-02 & 1.50 & 8.2e-02 & 1.00 & 4.1e-03 & 1.00 & 2.1e-03 & 0.99\\
\bottomrule
\end{tabular}}
\caption{Convergence history (individual errors in the $L^\infty(0,t_{\rm end}; V)$ norm and experimental order of convergence) for the proposed method using different finite element families (for families (b) and (c) we only show the lowest-order case) and a constant time step $\Delta t = 10^{-3}$.}
\label{tab:convergence}
\end{table}

We prescribe the exact primal unknowns (displacement and reaction-diffusion variables) as
\begin{gather*}
\bu(x,y,t)=\frac{e^{-t}}{5}
\begin{pmatrix}
\cos(\pi(x+y)) \\ \sin(\pi(x-y))
\end{pmatrix},\quad 
\varphi(x,y,t)=t\cos(\pi x)\sin(\pi y),
\quad  
\theta(x,y,t)=t^2\sin(\pi x)\cos(\pi y),
\end{gather*}
and we note that they have different time scalings. 
From these smooth fields we construct the exact   velocity  $\bv=\dot{\bu}$, and exact elastic and viscous stresses, spin and diffusive flux using 
\eqref{eq:reconstructed-mixed-variables}. 
The nonlinear coupling coefficient and reaction terms are
\[
\alpha(\varphi)=\frac{\alpha_0\varphi^2}{1+\varphi^2}, \quad 
f(\varphi,\theta)
=\gamma_0\theta\frac{\gamma_1+\varphi}{1+\varphi}
-\gamma_2\frac{\varphi}{\gamma_3+\varphi},
\quad
g(\varphi,\theta)=\frac{1}{1+\varphi^2}-\theta.
\]
The source terms and boundary data are obtained by substituting these exact fields into the strong equations.  
The viscoelasticity equations are discretised using the three finite element families described in Section \ref{sec:FE-examples}. 
For this first set of tests, the adimensional material and model parameters are
\begin{gather*}
\rho=1,\quad \alpha_0=1,\quad \beta=1,\quad D=1,\quad 
\gamma_0=1,\quad \gamma_1=0.5,\quad
\gamma_2=0.1,\quad \gamma_3=1,
\end{gather*}
with elastic and viscoelastic parameters
\[
E_1=100,\quad \nu_1=0.45,\quad 
E_2=1000,\quad \nu_2=0.475.
\]
The Lam\'e coefficients are computed from
\[
\lambda_i=\frac{E_i\nu_i}{(1+\nu_i)(1-2\nu_i)},
\qquad
\mu_i=\frac{E_i}{2(1+\nu_i)},
\qquad i=1,2.
\]

\begin{table}[t!]
\centering 
{\small \textbf{Global diagnostics}\\[1mm]
\begin{minipage}{0.48\textwidth}
\begin{tabular}{ccccc}
\toprule
$e_{\rm tot}$ & rate & mom. loss & mass loss & avg. it. \\
\midrule
\multicolumn{5}{c}{AFW family with $k=0$} \\
\midrule
1.05e+04 & -- & 7.45e-12 & 5.25e-09 & 3.00 \\
 5.47e+03 & 0.94 & 2.40e-12 & 1.74e-08 & 2.67\\
 2.76e+03 & 0.99 & 8.24e-13 & 5.71e-09 & 2.00\\
 1.38e+03 & 1.00 & 4.22e-13 & 1.52e-09 & 2.00\\
 6.92e+02 & 1.00 & 2.58e-13 & 3.87e-10 & 2.00\\
 3.46e+02 & 1.00 & 1.65e-13 & 9.71e-11 & 2.00\\
\midrule
\multicolumn{5}{c}{AFW family with $k=1$} \\
\midrule
 3.17e+03 & -- & 1.47e-11 & 1.98e-08 & 2.00 \\
 8.30e+02 & 1.93 & 7.46e-12 & 7.00e-09 & 2.00\\
 2.10e+02 & 1.98 & 4.44e-12 & 1.98e-09 & 2.00\\
 5.26e+01 & 2.00 & 2.29e-12 & 5.13e-10 & 2.00\\
 1.33e+01 & 1.98 & 1.14e-12 & 1.29e-10 & 2.00\\
 3.74e+00 & 1.92 & 5.96e-13 & 3.24e-11 & 2.00\\
\bottomrule
\end{tabular}\end{minipage}\hfill\begin{minipage}{0.48\textwidth}\begin{tabular}{ccccc}
\toprule
$e_{\rm tot}$ & rate & mom. loss & mass loss & avg. it. \\
\midrule
\multicolumn{5}{c}{FF family with $k=0$} \\
\midrule
 1.05e+04 & -- & 7.56e-12 & 5.25e-09 & 2.99 \\
 5.47e+03 & 0.94 & 2.42e-12 & 1.74e-08 & 2.67\\
 2.76e+03 & 0.99 & 7.39e-13 & 5.71e-09 & 2.00\\
 1.38e+03 & 1.00 & 3.79e-13 & 1.52e-09 & 2.00\\
 6.92e+02 & 1.00 & 2.72e-13 & 3.87e-10 & 2.00\\
 3.46e+02 & 1.00 & 1.56e-13 & 9.71e-11 & 2.00\\
\midrule
\multicolumn{5}{c}{PEERS family with $k=0$} \\
\midrule
1.09e+04 & -- & 3.15e-12 & 6.91e-09 & 2.99 \\
 5.73e+03 & 0.92 & 8.99e-13 & 2.76e-09 & 3.00\\
 2.93e+03 & 0.97 & 3.94e-13 & 5.46e-09 & 2.66\\
 1.50e+03 & 0.96 & 1.81e-13 & 1.53e-09 & 2.00\\
 7.71e+02 & 0.96 & 1.06e-13 & 3.87e-10 & 2.00\\
 3.79e+02 & 0.98 & 5.46e-14 & 9.71e-11 & 2.00\\
\bottomrule
\end{tabular}\end{minipage}}
\caption{Total error, maximum momentum and mass balance error in the $\ell^\infty$ norm, and average (over the 500 time steps) number of Newton--Raphson iterations required for convergence. Results associated with the proposed method using different finite element families and a constant time step $\Delta t = 10^{-3}$. The rows  correspond to the same meshes used for Table \ref{tab:convergence}.}
\label{tab:convergence2}
\end{table}

The computations are performed on a sequence of uniformly refined meshes. For each mesh, the nonlinear problem is solved at every time step using Newton--Raphson's method with a relative and absolute tolerance of $10^{-8}$. The reported error for each variable is the maximum-in-time error,
\[
e_q=\max_{0\leq t\leq t_{\mathrm{end}}}\|q(t)-q_h(t)\|_q,
\]
where the norm is the corresponding type of norm. The total error is defined as the sum of the component errors,
\[
e_{\mathrm{tot}}
=e_{\bsigma_1}+e_{\bsigma_2}+e_{\bv}+e_{\bomega}+e_{\bz}+e_\varphi+e_\theta.
\]

For two consecutive meshes with mesh sizes $h_{\mathrm{coarse}}$ and $h_{\mathrm{fine}}$, the experimental order of convergence is computed as
\[
\operatorname{rate}
=
\frac{\log(e_{\mathrm{fine}}/e_{\mathrm{coarse}})}
{\log(h_{\mathrm{fine}}/h_{\mathrm{coarse}})}.
\]
Since $h_{\mathrm{fine}}<h_{\mathrm{coarse}}$, this quantity is positive when the error decreases under refinement. The first mesh has no rate value because no coarser reference mesh is available.

\begin{figure}[t!]
    \centering
    \includegraphics[width=0.245\linewidth]{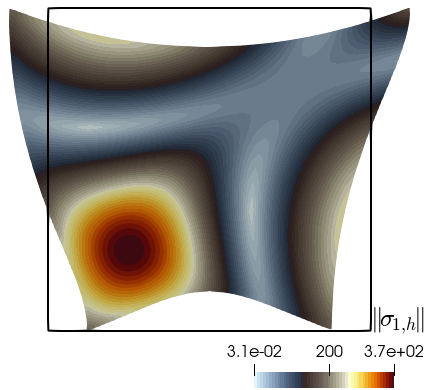}
       \includegraphics[width=0.245\linewidth]{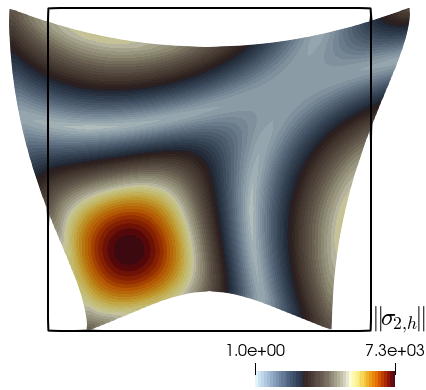}
       \includegraphics[width=0.245\linewidth]{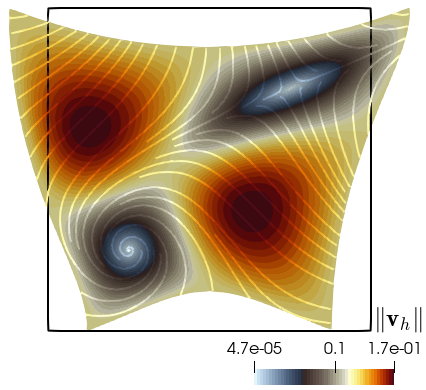}
       \includegraphics[width=0.245\linewidth]{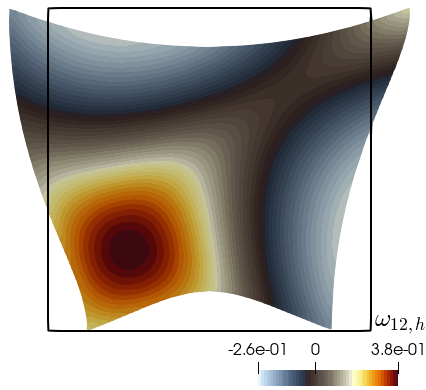}
    \includegraphics[width=0.245\linewidth]{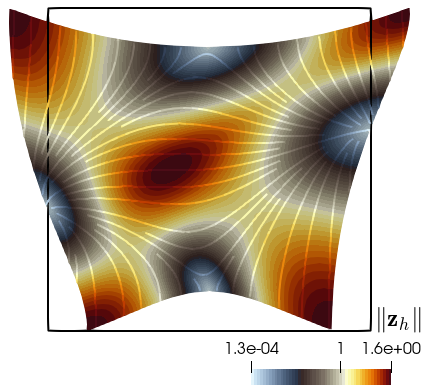}
       \includegraphics[width=0.245\linewidth]{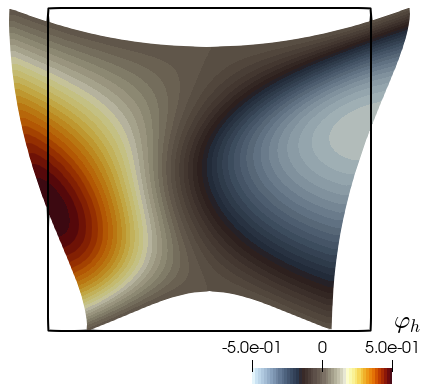}
       \includegraphics[width=0.245\linewidth]{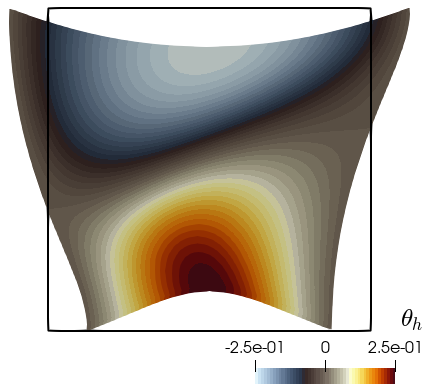}
    \caption{Convergence test against smooth solutions. Sample of approximate fields (from top-left to bottom-right: magnitudes of elastic and modified viscous stresses, magnitude and evenly spaced streamlines of velocity, non-trivial entry of the spin tensor, magnitude and streamlines of diffusive flux, and distribution of calcium and iIT, respectively) plotted on the deformed (according to the displacement $\bu_h$) configuration at the final time.}
    \label{fig:convergence}
\end{figure}

The numerical results are collected in Tables \ref{tab:convergence} and \ref{tab:convergence2}. They show the expected optimal convergence rates for the Arnold--Falk--Winther, Farhloul--Fortin and PEERS  viscoelastic  discretisations (we report lowest and second-lowest polynomial degrees for the first family and only lowest-order for the other two cases).  The component-wise errors and the total error decrease with the rates predicted by the polynomial approximation orders of the corresponding finite element spaces. We observe that the momentum and mass defects ($\ell^\infty$ norm of the assembled residuals of the momentum and mass balances) are small. Note that in the first panel of Table \ref{tab:convergence2} these quantities are not machine precision because, for example, $f(\varphi_h,\theta_h)$ and $\tr(\bsigma_{1,h})$ are not necessarily in $\mathbb{P}_k^{\mathrm{DG}}$. Samples of the approximate solutions obtained with the lowest-order AFW  family are portrayed in Figure \ref{fig:convergence}, showing well-resolved profiles. These plots are rendered on the deformed configuration (using the discrete displacement reconstructed from the discrete velocity as in \eqref{eq:reconstructed-displacement}).

\begin{table}[t]
\centering
\begin{tabular}{rc cc cc cc}
\toprule
DoFs & $h$ & $e(\bv_h)$ & rate & $e(\bu_h)$ & rate & $e(\bu_h^\star)$ & rate \\
\midrule
   201 & 0.707 & 2.12e-01 & --   & 1.01e-01 & --   & 8.98e-02 & --   \\
   713 & 0.354 & 6.45e-02 & 1.71 & 3.71e-02 & 1.44 & 2.15e-02 & 2.06 \\
  2673 & 0.177 & 2.83e-02 & 1.19 & 1.68e-02 & 1.15 & 5.82e-03 & 1.89 \\
 10337 & 0.088 & 1.34e-02 & 1.08 & 8.07e-03 & 1.05 & 1.56e-03 & 1.90 \\
 40641 & 0.044 & 6.58e-03 & 1.03 & 3.99e-03 & 1.02 & 4.02e-04 & 1.98 \\
161153 & 0.022 & 3.27e-03 & 1.01 & 1.99e-03 & 1.00 & 1.01e-04 & 1.99 \\
\bottomrule
\end{tabular}
\caption{Displacement post-processing for the Farhloul--Fortin family
($k=0$): errors and experimental rates of convergence of the velocity
$v_h$, the reconstructed displacement $\bu_h$ (pre-processing), and the
post-processed displacement $\bu_h^\star$. The optimal rate is $k+1=1$; the
superconvergent rate predicted by \eqref{eq:FF-RT-displacement-postprocessing-rate} is $k+2=2$.}
\label{tab:postprocess-FF}
\end{table}

\subsection{Verification of the displacement superconvergence}
\label{subsec:postprocessing}

We now verify numerically the local displacement post-processing of
Section~\ref{sec:superconvergence} and the superconvergence estimate
\eqref{eq:FF-RT-displacement-postprocessing-rate}. The setup is identical to Section \ref{sec:corroboration} (same manufactured
solution, uniform triangulations of the unit square, backward Euler in time with
$\Delta t = 10^{-3}$, and $t_{\mathrm{end}} = 0.5$), with the lowest-order case
$k=0$. At the final time step we reconstruct the displacement and spin by \eqref{eq:FF-RT-displacement-recovery} and \eqref{eq:FF-RT-integrated-spin}, respectively, 
accumulated at every time step by the backward-Euler update
$\bu_h^{n+1} = \bu_h^n + \Delta t\, \bv_h^{n+1}$ and
$\brho_h^{n+1} = \brho_h^n + \Delta t\, \bomega_h^{n+1}$, and then solve  the local problem \eqref{eq:FF-RT-displacement-postprocessing} for $\bu_h^{\star}\in \bP_{k+2}(K)$. This reconstruction space   and the constraint space $\bP_{k+1}(K)$ (in FF) are
  discontinuous, this represents a single global linear solve with a
block-diagonal (element-local) matrix, at negligible extra cost. Table~\ref{tab:postprocess-FF}
reports, for the Farhloul--Fortin pair, the $\bL^2$-errors and experimental rates
of the discrete velocity $\bv_h$, of the reconstructed (pre-processing)
displacement $\bu_h$, and of the post-processed displacement $\bu_h^{\star}$. As
predicted by Theorem~\ref{thm:displacement-superconvergence}, both $\bv_h$ and $\bu_h$ converge at the optimal
rate $k+1=1$, while $\bu_h^{\star}$ converges asymptotically at the  
rate $k+2=2$ predicted by \eqref{eq:FF-RT-displacement-postprocessing-rate}.

\begin{table}[t]
\centering
\begin{tabular}{rc cc cc cc}
\toprule
DoFs & $h$ & $e(\bv_h)$ & rate & $e(\bu_h)$ & rate & $e(\bu_h^\star)$ & rate \\
\midrule
\multicolumn{8}{c}{AFW family with $k=0$} \\
\midrule
   200 & 0.707 & 1.76e-01 & --   & 9.06e-02 & --   & 7.77e-02 & --   \\
   720 & 0.354 & 6.31e-02 & 1.48 & 3.67e-02 & 1.30 & 2.09e-02 & 1.90 \\
  2720 & 0.177 & 2.82e-02 & 1.16 & 1.68e-02 & 1.13 & 5.85e-03 & 1.84 \\
 10560 & 0.088 & 1.34e-02 & 1.08 & 8.08e-03 & 1.05 & 1.57e-03 & 1.89 \\
 41600 & 0.044 & 6.58e-03 & 1.03 & 3.99e-03 & 1.02 & 4.07e-04 & 1.96 \\
165120 & 0.022 & 3.27e-03 & 1.01 & 1.99e-03 & 1.00 & 1.11e-04 & 1.97 \\
\midrule
\multicolumn{8}{c}{PEERS family with $k=0$} \\
\midrule
   153 & 0.707 & 7.04e-01 & --   & 2.66e-01 & --   & 3.05e-01 & --   \\
   561 & 0.354 & 1.93e-01 & 1.87 & 7.52e-02 & 1.82 & 7.99e-02 & 1.93 \\
  2145 & 0.177 & 5.15e-02 & 1.91 & 2.26e-02 & 1.73 & 1.86e-02 & 2.11 \\
  8385 & 0.088 & 1.74e-02 & 1.56 & 8.97e-03 & 1.33 & 4.80e-03 & 1.95 \\
 33153 & 0.044 & 7.16e-03 & 1.28 & 4.11e-03 & 1.13 & 1.21e-03 & 1.98 \\
131841 & 0.022 & 3.35e-03 & 1.10 & 2.00e-03 & 1.04 & 2.95e-04 & 2.04 \\
\bottomrule
\end{tabular}
\caption{Displacement post-processing for the AFW and PEERS families
($k=0$), shown for comparison with Table~\ref{tab:postprocess-FF}. The
superconvergent rate $k+2=2$ is not covered by Theorem~\ref{thm:displacement-superconvergence} for
these pairs, but is still observed numerically.}
\label{tab:postprocess-AFW-PEERS}
\end{table}

For comparison, Table~\ref{tab:postprocess-AFW-PEERS} applies the   post-processing to the AFW and PEERS pairs (but using $\bP_{k+1}(K)$ rather than $\bP_{k+2}(K)$), for which Theorem~\ref{thm:displacement-superconvergence} is
not established (see also Remark~\ref{rem:rates-other-families}, pointing that neither pair inherits the extra order of accuracy that motivates choosing  $\bP_{k+2}(K)$ for FF). In this particular case of smooth solutions, postprocessing into one degree above the usual velocity space, produces for both families 
a rate close to $k+2$.

\subsection{Simple viscoelasticity benchmarks}\label{sec:viscoelasticity-benchmarks}

In order to validate the viscoelasticity formulation and the associated solvers in a setting where the qualitative mechanical response is well documented, we discard the reaction-diffusion subsystem~\eqref{eq:calcium-dynamics-1}--\eqref{eq:calcium-dynamics-2} and set $\alpha\equiv 0$ in~\eqref{eq:strong-viscoelast}, so that the coupling term disappears from the constitutive law of $\bsigma_2$. What remains is the uncoupled viscoelasticity block
that is, a Kelvin--Voigt solid with an elastic spring $\bsigma_1=\cC_1\beps(\bu)$ and a viscous dashpot $\bsigma_2=\cC_2\beps(\bv)$ acting in parallel on the same total strain, with $\bsigma_1$ recovered in rate form (cf.~\eqref{eq:Ais}). We follow \cite{meddahi23,rognes10} (see also \cite{Lee-thesis}). We use the lowest-order ($k=0$) Farhloul--Fortin pair of Section~\ref{sec:FE-examples}, now in three dimensions, on a structured tetrahedral partition of the slab $\Omega=(0,1)\times(0,0.5)\times(0,0.5)$ ($1500$ elements, $64\,188$ degrees of freedom), and backward Euler in time. The body is clamped, $\bv=\cero$, on $\Gamma^{\bu}=\{x=0\}$, and a time-dependent traction $(\bsigma_1+\bsigma_2)\bn = \bg(t)$ is imposed on $\Gamma^{\bsigma}=\partial\Omega\setminus\Gamma^{\bu}$, with $\bg\equiv\cero$ on $\{y=0\}\cup\{y=0.5\}\cup\{z=0\}\cup\{z=0.5\}$ (stress-free) and a prescribed traction on $\{x=1\}$, specified below for each test; since the Farhloul--Fortin pair imposes traction through Dirichlet data on the stress spaces, $\bg$ is assigned to $\bsigma_1$'s boundary data there (with $\bsigma_2$'s left at zero), which reproduces $(\bsigma_1+\bsigma_2)\bn=\bg(t)$. The material parameters are $\mu_1=200$, $\lambda_1=1000$ (elastic branch) and $\mu_2=50$, $\lambda_2=200$ (viscous branch), with a small density $\varrho=10^{-3}$ so that the discrete elastodynamic scheme approximates a quasi-static response. These values do not coincide with those in \cite[Table 4]{rognes10} (quasi-static, no inertia, and Kelvin--Voigt) and \cite[Section 6]{meddahi23} (four-parameter Zener/standard-linear-solid model with an additional relaxation-time variable), since system~\eqref{eq:strong-viscoelast} with $\alpha\equiv0$ has an inertial term and two material branches. 

\begin{figure}[t!]
\centering
\includegraphics[width=0.95\linewidth]{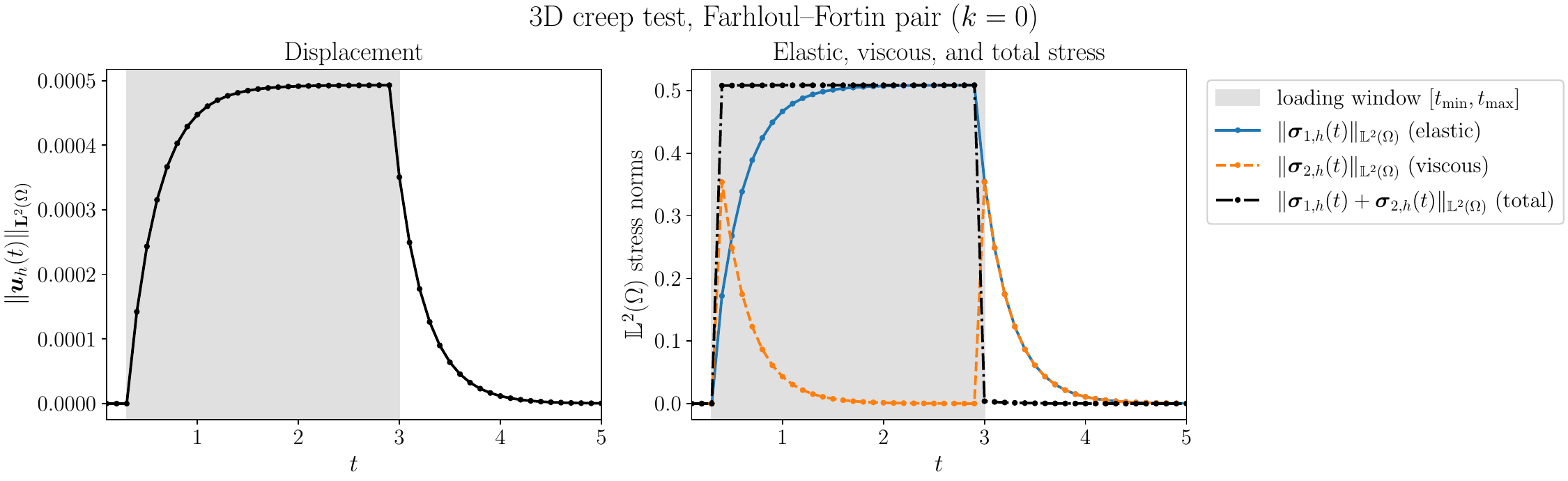}
\caption{Creep test for the uncoupled viscoelasticity block, Farhloul--Fortin pair with $k=0$, in the three-dimensional slab $\Omega=(0,1)\times(0,0.5)\times(0,0.5)$. Left: $\bL^2(\Omega)$-norm of the reconstructed displacement $\bu_h(t)$. Right: $\bbL^2(\Omega)$-norms of the elastic ($\bsigma_{1,h}$), viscous ($\bsigma_{2,h}$), and total ($\bsigma_{1,h}+\bsigma_{2,h}$) stresses. The shaded band marks the loading window $[t_{\min},t_{\max}]$.}
\label{fig:creepFF}
\end{figure}

For the creep test, we apply a traction $\bg(t) = (1,0,0)$ for $t\in[t_{\min},t_{\max}]=[0.3,3.0]$, and $\bg\equiv\cero$ otherwise, with $\Delta t=0.1$, running the simulation until $t_{\mathrm{end}}=5$. We track the $\bL^2(\Omega)$/$\bbL^2(\Omega)$ norms of the reconstructed displacement $\bu_h(t)$ (obtained exactly as in~\eqref{eq:reconstructed-displacement}, by backward-Euler time-integration of $\bv_h$ from the undeformed state) and of the elastic, viscous, and total stresses. Figure~\ref{fig:creepFF} shows the expected Kelvin--Voigt creep response: while the traction is active, the displacement rises and saturates at its elastic limit, the viscous stress decays from its initial peak as the elastic stress increases to sustain the (approximately constant) total stress, and once the load is released at $t=t_{\max}$ all quantities relax back to zero. This reproduces, qualitatively, the Kelvin--Voigt response of \cite[Fig.~5]{rognes10} and the creep transients of \cite[Fig.~6.3, bottom-left]{meddahi23}, and we note that the recovery after unloading in Figure~\ref{fig:creepFF} is governed by a single time constant. Sample of solutions of the 3D deformed domain at three time instants are shown in Figure \ref{fig:creep-3D-deformed}.

\begin{figure}[t!]
\centering
\begin{overpic}[width=0.325\linewidth]{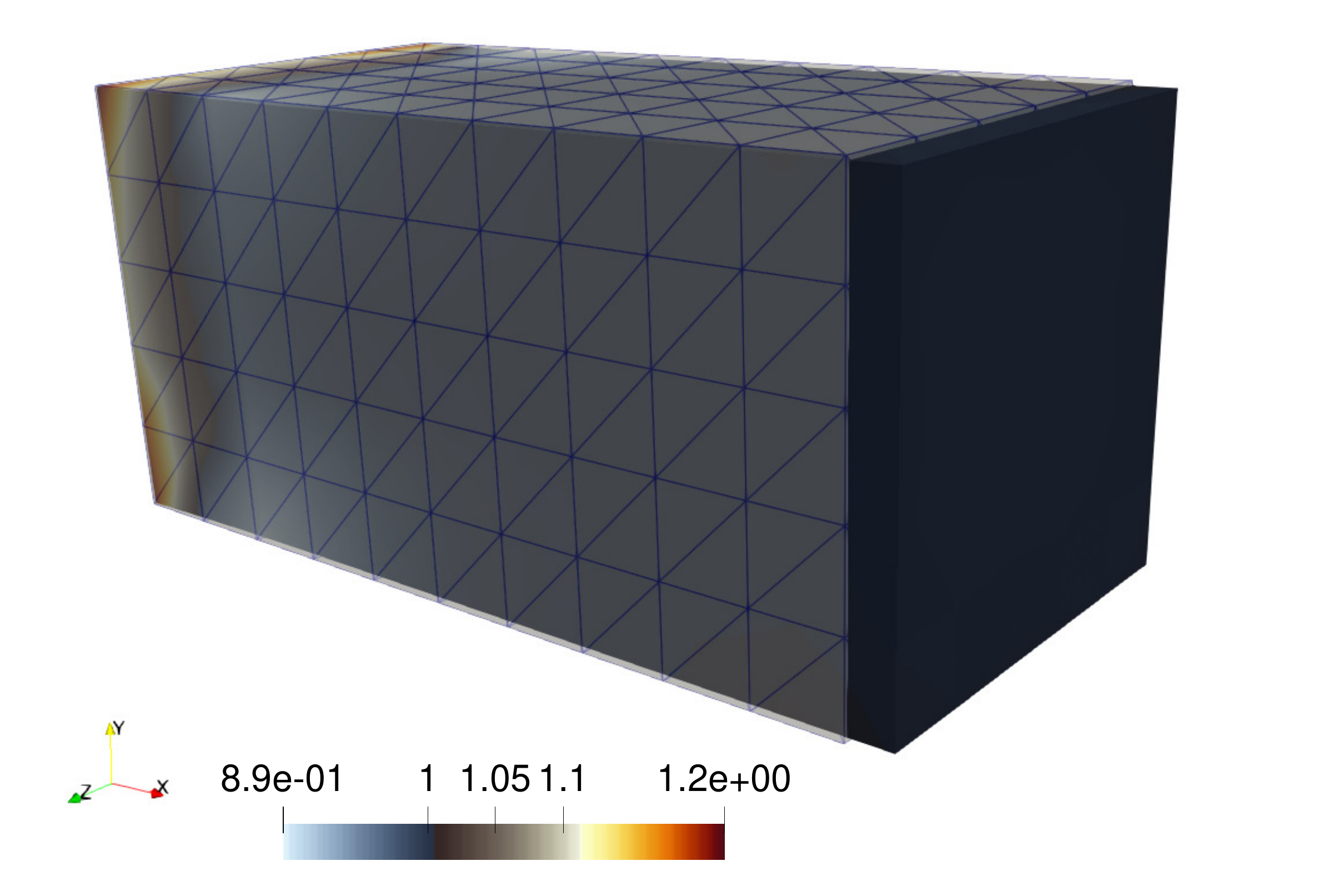}
\put(24,14){\small$\|\sigma_h\|$}\end{overpic}
\begin{overpic}[width=0.325\linewidth]{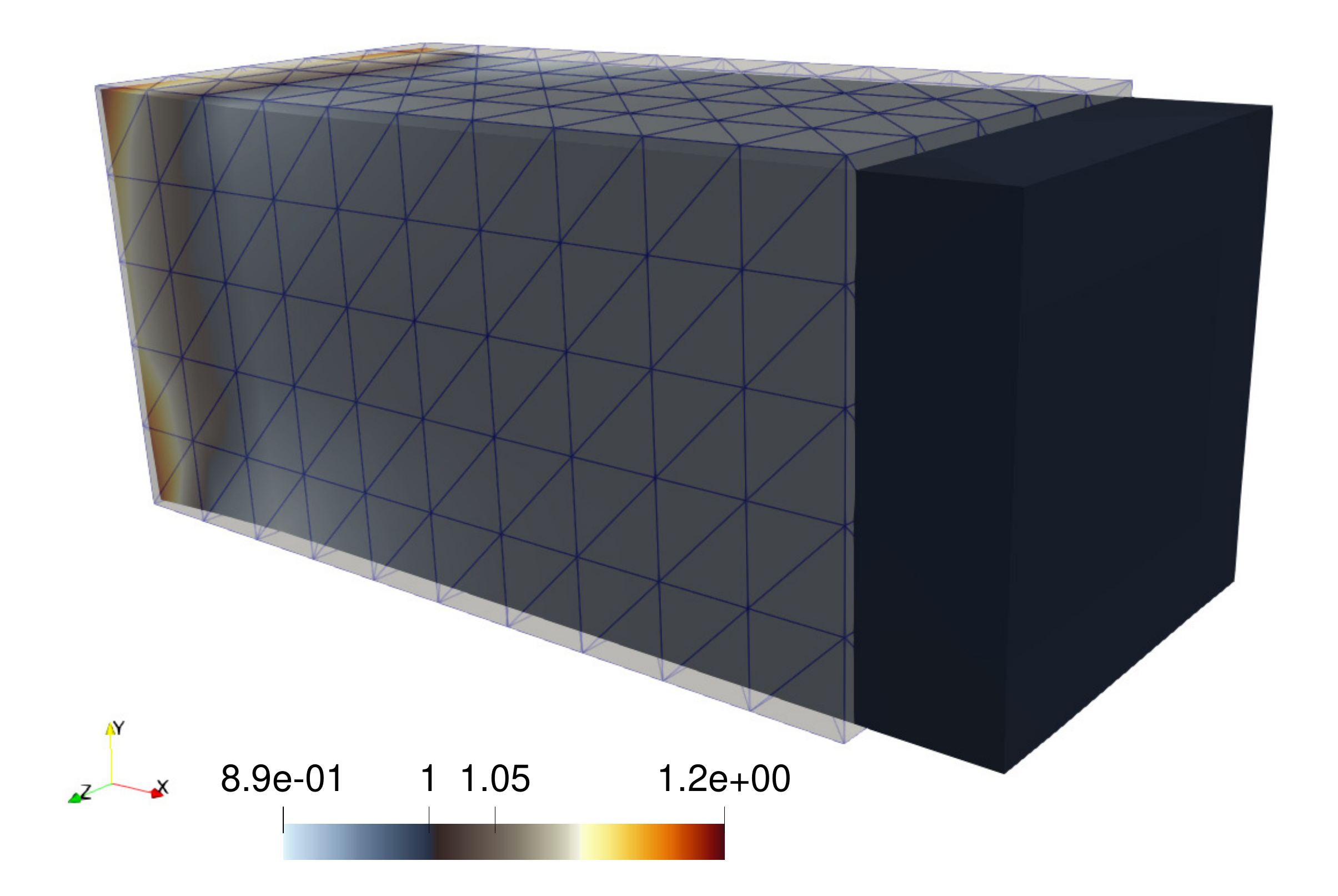}
\put(24,14){\small$\|\sigma_h\|$}\end{overpic}
\begin{overpic}[width=0.325\linewidth]{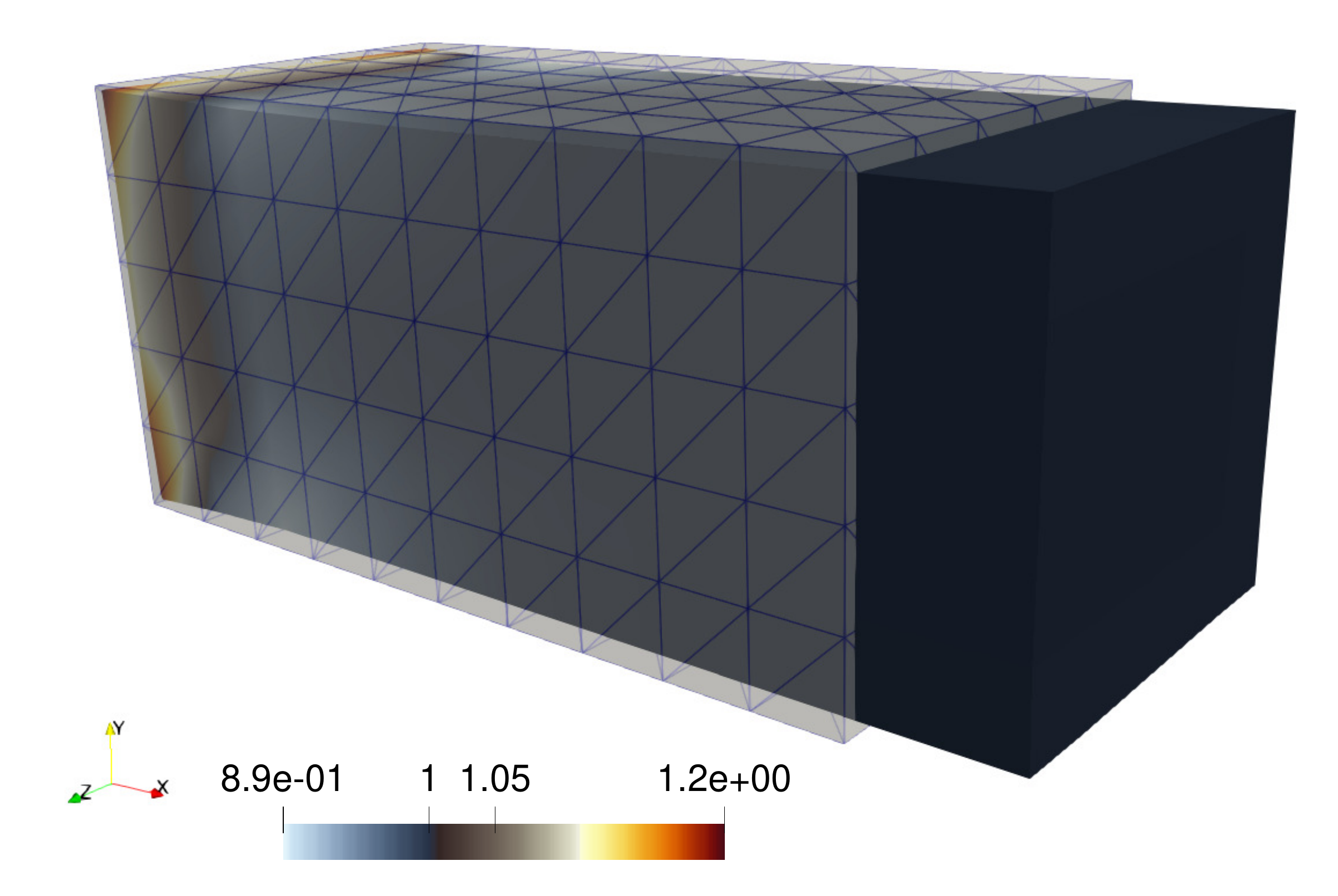}
\put(24,14){\small$\|\sigma_h\|$}\end{overpic}
\caption{Deformed configuration (100x magnified displacement) of the viscoelastic slab during the creep test of Figure~\ref{fig:creepFF}, on a coarse mesh and at three time instants during loading and recovery ($t=0.4$, $t=0.9$, $t=2.9$) coloured by the magnitude of the total stress $\bsigma_1+\bsigma_2$ (hollow undeformed domain is also shown for reference).}
\label{fig:creep-3D-deformed}
\end{figure}

Next, we replace the traction pulse by the oscillatory loading $\bg(t) = (0.5\sin(\pi t), 0, 0)$ for $t\le t_{\max}=3$, released thereafter, with $\Delta t=0.05$, as in \cite[Example~3]{meddahi23}, again running until $t_{\mathrm{end}}=5$. At every time step we evaluate, at the point $\mathbf{p}=(0.5,0.25,0.25)$ of the domain (the same point used in~\cite{meddahi23}), the $x$-component of $\bu_h$, the $xx$-components of $\bsigma_{1,h}$, $\bsigma_{2,h}$, and their sum, and the $xx$-component of the strain recovered from the elastic branch's inverse constitutive law, $\beps_h(\mathbf{p};t) = \cA_1\bsigma_{1,h}(\mathbf{p};t)$, which by the first relation in~\eqref{eq:Ais} coincides exactly with $\beps(\bu_h)(\mathbf{p};t)$. Figure~\ref{fig:cyclicFF} reports these transients together with the resulting $\sigma_{xx}$-versus-$\varepsilon_{xx}$ stress--strain curve. The displacement and stresses go into a periodic response with a phase lag between the elastic and viscous branches, and the stress--strain shows an open hysteresis loop (its enclosed area quantifies the energy dissipated per cycle by the dashpot). After unloading at $t=t_{\max}$ the trajectory spirals into the origin as the free oscillation decays. This is in qualitative agreement with the behaviour reported in \cite[Fig.~6.3, bottom-right]{meddahi23}.

\begin{figure}[t!]
\centering
\includegraphics[width=\linewidth]{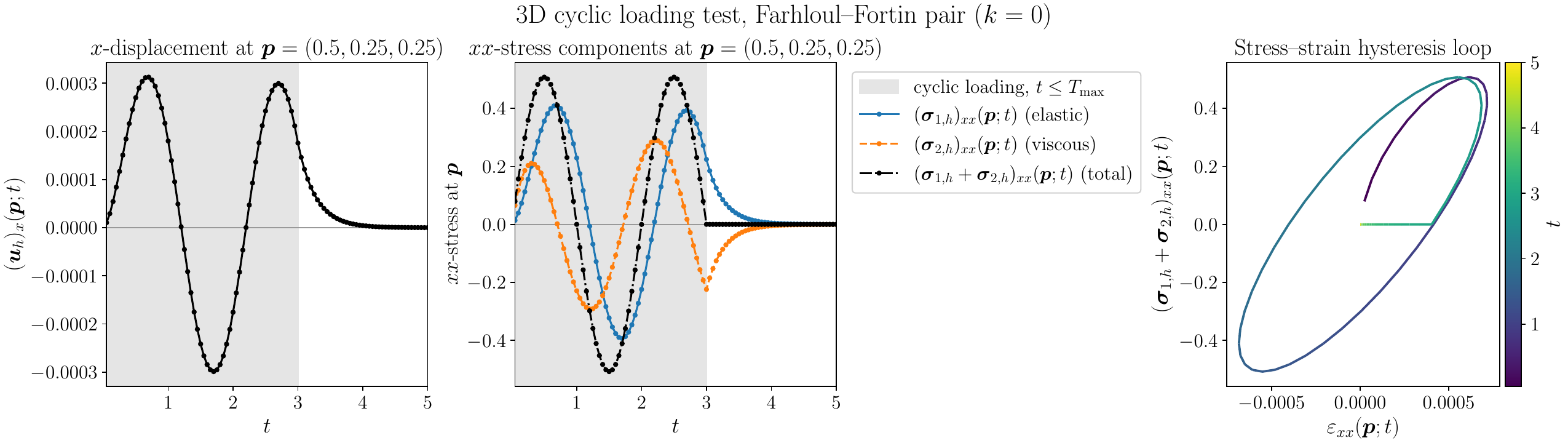}

\vspace{-0.3cm}
\caption{Cyclic loading test for the uncoupled viscoelasticity block, Farhloul--Fortin pair with $k=0$, in the three-dimensional slab $\Omega=(0,1)\times(0,0.5)\times(0,0.5)$, evaluated at the point $\mathbf{p}=(0.5,0.25,0.25)$. Left: $x$-displacement. Centre: $xx$-components of the elastic, viscous, and total stress. Right: $\sigma_{xx}$-versus-$\varepsilon_{xx}$ hysteresis loop, colour-coded by time; the strain is recovered from $\bsigma_{1,h}$ via the elastic inverse constitutive law.}
\label{fig:cyclicFF}
\end{figure}

\subsection{Mechanically driven pattern formation and calcium waves}
We close with a test that illustrates how the calcium dynamics change as the mechanochemical couplings are switched on one at a time, starting from the purely reaction--diffusion model of Atri et al.~\cite{atri93} and ending with the fully coupled model of~\cite{kaouri19,kaouri22}. We consider the following hierarchy, where all other parameters and the initial stimulus are kept fixed:
(A) \emph{Atri model:} $\alpha_0=\beta=0$. Then $\bv\equiv\cero$ and the calcium system decouples;
(B) \emph{one-way coupling:} $\alpha_0>0$, $\beta=0$. Calcium generates an active stress and deforms the tissue, but the deformation does not act back on the calcium;
(C) \emph{two-way coupling:} $\alpha_0>0$, $\beta>0$. The  source $\beta\,\vdiv\bu$ closes the mechanochemical feedback loop;
(D) as (C), with the viscosity of the dashpot (that is, $\mu_2,\lambda_2$) scaled by a factor between $10^{-2}$ and $1$.

The domain is the square $\Omega=(-a,a)^2$ with $a=1.43\sqrt{\pi}/2\approx 1.267$, which has the same area as the disc of radius $1.43$ used in~\cite{kaouri22}. The tissue is clamped at the bottom edge, $\Gamma^{\bu}=\{y=-a\}$, and traction-free on the remaining three sides, $\Gamma^{\bsigma}=\partial\Omega\setminus\overline{\Gamma^{\bu}}$. For the calcium we take zero flux on the whole boundary, $\bz\cdot\bn=0$ on $\partial\Omega$, instead of~\eqref{eq:bc-phi}. We use the adimensionalisation of~\cite{kaouri22}, with time scale $\tau_j=2\,$s and length scale $L=100\,\mu$m. The reaction terms are~\eqref{eq:specific-reaction-functions} with
\[
\gamma_0 = 13.886,\quad \gamma_1 = 0.111,\quad \gamma_2 = 5.714,\quad \gamma_3 = 0.143,\qquad \diff = 0.004\,\bbI.
\]
For these values the homogeneous steady state of~\eqref{eq:calcium-dynamics-1}--\eqref{eq:calcium-dynamics-2} with $\beta=0$, namely $\varphi^*=0.5563$ and $\theta^*=1/(1+(\varphi^*)^2)$, is an unstable focus. In this regime a single localised stimulus produces a train of calcium waves, rather than the solitary pulse obtained in the excitable regime. The Young modulus is $E=44.26\,$Pa, except in a layer of width $0.017$ along $\partial\Omega$ where it is $E=24.34\,$Pa (with the same area fraction as the softer rim of the disc in~\cite{kaouri22}). We use $\nu=0.4$ and viscosities $\eta_1=3790\,$Pa\,s and $\eta_2=550\,$Pa\,s. This gives, in the stiffer core, the adimensional Lam\'e parameters
\[
\mu_1 = 0.5,\quad \lambda_1 = 0.0633,\quad \mu_2 = 29.97,\quad \lambda_2 = 0.2751,
\]
and $\varrho=10^{-3}$, so that the response is essentially quasi-static. The active stress is~\eqref{eq:def-Hill-function} with $\alpha_0=T_0(1+\nu)/E=0.158$ (corresponding to $T_0=5\,$Pa) in cases (B)--(D).
Instead of starting from an arbitrary state, we compute the rest state of the coupled system: $\bv=\cero$, $\cA_2\bsigma_2=-\alpha(\varphi)/(2\mu_2+d\lambda_2)\,\bbI$, and $(\bsigma_1,\bu,\varphi)$ solve the stationary mixed elasticity problem with $(\bsigma_1+\bsigma_2)\bn=\cero$ on $\Gamma^{\bsigma}$, coupled to the stationary calcium equation. This nonlinear problem is solved with Newton's method, using continuation in $\beta$. For $\beta>0$ the resulting rest state is non-uniform, because the clamp makes $\vdiv\bu$ vary in space. For instance, for $\beta=2$ we have $\varphi\in[0.574,0.669]$ and $\vdiv\,\bu\in[-0.158,0.156]$, with the extreme values attained near the two corners where $\Gamma^{\bu}$ meets $\Gamma^{\bsigma}$. The initial condition is this rest state plus the calcium pulse $\exp(-|\bx|^2/0.01)$ centred at the origin.

\begin{figure}[t!]
\centering
\includegraphics[width=\linewidth]{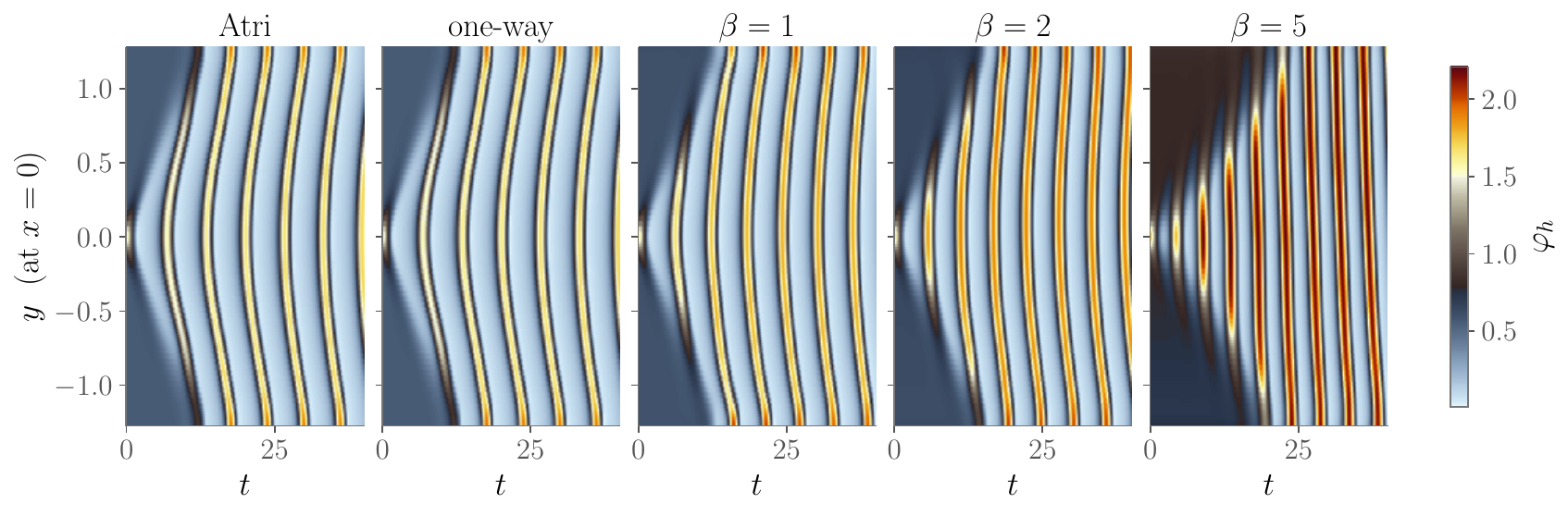}

\vspace{-0.3cm}
\caption{Kymographs of $\varphi_h$ along $x=0$ (the clamped edge is at the bottom) for the hierarchy (A)--(C): Atri model ($\alpha_0=\beta=0$), one-way coupling ($\beta=0$), and two-way coupling with $\beta\in\{1,2,5\}$. Mixed formulation, Farhloul--Fortin family with $k=0$.}
\label{fig:calcium-kymo}
\end{figure}

We use the lowest-order Farhloul--Fortin family ($k=0$) for the viscoelasticity  and $(\bZ_h^{\rm BDM},\Phi_h,\Theta_h)$ with $\mathrm{BDM}_1$ flux and piecewise constant $\varphi_h,\theta_h$ for the calcium. The mesh has $134\times134\times2 = 35\,912$ triangles, giving $523\,137$ (viscoelasticity) plus $180\,096$ (calcium) degrees of freedom. The time step is $\Delta t=0.1$ and $t_{\rm end}=40$. At each step we first solve the (linear) viscoelasticity block with the active stress evaluated at $\varphi_h^n$, whose matrix is constant and is therefore factorised only once. We then solve the calcium block with $\tr\bsigma_{1,h}^{n+1}$, implicit diffusion and explicit reaction terms. Since the active stress enters $\bsigma_2$, the individual tractions $\bsigma_1\bn$ and $\bsigma_2\bn$ do not vanish on $\Gamma^{\bsigma}$ (at rest $\bsigma_2\bn=-\alpha(\varphi)\bn$); only their sum does. Therefore, instead of imposing $\bsigma_{1,h}\bn=\bsigma_{2,h}\bn=\cero$ essentially, we impose the total traction weakly, by adding to the first two equations of the discrete version of~\eqref{eq:mixed-weak} the penalty terms
\[
\frac{\gamma}{h_F}\,\big((\bsigma_{1,h}+\bsigma_{2,h})\bn,\btau_i\bn\big)_{\Gamma^{\bsigma}}\, ,\qquad i\in\{1,2\}\, ,
\]
with $h_F$ the facet size and $\gamma=10$. The discrete rest state described above is an exact fixed point of the resulting scheme: without stimulus, $\varphi_h$ stays at rest up to $10^{-12}$. With $\gamma\in\{1,10,100\}$ the weakly imposed traction satisfies $\max_t\|(\bsigma_{1,h}+\bsigma_{2,h})\bn\|_{0,\Gamma^{\bsigma}}\approx 10^{-4},\,10^{-5},\,10^{-6}$, respectively, and the calcium changes by less than $5\times10^{-4}$. As a reference we also solve the same model in primal form used in \cite{kaouri22}, with displacement and pressure as unknowns ($\bP_1$ plus bubble and $\mathrm{P}_1$), a $\mathrm{P}_1$ calcium approximation, the same mesh, and the same splitting.

\begin{figure}[t!]
\centering
\includegraphics[width=0.8\linewidth]{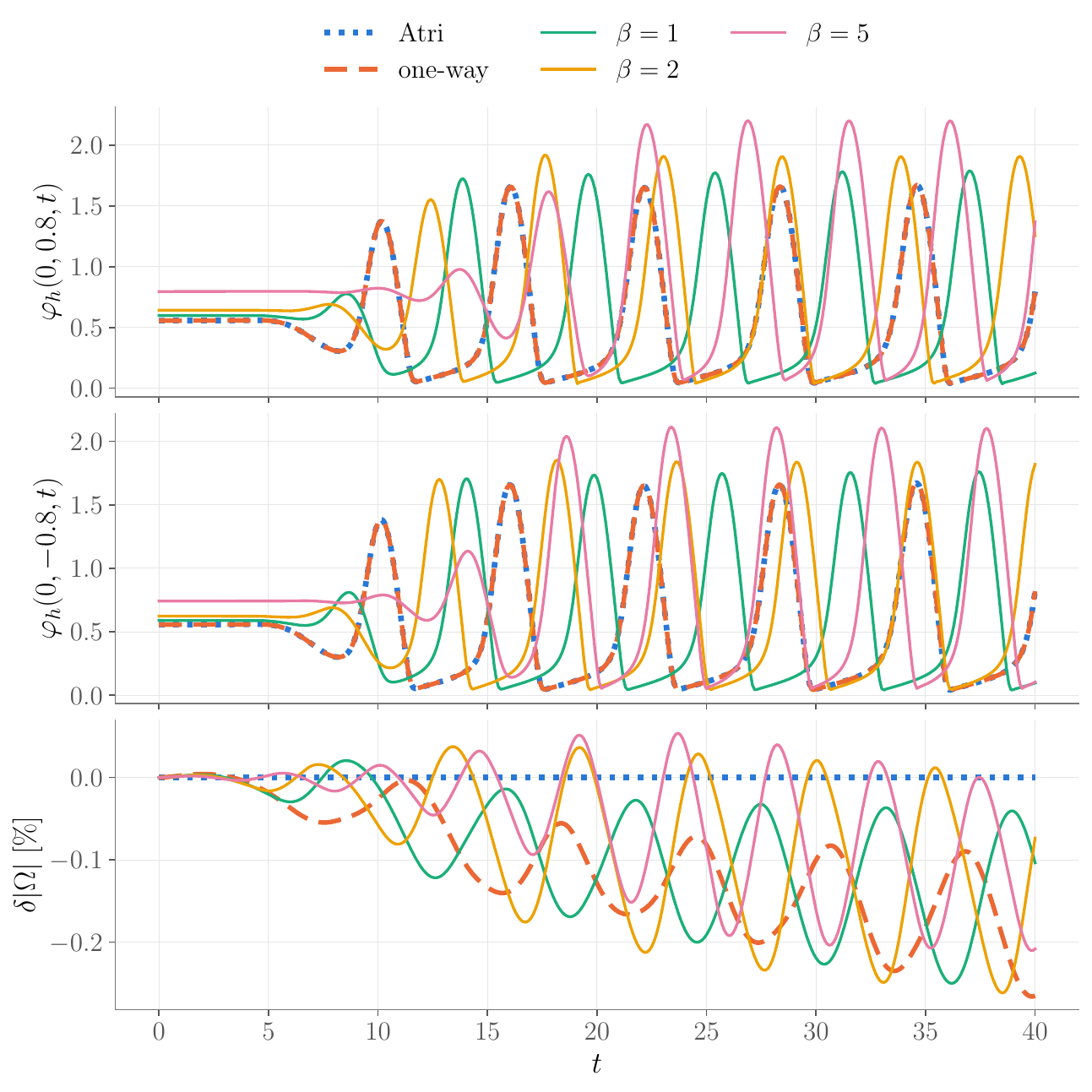}
\caption{Calcium at the points $(0,0.8)$ (towards the free edge, top) and $(0,-0.8)$ (towards the clamped edge, middle), and relative area change $\delta|\Omega|$ with respect to the rest state (bottom), for the hierarchy (A)--(C).}
\label{fig:calcium-series}
\end{figure}

Figure~\ref{fig:calcium-kymo} shows kymographs (space-time plots) of $\varphi_h$ along the vertical line $x=0$, and Figure~\ref{fig:calcium-series} shows time series of $\varphi_h$ at the points $(0,\pm0.8)$ together with the relative area change
\[
\delta|\Omega|(t) := \frac{100}{|\Omega|}\int_\Omega \vdiv(\bu_h-\bu_{\rm rest}) = \frac{100}{|\Omega|}\int_\Omega \tr\,\cA_1(\bsigma_{1,h}-\bsigma_{1,{\rm rest}})\quad [\%]\, .
\]
In case (A) the stimulus generates target waves that propagate from the centre at speed $\approx0.18$ with period $6.23$ and reach the boundary at $t\approx 10$. Behind them the (unstable) rest state develops bulk oscillations of the same period. In case (B) the calcium is unchanged, since the curves for (A) and (B) overlap, while the tissue now dilates and contracts at the period of the wavetrain. The area change oscillates with a peak-to-peak amplitude of about $0.1\%$ on top of a slow drift, whose time scale is the Kelvin--Voigt retardation time ($\approx 54$ time units). Once the feedback is active (case C), the dilation-dependent source raises the rest level of calcium (at the point $(0,0.6)$ it is $0.598$, $0.642$ and $0.793$ for $\beta=1,2,5$), and the wavetrain becomes faster and stronger. The period decreases from $6.23$ to $5.78$, $5.42$ and $4.60$, and the peak concentration increases from $1.91$ to $2.03$, $2.05$ and $2.21$, respectively. For $\beta=5$ the first few fronts are also delayed and weakened before the train is established. For larger coupling (e.g.\ $\beta=10$) the positive feedback prevents the existence of a proper rest state.

 \begin{figure}[t!]
\centering
\includegraphics[width=\linewidth]{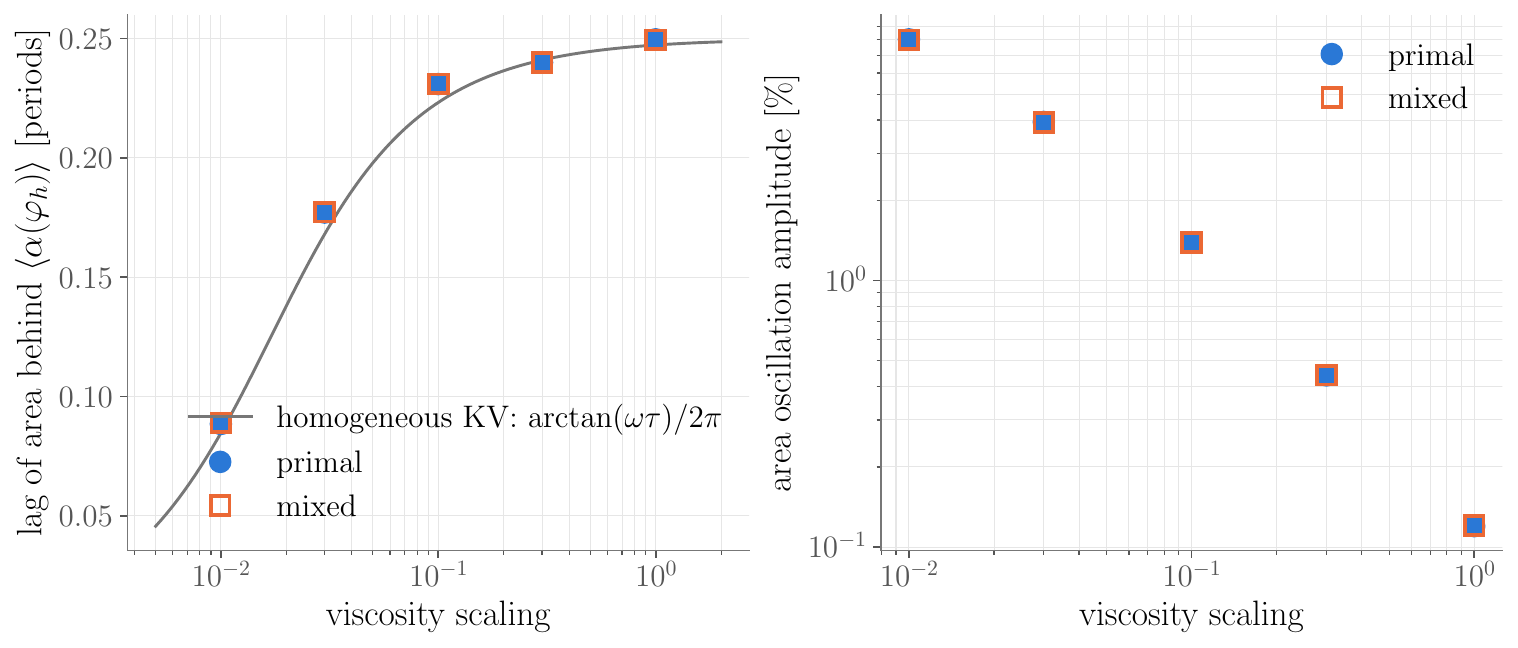}

\vspace{-0.3cm}
\caption{Viscosity sweep for $\beta=2$ (case D). Left: lag of the area change behind the mean active stress, in periods of the wavetrain, together with the homogeneous Kelvin--Voigt estimate. Right: amplitude of the oscillation of $\delta|\Omega|$.}
\label{fig:calcium-lag}
\end{figure}
 
 Finally, for $\beta=2$ we scale $\mu_2,\lambda_2$ by a factor $s\in\{0.01,0.03,0.1,0.3,1\}$ (case D). For a homogeneous Kelvin--Voigt body subject to an isotropic active stress oscillating at frequency $\omega$, the dilation lags behind the active stress by $\arctan(\omega\tau)/\omega$, with retardation time $\tau=(\lambda_2+\mu_2)/(\lambda_1+\mu_1)$ (in 2D). We measure this lag as the phase difference, at the frequency of the wavetrain, between the spatial mean of $\alpha(\varphi_h)$ and $\delta|\Omega|$ over $t\in[20,40]$. Figure~\ref{fig:calcium-lag} shows the lag and the amplitude of the area oscillation as functions of $s$. The lag increases from $0.09$ to   the quarter period of the purely viscous limit. It follows the homogeneous estimate up to $5\%$ for $s\geq0.03$ (and $15\%$ for $s=0.01$), even though the clamped boundary makes the deformation inhomogeneous. The amplitude of the area oscillation decreases from $8\%$ to $0.12\%$. The viscosity therefore acts as a low-pass filter on the mechanochemical feedback. The phase of the feedback also changes the wavetrain (its period grows from $5.42$ for $s=1$ to $6.45$ for $s=0.01$, which is longer than the period of the Atri model). The two formulations  give lags that agree to   $0.001$ periods, and amplitudes that agree to   $0.6\%$.

Figure~\ref{fig:calcium-deformed} shows the deformed configuration for $\beta=2$ at four instants of an established cycle of the wavetrain, with the displacement relative to the rest state, $\bu_h-\bu_{\rm rest}$, magnified 200 times. At $t=28.6$ most of the tissue is near the peak of the calcium oscillation (the spatial mean of $\varphi_h$ is $1.39$) and the tissue is dilated, with the largest displacements at the free top corners. At $t=29.6$ and $t=30.2$ the calcium decreases in the interior, while the last part of the excitation reaches the corners. The total stress $\bsigma_{1,h}+\bsigma_{2,h}$ concentrates at the two corners where the clamped edge meets the traction-free edges and reaches its maximum there. At $t=32.8$ a new front leaves the centre, the stress has relaxed to low values, and the tissue deforms again as the viscous response follows the calcium. We do not observe radially symmetric patterns  since our boundary conditions differ from the pure stress-free configuration used in \cite{kaouri22}.

\begin{figure}[t!]
\centering
\includegraphics[width=0.245\linewidth]{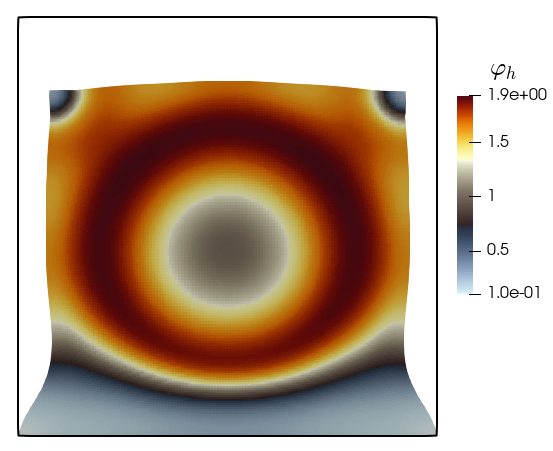}
\includegraphics[width=0.245\linewidth]{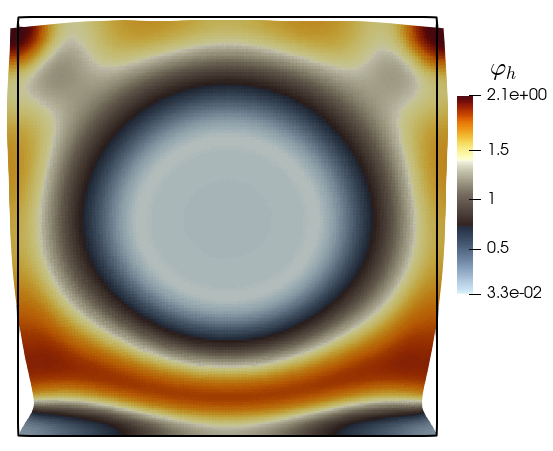}
\includegraphics[width=0.245\linewidth]{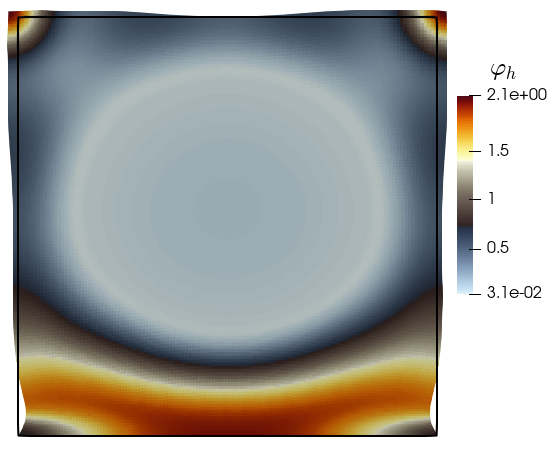}
\includegraphics[width=0.245\linewidth]{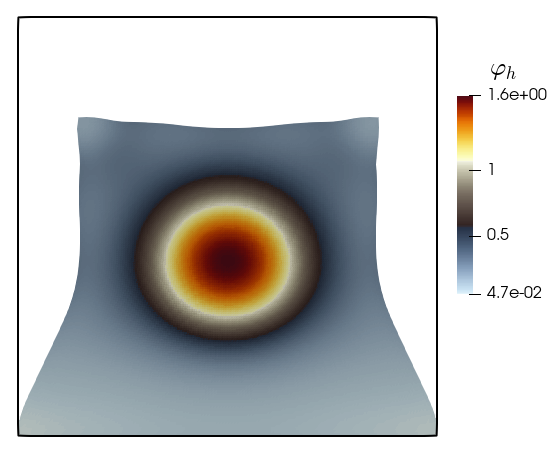}\\
\includegraphics[width=0.245\linewidth]{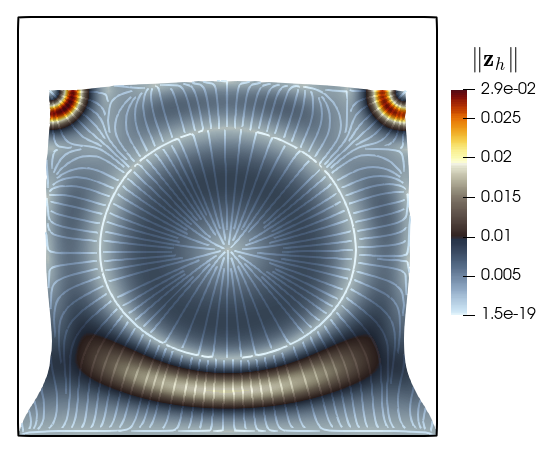}
\includegraphics[width=0.245\linewidth]{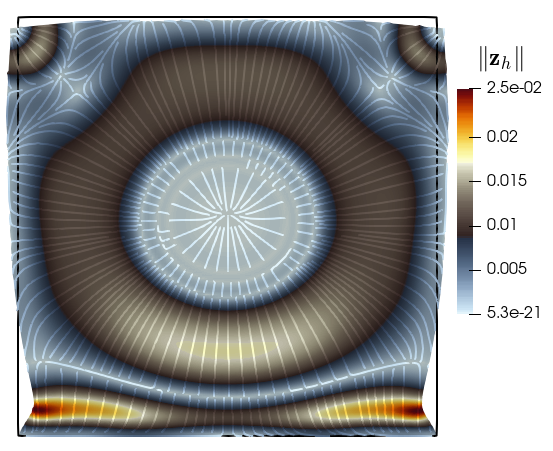}
\includegraphics[width=0.245\linewidth]{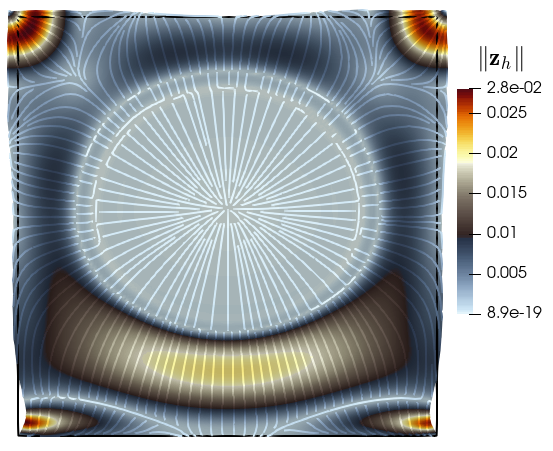}
\includegraphics[width=0.245\linewidth]{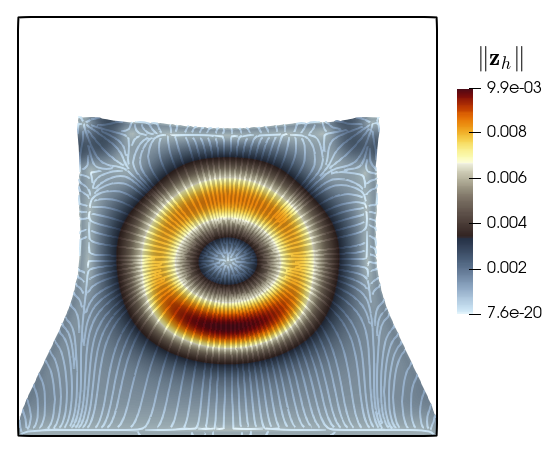}\\
\includegraphics[width=0.245\linewidth]{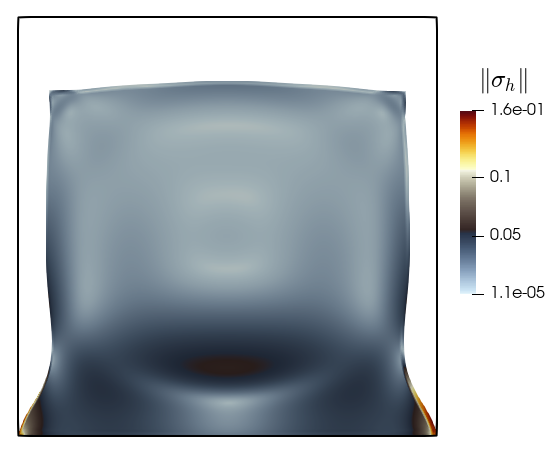}
\includegraphics[width=0.245\linewidth]{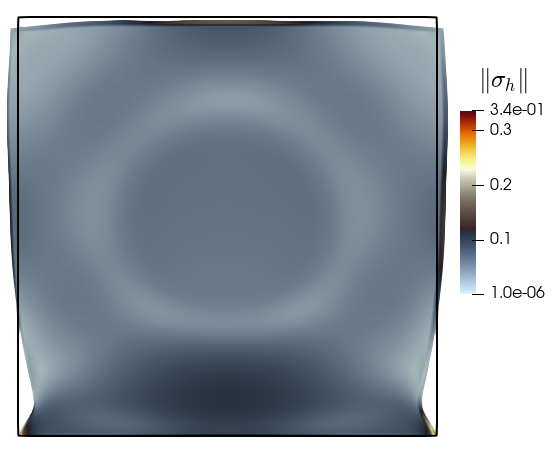}
\includegraphics[width=0.245\linewidth]{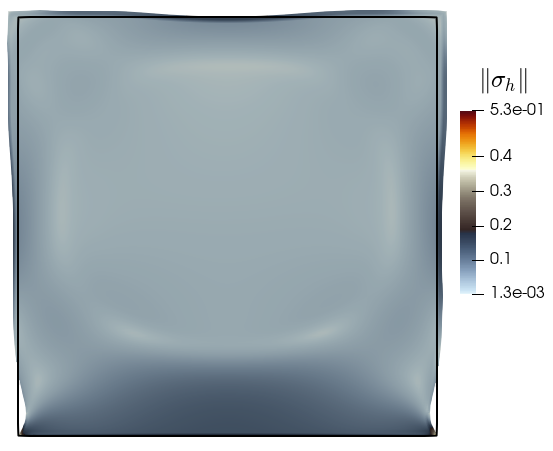}
\includegraphics[width=0.245\linewidth]{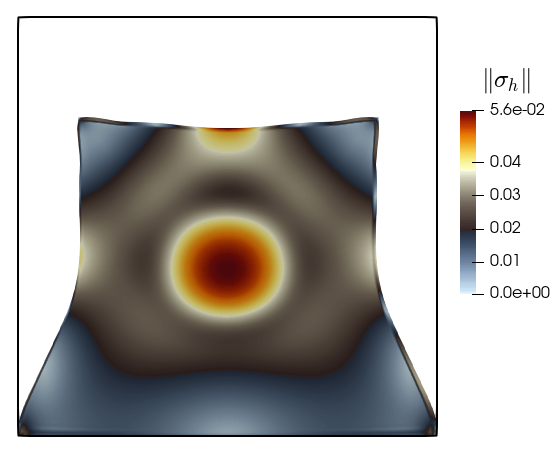}

\vspace{-0.2cm}
\caption{Sample numerical solutions for Case (C) with $\beta=2$. Calcium concentration $\varphi_h$ (top), magnitude and streamlines of the diffusive flux $\bz_h$ (middle row), and magnitude of the total stress $|\bsigma_{1,h}+\bsigma_{2,h}|$ (bottom) at $t=28.6,\,29.6,\,30.2,\,32.8$, plotted on the configuration deformed by $\bu_h-\bu_{\rm rest}$ (magnified $200\times$). Mixed formulation, Farhloul--Fortin family with $k=0$.}
\label{fig:calcium-deformed}
\end{figure}


\bibliographystyle{siam}
\small
\bibliography{blr_bib}

\end{document}